\documentclass[12pt]{article}
\usepackage{amsmath}
\usepackage{amssymb}
\usepackage{amsthm}
\usepackage{amsfonts}
\usepackage{graphicx}
\usepackage{algorithm}
\usepackage{algpseudocode}
\usepackage[usenames,dvipsnames]{xcolor}
\usepackage[listings, skins]{tcolorbox}
\usepackage{tabularx,siunitx}
\usepackage{subcaption}
\usepackage{authblk}
\usetikzlibrary{positioning,arrows.meta}
\usepackage{multirow}
\usepackage{makecell}

\colorlet{Accent}{Black}

\usepackage{booktabs} 
\usepackage[margin=1in]{geometry} 
\usepackage{enumitem} 

\usepackage{mathtools} 
\usepackage{bm} 
\usepackage{hyperref} 

\hypersetup{
    colorlinks=blue,
    linkcolor=blue,
    filecolor=magenta,      
    urlcolor=cyan,
    pdftitle={Operational Levers},
    pdfpagemode=FullScreen,
}

\usepackage{enumitem}            

\setlist[enumerate]{%
  font=\sffamily\bfseries\color{Accent},   
  label=\arabic*.                        
}

\makeatletter
\let\old@maketitle\@maketitle

\renewcommand{\@maketitle}{%
  {\sffamily\color{Accent}%
    \newpage
    \null\vskip 2em
    \begin{center}%
      {\LARGE \bfseries \@title \par}%
      \vskip 1.5em%
      {\large
        \lineskip .5em%
        \begin{tabular}[t]{c}%
          \@author
        \end{tabular}\par}%
      \vskip 1em%
      {\large \@date \par}%
    \end{center}%
    \par\vskip 1.5em}%
}
\makeatother

\usepackage{titlesec}

\titleformat{\section}
  {\normalfont\normalsize\sffamily\bfseries}{\thesection}{1em}{\MakeUppercase}
\titleformat{\subsection}
  {\normalfont\normalsize\sffamily\bfseries}{\thesubsection}{1em}{}
\titleformat{\subsubsection}
  {\normalfont\normalsize\sffamily\color{Accent}\itshape}{\thesubsubsection}{1em}{}

\makeatletter
\renewcommand{\tagform@}[1]{%
  \maketag@@@{\color{Accent}\sffamily(#1)}%
}
\makeatother

\renewenvironment{abstract}{%
	\sffamily
	\begin{center}
		\bfseries\abstractname
	\end{center}%
}{%
}

\usepackage{fancyhdr}
\renewcommand{\headrulewidth}{0.5pt} 
\renewcommand{\headrule}{\hbox to\headwidth{\leaders\hrule height \headrulewidth\hfill}} 

\renewcommand{\headrule}{%
  \vspace{-2.0ex} 
  \hbox to\headwidth{\leaders\hrule height \headrulewidth\hfill}%
  \vspace{-0.5ex} 
}
\tikzset{every picture/.style={line width=0.75pt}} 

\title{Regulation Zero 2: A Flow-Centric Sequential Regulation Planning Framework to Counter Regulation Cascading in Pre-tactical Air Traffic Flow Management}
\author[1,3]{Thinh Hoang\footnote{Work done at ENAC.}}
\author[2]{Zhengyi Wang}
\author[2]{Leïla Zerrouki}
\author[1]{Daniel Delahaye}

\affil[1]{École Nationale de l'Aviation Civile, Toulouse, France}
\affil[2]{EUROCONTROL Innovation Hub, Brétigny-sur-Orge, France}
\affil[3]{Department of Automotive Engineering, Van Lang University, Vietnam}
\date{\today}

\theoremstyle{plain}
\newtheorem{proposition}{Proposition}

\theoremstyle{definition}

\newcommand{\kl}[2]{D_{\mathrm{KL}}(#1 \parallel #2)}

\newcommand{\expect}[2]{\mathbb{E}_{#1}\left[#2\right]}

\newcommand{\CC}{\mathcal{C}}

\newcommand{\VV}{\mathcal{V}}

\newcommand{\HV}{\operatorname{HV}}
\newcommand{\II}{\mathcal{I}}
\newcommand{\TC}{\operatorname{TC}}
\newcommand{\HC}{\operatorname{HC}}

\begin{document}

\maketitle

\begin{abstract}
  Air Traffic Flow Management (ATFM) traffic regulations are being increasingly used as rising demand meets persistent workforce shortages. This operational strain has amplified a critical phenomenon that we call \emph{regulation cascading}: the compounding, non-linear interactions that occur when multiple regulations influence one another in unpredictable ways. As the number and complexity of regulations grow, cascading effects become more pronounced, undermining the network operator's ability to protect sectors reliably.

  To address this challenge, we introduce Regulation Zero 2, an updated sequential planning framework that natively operates in the regulation space, optimizing over ordered sequences of flow-level regulations that remain compatible as much as possible with existing slot-allocation systems such as CASA and RBS++. We equipped Regulation Zero 2 with new heuristics to render flow finding more efficient. At its core, the method employs a hierarchical Monte Carlo Tree Search (MCTS) that first samples congestion hotspots and then selects candidate regulations synthesized by a local proposal engine. Each proposal is evaluated by a fast First-Planned-First-Served (FPFS) allocator to estimate its reward, with these feedbacks guiding the subsequent MCTS exploration.

  Experiments on many pan-European summer-peak traffic days that Regulation Zero delivers promising and consistent performance. Compared to a flight-centric simulated-annealing and NSGA-II baselines, it achieves markedly higher objective improvements, while maintaining a tighter scope of impact on the network. Ablation studies also found that Regulation Cascading could reduce up to 50\% of potential effectiveness. RZ also preserves FPFS fairness and supports expert knowledge injection, offering a pragmatic and low-disruption pathway toward automation in operations.
\end{abstract}

\section{Introduction}
\subsection{Air Traffic Flow \& Capacity Management (ATFCM) Operations and the Need for Decision Support Within Existing Architectures}
\label{subsec:atfm_intro}
To ensure that controller workload remains within safe limits, Traffic Management Initiatives (TMIs) are implemented to prevent any controller from being overwhelmed by excessive aircraft counts at any given time. This concept, known as Demand-Capacity Balancing (DCB), is coordinated by EUROCONTROL's Network Manager in Europe and the Federal Aviation Administration's (FAA) Air Traffic Control System Command Center (ATCSCC) in the United States.

\begin{figure}
  \centering
  \begin{tikzpicture}[x=0.75pt,y=0.75pt,yscale=-1,xscale=1]

\draw   (130,250) -- (190,250) -- (190,292) -- (130,292) -- cycle ;
\draw   (130,130) -- (348,130) -- (348,172) -- (130,172) -- cycle ;
\draw   (210,250) -- (270,250) -- (270,292) -- (210,292) -- cycle ;
\draw   (290,250) -- (350,250) -- (350,292) -- (290,292) -- cycle ;
\draw    (160,250) -- (236.6,171.43) ;
\draw [shift={(238,170)}, rotate = 134.27] [color={rgb, 255:red, 0; green, 0; blue, 0 }  ][line width=0.75]    (10.93,-3.29) .. controls (6.95,-1.4) and (3.31,-0.3) .. (0,0) .. controls (3.31,0.3) and (6.95,1.4) .. (10.93,3.29)   ;
\draw    (240,250) -- (238.05,172) ;
\draw [shift={(238,170)}, rotate = 88.57] [color={rgb, 255:red, 0; green, 0; blue, 0 }  ][line width=0.75]    (10.93,-3.29) .. controls (6.95,-1.4) and (3.31,-0.3) .. (0,0) .. controls (3.31,0.3) and (6.95,1.4) .. (10.93,3.29)   ;
\draw    (328,250) -- (239.49,171.33) ;
\draw [shift={(238,170)}, rotate = 41.63] [color={rgb, 255:red, 0; green, 0; blue, 0 }  ][line width=0.75]    (10.93,-3.29) .. controls (6.95,-1.4) and (3.31,-0.3) .. (0,0) .. controls (3.31,0.3) and (6.95,1.4) .. (10.93,3.29)   ;
\draw   (131,49) -- (349,49) -- (349,91) -- (131,91) -- cycle ;
\draw   (60,20) -- (380,20) -- (380,180) -- (60,180) -- cycle ;
\draw   (70,240) -- (380,240) -- (380,310) -- (70,310) -- cycle ;
\draw   (390,138) -- (500,138) -- (500,180) -- (390,180) -- cycle ;
\draw   (510,138) -- (620,138) -- (620,180) -- (510,180) -- cycle ;
\draw    (450,50) -- (450,138) ;
\draw [shift={(450,140)}, rotate = 270] [color={rgb, 255:red, 0; green, 0; blue, 0 }  ][line width=0.75]    (10.93,-3.29) .. controls (6.95,-1.4) and (3.31,-0.3) .. (0,0) .. controls (3.31,0.3) and (6.95,1.4) .. (10.93,3.29)   ;
\draw    (570,50) -- (570,138) ;
\draw [shift={(570,140)}, rotate = 270] [color={rgb, 255:red, 0; green, 0; blue, 0 }  ][line width=0.75]    (10.93,-3.29) .. controls (6.95,-1.4) and (3.31,-0.3) .. (0,0) .. controls (3.31,0.3) and (6.95,1.4) .. (10.93,3.29)   ;
\draw    (380,50) -- (570,50) ;
\draw    (240,130) -- (240,92) ;
\draw [shift={(240,90)}, rotate = 90] [color={rgb, 255:red, 0; green, 0; blue, 0 }  ][line width=0.75]    (10.93,-3.29) .. controls (6.95,-1.4) and (3.31,-0.3) .. (0,0) .. controls (3.31,0.3) and (6.95,1.4) .. (10.93,3.29)   ;

\draw (135,262) node [anchor=north west][inner sep=0.75pt]   [align=left] {FMP 1};
\draw (169,142) node [anchor=north west][inner sep=0.75pt]   [align=left] {Network Manager};
\draw (215,262) node [anchor=north west][inner sep=0.75pt]   [align=left] {FMP 2};
\draw (295,262) node [anchor=north west][inner sep=0.75pt]   [align=left] {FMP 3};
\draw (96,201) node [anchor=north west][inner sep=0.75pt]   [align=left] {Regulations};
\draw (190,61) node [anchor=north west][inner sep=0.75pt]   [align=left] {CASA Algorithm};
\draw (71,160) node [anchor=north west][inner sep=0.75pt]  [rotate=-270] [align=left] {EUROCONTROL};
\draw (73,298) node [anchor=north west][inner sep=0.75pt]  [rotate=-270] [align=left] {ACCs\\Airports};
\draw (481,62) node [anchor=north west][inner sep=0.75pt]   [align=left] {Delays};
\draw (418,151) node [anchor=north west][inner sep=0.75pt]   [align=left] {Flight 1};
\draw (538,151) node [anchor=north west][inner sep=0.75pt]   [align=left] {Flight 2};
\draw (631,151) node [anchor=north west][inner sep=0.75pt]   [align=left] {...};
\draw (361,262) node [anchor=north west][inner sep=0.75pt]   [align=left] {...};

\end{tikzpicture}
  \caption{The general Flow Regulation Workflow in the European Context.}
  \label{fig:ectl_workflow}
\end{figure}
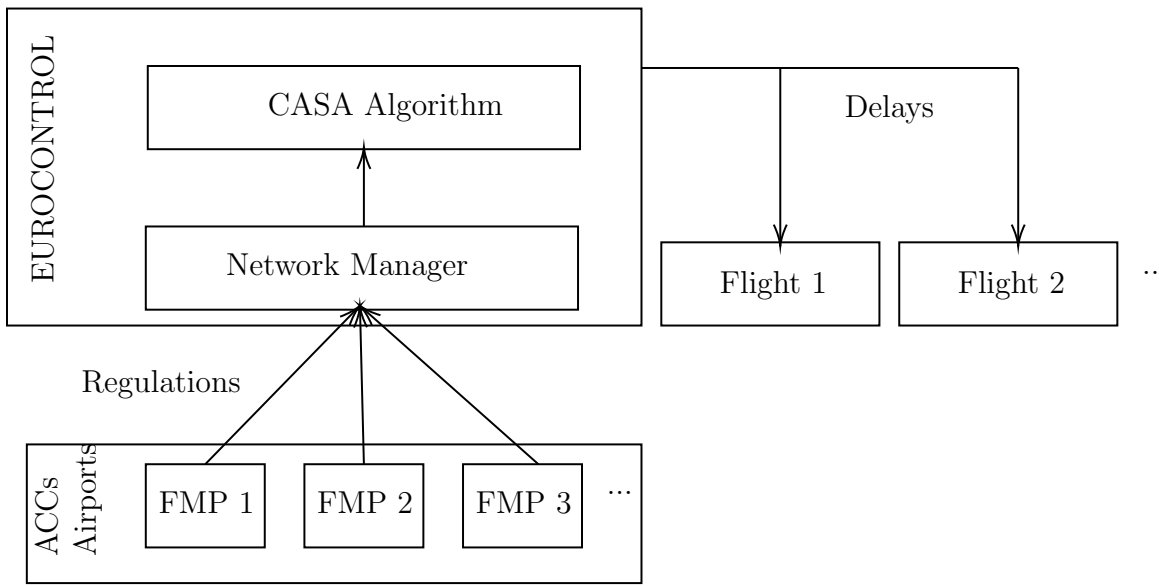

In the European setting, a Flow Management Position (FMP) serves as the interface between local ATFCM partners (Air Traffic Control--ATC, Aircraft Operators--AO, and airports) and the Network Manager (NM). FMPs monitor demand using Monitoring Values (MVs) against capacity values provided to the NM through Traffic Volumes (TVs), and request ATFCM \emph{regulations} when demand exceeds declared capacity.

\begin{figure}
  \centering\begin{tikzpicture}[x=0.75pt,y=0.75pt,yscale=-1,xscale=1]

\draw  [fill={rgb, 255:red, 255; green, 255; blue, 255 }  ,fill opacity=1 ] (320,70) -- (460,70) -- (460,240) -- (320,240) -- cycle ;
\draw  [fill={rgb, 255:red, 255; green, 255; blue, 255 }  ,fill opacity=1 ] (50,60) -- (190,60) -- (190,240) -- (50,240) -- cycle ;
\draw  [fill={rgb, 255:red, 255; green, 255; blue, 255 }  ,fill opacity=1 ] (60,50) -- (200,50) -- (200,230) -- (60,230) -- cycle ;
\draw  [fill={rgb, 255:red, 255; green, 255; blue, 255 }  ,fill opacity=1 ] (70,40) -- (210,40) -- (210,220) -- (70,220) -- cycle ;
\draw  [fill={rgb, 255:red, 255; green, 255; blue, 255 }  ,fill opacity=1 ] (330,60) -- (470,60) -- (470,230) -- (330,230) -- cycle ;
\draw  [fill={rgb, 255:red, 255; green, 255; blue, 255 }  ,fill opacity=1 ] (340,40) -- (480,40) -- (480,220) -- (340,220) -- cycle ;

\draw (145,135) node  [font=\footnotesize] [align=left] {\begin{minipage}[lt]{88.4pt}\setlength\topsep{0pt}
\begin{center}
\textbf{{\footnotesize REGULATION}}
\end{center}
{\footnotesize \textbf{1. Ref. Location:}}\\{\footnotesize LECBBAS}\\{\footnotesize \textbf{2. Flows}}\\{\footnotesize >LEMD}\\{\footnotesize \textbf{3. Time window}}\\{\footnotesize 08:00-09:30}\\{\footnotesize \textbf{4. Entering Flow Rate}}\\{\footnotesize 25}
\end{minipage}};
\draw (18,261) node [anchor=north west][inner sep=0.75pt]   [align=left] {EUROCONTROL Network Manager};
\draw (415,135) node  [font=\footnotesize] [align=left] {\begin{minipage}[lt]{88.4pt}\setlength\topsep{0pt}
\begin{center}
{\footnotesize \textbf{PROGRAM (GDP/AFP)}}
\end{center}
{\footnotesize \textbf{1. Airport or FCA:}}\\{\footnotesize KSFO}\\{\footnotesize \textbf{2. Scope}}\\{\footnotesize To KATL}\\{\footnotesize \textbf{3. Time window}}\\{\footnotesize 08:00-09:30}\\{\footnotesize \textbf{4. Program Rate}}\\{\footnotesize 25}
\end{minipage}};
\draw (400,275) node   [align=left] {\begin{minipage}[lt]{176.8pt}\setlength\topsep{0pt}
\begin{center}
U.S. FAA's Ground Delay and \\Airspace Flow Program
\end{center}

\end{minipage}};

\end{tikzpicture}
  \caption{Regulations (EU) and Flow Programs (US) that will serve as input for slot allocation algorithms.} \label{fig:regulations_and_flow_programs}
\end{figure}
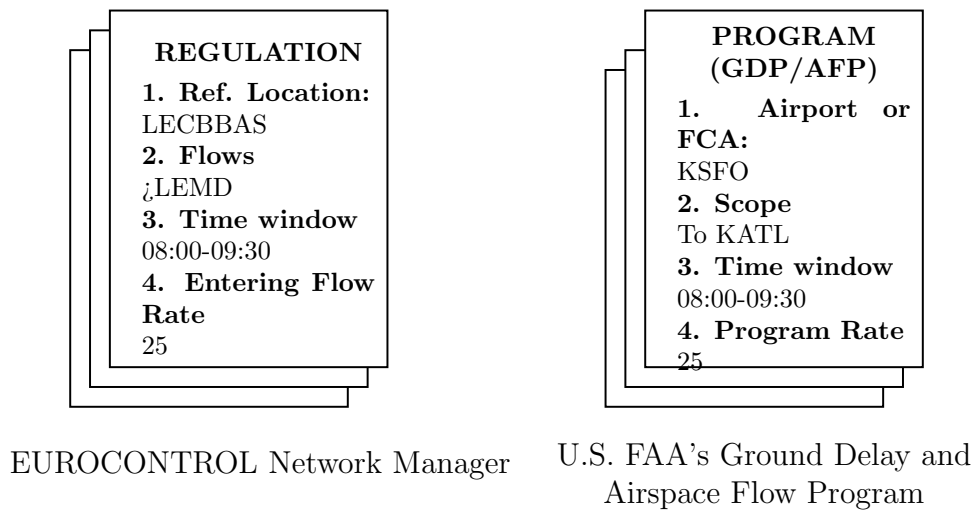

A \textbf{regulation} must specify the \emph{reference location} as the metering point where an FCFS queue is set up, the \emph{time window} during which the queue is active, the \emph{list of flights} (also known as the \emph{flows}) targeted for potential ground delays, and a capped \emph{entry rate} to reduce entries into the area within that time window (Figure \ref{fig:regulations_and_flow_programs}).

The NM may then activate the regulation following coordination,  after which the Computer Assisted Slot Allocation (CASA) algorithm builds the slot list at the declared rates and assigns Calculated Take-Off Times (CTOTs) to captured flights (Figure~\ref{fig:ectl_workflow}) \cite{Eurocontrol_ATFCM_2018}.

Similarly, in the United States, the ATCSCC manages Ground Delay Programs (GDPs) and Airspace Flow Programs (AFPs) by specifying the program window, the affected flight set (through scope, filters, and exemptions), and the target rate. The Traffic Flow Management System (TFMS) then applies the Ration-by-Schedule (RBS) algorithm to assign departure slots and compute EDCTs for impacted flights \cite{FAA_FSM_Users_Guide_2016}.

Although various optimization schemes had been proposed for the DCB problem, commercial implementations of flow control systems worldwide have converged on a single, unified procedure consisting of three steps: flight selection, rate imposition, and First Come First Served (FCFS) queueing for slot assignment. Fundamentally replacing this entrenched workflow would incur prohibitive costs for workforce retraining and require extensive re-certification. Moreover, it would forfeit decades of tacit operational knowledge, resulting in reduced efficiency as untested systems encounter unforeseen scenarios.

This motivates our proposal of Regulation Zero (RZ), a novel framework for automated regulation design that is more compatible with existing flow control operations. However, the conception of any such system must overcome a fundamental technical challenge that is also the primary source of unpredictability in regulation design: the nonlinear and complex interaction between regulations, a phenomenon we term Regulation Cascading (RC). We detail this mechanism in the following section.

\subsection{Regulation Cascading: The Complex Interactions between Regulations, and the Sequential Regulation Scheme}
Designing effective flow regulations is a complex, experience-driven task, especially under uncertain or rapidly evolving conditions like adverse weather. Even with perfect traffic forecasts, the fragmented process by which individual FMP positions at each ACC propose regulations can produce measures that conflict with one another. For instance, during periods of convective weather over Central Europe, adjacent FMPs in Munich, Vienna, and Zurich may independently impose flow regulations to protect their sectors. Without cross-border coordination, these local measures can create overlapping or contradictory restrictions along shared traffic flows: one regulation may inadvertently divert flights into sectors already constrained by another, producing unpredictable cascading effects and increasing the NM's workload as it reconciles competing constraints under time pressure.

\subsubsection{Cascading Mechanism}

However, an important question remains: through what exact mechanism do such cascading effects arise? A baseline regulation strategy is to directly regulate TVs exhibiting DCB imbalances by applying a ``blanket'' entry-rate cap equal to the declared capacity that controllers can safely accommodate. Because each flight usually traverses multiple TVs, this approach leads to three fundamental problems:

\begin{enumerate}
  \item First, the geographic extent of demand redistribution after regulation activation can be too broad to contain effectively. To mitigate this, the EUROCONTROL Flow Concept of Operations \cite{EurocontrolFlowConops2024} explicitly advocates a narrower, more surgical approach that designates the exact groups of flights (i.e., \emph{flows}) to regulate. Although this reduces the risk of secondary hotspots, it introduces some new questions: which flows should be selected, and what specific rate should be capped for each flow, given that the baseline strategy of using a TV's declared capacity does not transfer directly to the flow level.
  \item Second, each regulation reshapes demand elsewhere in the network, so the flights targeted by one FMP may no longer match those observed by another FMP, eroding the shared situational awareness on which coordination depends.
  \item Third, in CASA, all regulations are applied simultaneously to obtain the candidate slots, and if the same flight is targeted by multiple regulations, it receives the maximum delay under the Most Penalizing Regulation (MPR) \cite{EUROCONTROL_ATFCM_Manual_2024}. This delay arbitration mechanism introduces a late-stage demand shift that further contributes to the scheme's unpredictability.
\end{enumerate}

\subsubsection{Regulation Planning is a Markov Decision Process (MDP)}
This analysis suggests that, to address the RC problem, regulations should be planned as an ordered sequence to avoid ambiguity in the demand picture and to prevent late-stage delay arbitration under MPR (Figure \ref{fig:arbitration}). This sequential structure aligns naturally with an MDP, in which each regulatory action induces a state transition. For example, displacing traffic into an adjacent area alters its flow composition and thus affects the next potential regulation in that area.

However, the MDP formulation introduces substantially greater complexity: the planner must account for temporal dependencies and the delayed effects of each regulation. In practice, this means the agent cannot optimize each regulation in isolation; instead, it must learn a policy that anticipates how local relief measures may propagate through the network and inadvertently compromise downstream stability.

\begin{figure}
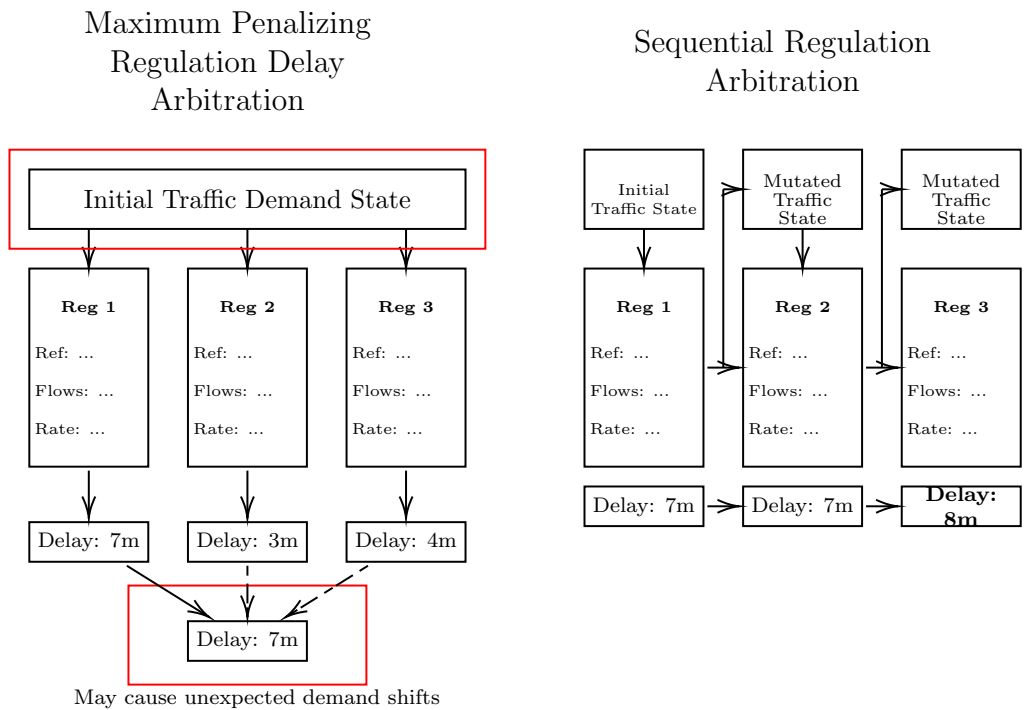

  \centering
  \include{arbitration}
  \caption{Comparison of the MPR and RZ delay-arbitration schemes when the same flight is targeted by two regulations. In the MPR scheme, all regulations are applied in batch, and delays are arbitrated only at the final step, causing unpredictable demand shifts. In the RZ scheme, delays are applied before each subsequent regulation, so the traffic demand state remains accurate at every step. As a result, the last-minute demand shifts observed under MPR are avoided. The MPR scheme also assumes a common demand picture for all regulations during their conception, which may be stale. The sequential scheme does not exhibit this vulnerability.}
  \label{fig:arbitration}
\end{figure}

\subsection{Our Contributions}
The preceding discussion prompted us to design RZ based on Monte Carlo Tree Search (MCTS) \cite{silver2018general} to find \textbf{the optimal sequence of flow regulations} that resolves DCB imbalances at the network level. Because the traffic situation evolves naturally within the framework, RZ aims to support operators in managing complex cross-border regulation interactions that are prone to unpredictability and require substantial coordination effort.

The framework is also compatible, in principle, with existing ATFM workflows worldwide, allowing integration with minimal operational disruption. It preserves fairness through the familiar First Planned First Served (FPFS) principle while delivering a level of network performance that often exceeds what was traditionally achieved only by flight-centric solvers. Moreover, because the search process operates within the regulation space---the same mental model used by flow management experts---as it naturally accommodates expert intuition, enabling experts to constrain and guide the search toward preferred solutions when needed.

While the end goal is still delaying flights, it is crucial to emphasize that Regulation Zero is not a \emph{slot allocation optimizer}, but a dynamic learning system designed to refine regulatory decisions through testing and observing the network responses iteratively. \textbf{This is made evident by noting that Regulation Zero interacts with the FPFS slot allocation machinery, steering it to yield favorable outcomes rather than doing the optimization by itself. In a strict sense, Regulation Zero is a meta-planner for the FPFS slot allocation algorithm.}

\subsection{Paper Organization}
The remainder of this paper is organized as follows. Section 2 reviews the literature on Traffic Management Initiative (TMI) optimization, with emphasis on ground delay strategies. Section 3 introduces the proposed framework, formalized as a cascaded optimization scheme with three components: an MCTS-based regulation search process, a flow-level regulation proposal engine for each overloaded TV, and a slot allocation algorithm for computing individual flight delays. Section 4 describes large-scale experiments simulating several full days of pan-European air traffic during the 2023 summer peak, each involving 20,000-25,000 flights, along with comparison gainst Simulated Annealing (SA) and NSGA-II baselines. We discuss the characteristics of RZ solutions and broader findings in Section 5. Section 6 concludes the paper.

\section{Literature Review}
\begin{figure}
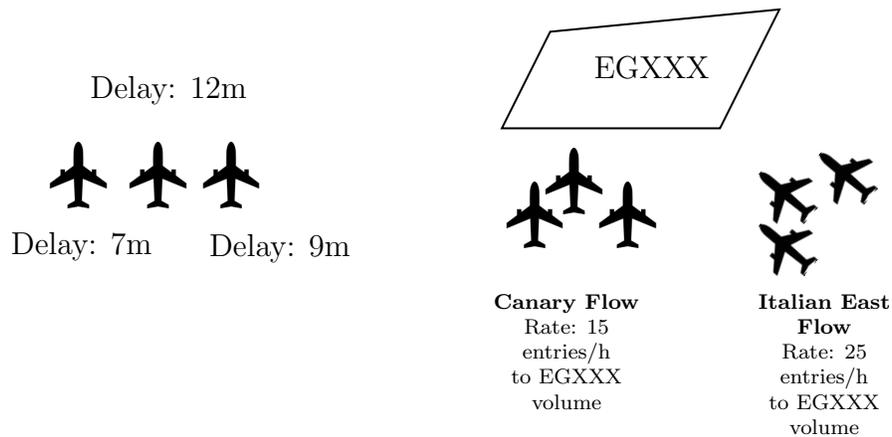

  \centering
  \include{trajectory-space-v-regulation-space}
  \caption{Illustration of Trajectory (left) and Regulation Spaces (right). In Trajectory Space, decision variables are assigned per flight. In Regulation Space, decision variables are per flow.}
  \label{fig:trajvreg_space}
\end{figure}

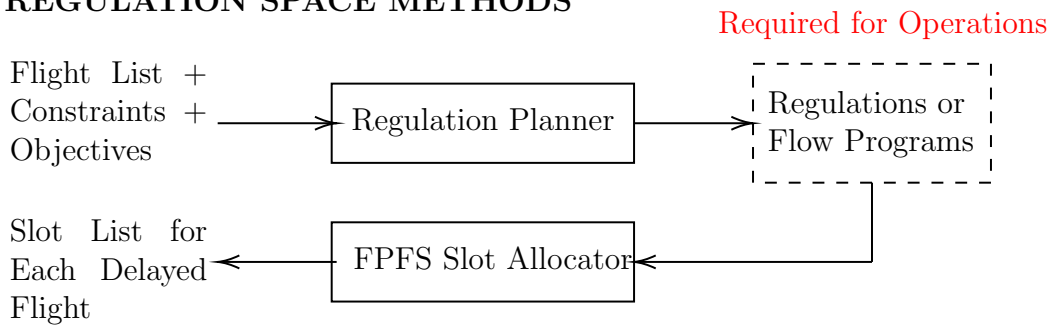
\begin{figure}
  \centering
  \begin{tikzpicture}[x=0.75pt,y=0.75pt,yscale=-1,xscale=1]

\draw   (247.5,70) -- (430,70) -- (430,110) -- (247.5,110) -- cycle ;
\draw    (190,90) -- (248,90) ;
\draw [shift={(250,90)}, rotate = 180] [color={rgb, 255:red, 0; green, 0; blue, 0 }  ][line width=0.75]    (10.93,-3.29) .. controls (6.95,-1.4) and (3.31,-0.3) .. (0,0) .. controls (3.31,0.3) and (6.95,1.4) .. (10.93,3.29)   ;
\draw    (430,90) -- (488,90) ;
\draw [shift={(490,90)}, rotate = 180] [color={rgb, 255:red, 0; green, 0; blue, 0 }  ][line width=0.75]    (10.93,-3.29) .. controls (6.95,-1.4) and (3.31,-0.3) .. (0,0) .. controls (3.31,0.3) and (6.95,1.4) .. (10.93,3.29)   ;
\draw   (247.5,230) -- (400,230) -- (400,270) -- (247.5,270) -- cycle ;
\draw    (190,250) -- (248,250) ;
\draw [shift={(250,250)}, rotate = 180] [color={rgb, 255:red, 0; green, 0; blue, 0 }  ][line width=0.75]    (10.93,-3.29) .. controls (6.95,-1.4) and (3.31,-0.3) .. (0,0) .. controls (3.31,0.3) and (6.95,1.4) .. (10.93,3.29)   ;
\draw    (400,250) -- (458,250) ;
\draw [shift={(460,250)}, rotate = 180] [color={rgb, 255:red, 0; green, 0; blue, 0 }  ][line width=0.75]    (10.93,-3.29) .. controls (6.95,-1.4) and (3.31,-0.3) .. (0,0) .. controls (3.31,0.3) and (6.95,1.4) .. (10.93,3.29)   ;
\draw   (247.5,300) -- (400,300) -- (400,340) -- (247.5,340) -- cycle ;
\draw    (520,320) -- (402,320) ;
\draw [shift={(400,320)}, rotate = 360] [color={rgb, 255:red, 0; green, 0; blue, 0 }  ][line width=0.75]    (10.93,-3.29) .. controls (6.95,-1.4) and (3.31,-0.3) .. (0,0) .. controls (3.31,0.3) and (6.95,1.4) .. (10.93,3.29)   ;
\draw    (520,280) -- (520,320) ;
\draw    (250,320) -- (192,320) ;
\draw [shift={(190,320)}, rotate = 360] [color={rgb, 255:red, 0; green, 0; blue, 0 }  ][line width=0.75]    (10.93,-3.29) .. controls (6.95,-1.4) and (3.31,-0.3) .. (0,0) .. controls (3.31,0.3) and (6.95,1.4) .. (10.93,3.29)   ;
\draw  [dash pattern={on 4.5pt off 4.5pt}] (460,220) -- (580,220) -- (580,280) -- (460,280) -- cycle ;

\draw (256,82) node [anchor=north west][inner sep=0.75pt]   [align=left] {Slot Allocation Algorithm};
\draw (135,95) node   [align=left] {\begin{minipage}[lt]{74.8pt}\setlength\topsep{0pt}
Flight List + Constraints + Objectives
\end{minipage}};
\draw (555,95) node   [align=left] {\begin{minipage}[lt]{74.8pt}\setlength\topsep{0pt}
Slot List for Each Delayed Flight
\end{minipage}};
\draw (81,32) node [anchor=north west][inner sep=0.75pt]   [align=left] {\textbf{TRAJECTORY SOLUTION SPACE METHODS}};
\draw (81,181) node [anchor=north west][inner sep=0.75pt]   [align=left] {\textbf{REGULATION SPACE METHODS}};
\draw (135,245) node   [align=left] {\begin{minipage}[lt]{74.8pt}\setlength\topsep{0pt}
Flight List + Constraints + Objectives
\end{minipage}};
\draw (256,241) node [anchor=north west][inner sep=0.75pt]   [align=left] {Regulation Planner};
\draw (466,232) node [anchor=north west][inner sep=0.75pt]   [align=left] {Regulations or\\Flow Programs};
\draw (257,311) node [anchor=north west][inner sep=0.75pt]   [align=left] {FPFS Slot Allocator};
\draw (135,325) node   [align=left] {\begin{minipage}[lt]{74.8pt}\setlength\topsep{0pt}
Slot List for Each Delayed Flight
\end{minipage}};
\draw (441,192) node [anchor=north west][inner sep=0.75pt]  [color={rgb, 255:red, 255; green, 0; blue, 0 }  ,opacity=1 ] [align=left] {Required for Operations};

\end{tikzpicture}
  \caption{Comparison between the optimization pathways of slot-allocation algorithms and regulation-based methods.}
  \label{fig:compare_reg_spaces}
\end{figure}

\begin{table}[h]
  \centering
  \caption{Comparison between trajectory-space and regulation-space solution representations.} \label{tab:rs_vs_tbo}
  \begin{tabularx}{\linewidth}{p{50pt} X X}
    \toprule
    {\textbf{Aspect}} & {\textbf{Trajectory Space}} & {\textbf{Regulation Space}} \\
    \midrule
    Number of variables & \textbf{Fixed:} It can be represented by a fixed binary delay vector whose length equals the total number of flights. & \textbf{Variable:} The number of regulations to be created can vary from a dozen to hundreds.\\
    \midrule
    Historical dependence & \textbf{State-invariant action set:} The number of changes that can be made to any flight's decision variables depends very little on history (although each change may still be accepted or rejected later by the optimizer). & \textbf{State-dependent action set:} Each regulation proposal must be based on the selection of a reference location (typically a hotspot) and the traffic flows associated with that location. Both pieces of information depend on previously imposed regulations.\\
    \midrule
    Evaluation & \textbf{Direct:} Objective functions are relatively straightforward to compute after a decision-variable change. & \textbf{Multi-stage with a nested slot-allocation algorithm:} Regulations must first be parsed correctly; CASA then assigns delays, after which the objective function can be evaluated.\\
    \midrule
    Fine-tuning & \textbf{Limited:} If flight plans change, a complete re-run is usually required, optionally with an additional convexification term in the objective function. & \textbf{Flexible:} Multiple options are available: tuning the rate, adjusting flow filters, or performing a complete re-run.\\
    \bottomrule
  \end{tabularx}
\end{table}

While flight-centric formulations dominate the DCB optimization literature, regulation-based approaches are often better aligned with operational practice, particularly because they naturally accommodate flow-expert input at the design stage. In this section, we briefly review both approaches and clarify their distinctions. Figure~\ref{fig:trajvreg_space} illustrates the decision variables underlying each approach, while Figure~\ref{fig:compare_reg_spaces} contrasts their respective workflows. We also summarize the key differences in Table~\ref{tab:rs_vs_tbo}.

\subsection{Trajectory Space Slot Optimization Methods}
Classical slot allocation in Air Traffic Flow Management (ATFM) originates from ground-holding formulations that optimally assign departure delays under capacity constraints. The foundational single-airport static problem, as proposed by Richetta and Odoni in \cite{RichettaOdoni1993}, established exact optimization principles for minimizing total delay. This formulation was subsequently extended to multi-airport settings, where coordinated delay assignments across interconnected airports were shown in \cite{VranasBertsimasOdoni1994} to yield globally consistent scheduling solutions. Building on these developments, network-flow models later incorporated en-route capacities and routing decisions, anchoring the minimize-total-delay paradigm and demonstrating scalability to realistic traffic networks \cite{BertsimasPatterson1998,BertsimasPatterson2000}.

Recent advances in ATFM slot allocation and trajectory space optimization cluster around three interrelated methodological categories:

\subsubsection{Stochastic and MILP formulations under capacity and uncertainty}
Extending beyond single-airport scheduling, recent research integrates airport and fix capacities in multi-airport system (MAS) formulations. Wang \emph{et al.} \cite{Wang2023TRC} and Liu \emph{et al.} \cite{Liu2024TJSASS,Liu2025Aerospace} proposed mixed-integer and robust optimization models that explicitly capture uncertainties in flight and execution times through chance-constrained formulations. These models achieve tractable MILP representations that enhance schedule robustness with limited displacement amplification. Complementarily, He and Pan \cite{HePan2024PLOSONE} coupled capacity-envelope estimation with slot allocation, thereby integrating strategic flow control with empirically grounded throughput constraints.

\subsubsection{Market mechanisms and meta-heuristics}
To accommodate heterogeneous airline valuations and preserve privacy in collaborative decision-making, a growing body of research has introduced auction-based and credit-driven reallocation mechanisms for ATFM slots. Schuetz et al. \cite{Schuetz2022ICAS} proposed a privacy-preserving slot market that uses credits and evolutionary optimization to reorder slots based on private airline preferences. In parallel, Lee et al. \cite{Lee2024JATM} proposed an auction-based reallocation scheme that explicitly accounts for airlines' grandfather rights, together with allocation, payment, and compensation rules.

For large and tightly-constrained problems, such as single- and multi-airport settings, enhanced evolutionary and learning-inspired heuristics have delivered high-quality solutions at scale: for example, an estimation-of-distribution metaheuristic coordinating multiple ATFM measures for slot allocation \cite{Tian2021CIE}, and iterated local search / variable neighborhood search tailored to the simultaneous, system-wide slot allocation problem \cite{Pellegrini2011SSRN}. From the airline perspective during severe airspace flow programs, genetic algorithms that encode route and slot decisions have been shown to reduce disruption and cost \cite{Abdelghany2007JATM}. Beyond pure slotting, hybrid metaheuristics designed by \emph{Chaimatanan}, \emph{Delahaye}, and \emph{Mongeau} strategically deconflict trajectories at continental scale jointly adjusting routes, levels and departure times to balance safety and efficiency \cite{Chaimatanan2014CIM}. Complementing these, co-evolutionary genetic algorithms have been applied to ATFM under dynamic capacity, illustrating how cooperative search can handle network-level congestion while respecting capacity constraints and time-slot assignments \cite{Zhang2007ITSC}.

\subsubsection{Network-wide, multi-objective fairness, connectivity, and environmental goals}
Another emerging direction embeds broader operational and environmental criteria into slot allocation. Keskin and Zografos \cite{KeskinZografos2023TRB}, Feng \cite{Feng2023TRD}, and Dalmau \cite{Dalmau2024SIDs} formulated bi-objective and lexicographic optimization models that incorporate fairness across airlines, onward connectivity preservation, and environmental impact mitigation (e.g., noise abatement). These formulations explicitly characterize trade-offs between delay displacement, connectivity protection, and sustainability, contributing to a framework for MAS-aware, data-driven slot allocation that is both robust to uncertainty and sensitive to multi-stakeholder objectives.

Empirical analyses and simulation-based studies further contextualize these approaches. Lau \emph{et al.} \cite{LauBuddeBerlingGollnick2014} evaluated CASA-like heuristics at network scale, benchmarking their performance against optimization-based baselines. Ivanov \emph{et al.} \cite{IvanovEtAl2017} proposed a two-level mixed-integer model that shapes delay distributions to mitigate propagation effects and improve adherence to airport slot assignments. Similarly, Chen et al. \cite{ChenChenSun2017} demonstrated that chance-constrained formulations can incorporate capacity uncertainty while maintaining computational tractability for large-scale instances. Most recently, Berling, Lau, and Gollnick \cite{BerlingLauGollnick2024} integrated a constraint-reconciliation and column-generation allocator into R-NEST, achieving quantifiable reductions in both primary and reactionary delays while respecting operational constraints on the day of execution.

Despite important advances, most prior work frames ATFM as direct, deterministic, flight-level slot/delay allocation, which typically demands deep re-engineering, testing, and certification of legacy systems. This flight-centric view is also not fully aligned with practice: human operational expertise typically resides in regulation space, in terms of rates and counts. This makes controller intuition harder to inject into slot-based optimization and complicates verification of regulation-plan soundness against edge cases.

\subsection{Regulation Space Optimization Methods}
Despite pressing operational needs, optimization at the regulation level remains relatively underexplored. \cite{Dalmau2021Indispensable} observed that not all regulations are equally important, and building on this insight, \cite{Dalmau2022Optimal} proposed constructing regulations greedily by capping traffic flow at declared capacity.

\subsection{Positioning of our Paper}
The foregoing review reveals a clear gap between existing research and operational practice that this paper aims to address:

\begin{itemize}
  \item In current operations, virtually all ATFM systems worldwide implement the same slot-allocation principle, in which delays are produced by algorithms that execute sets of \emph{regulations} manually designed by expert flow managers.
  \item Even forthcoming system evolutions, such as the \emph{Targeted CASA} initiative, focus primarily on providing finer control over traffic selection and rate adjustment rather than fundamentally replacing the underlying FPFS allocation mechanism \cite{EUROCONTROL_NOP_2024_2029}.
  \item Existing formulations of \emph{automated regulation design} insufficiently capture two essential dimensions of a regulatory decision: adjusting the \emph{regulation rate} and explicitly selecting \emph{contributing flows}. Consequently, developing a comprehensive, fully automated regulation-planning methodology remains an open challenge.
  \item EUROCONTROL also noted that interacting flow-management measures can yield complex and sometimes counterintuitive effects \cite{EUROCONTROL_ATFCM_Manual_2024}. Therefore, automation of regulation design remains relatively new and under-researched, despite its potentially high operational impact.
\end{itemize}

Our paper directly addresses this unmet operational and methodological gap. We frame our contribution as a \emph{decision-support system} that extends existing ATFCM operations rather than requiring their replacement. Specifically, we propose develop an optimization-based methodology that automatically constructs regulation plans by jointly determining rate profiles and flow selection under realistic demand and capacity constraints. \textbf{The solutions produced by our framework are, in principle, compatible with both current CASA operations and forthcoming \emph{Targeted CASA}, since the framework supports either single-rate or multi-rate specifications across subflows} \cite{EUROCONTROL_NOP_2024_2029}.

\section{Methodology}\label{sec:method}
\subsection{Problem Formulation}
Owing to the structural resemblance of our problem to sequential decision-making under uncertainty, we adopt the MCTS framework from AlphaZero \cite{silver2018general} to derive the optimal \emph{sequence of regulations}. MCTS balances exploration and exploitation via optimistic tree search \cite{keller2012prost}: at each iteration, candidate regulations are simulated, and observed rewards are back-propagated to refine the value estimates of each action. As the exploration bonus decays, the policy converges toward the highest-value regulation sequence. Since the regulation decision space is multi-dimensional: spanning reference location, time window, flow selection, and capacity rate, and each candidate requires invoking an FCFS allocator for evaluation, we extend the standard MCTS to a \emph{hierarchical} setting to improve computational tractability.

\subsubsection{Definitions} \label{sec:formalization_reg}
Fix a bin width of $\Delta = 15$ minutes. Let the planning day be partitioned into $96$ consecutive time bins indexed by
\[
  \mathcal{T} \equiv \{0,1,\ldots,95\}.
\]
For $t \in \mathcal{T}$, bin $t$ corresponds to the half-open interval $[t\Delta,(t+1)\Delta)$ minutes after midnight.
For integers $a \le b$, we use the discrete interval notation:
\[
  [a,b]_{\mathbb{Z}} \equiv \{a,a+1,\ldots,b\}.
\]

Let $\VV$ denote the finite set of all monitored TVs, and let $\mathcal{F}$ denote the set of flights considered in the planning horizon. For a flight $f \in \mathcal{F}$ traversing volume $v \in \VV$, let $e_{v,f} \in [0, 24\mathrm{h})$ and $x_{v,f} \in (0, 24\mathrm{h}]$ denote its \textbf{entry and exit time}, respectively. Either value may be absent when a flight originates or terminates within $v$, or does not intersect $v$ at all.

We discretize the entry and exit time by defining the entry-bin map

\begin{equation}\label{eq:bin_map}
b(t) \equiv \lfloor t/\Delta \rfloor
\end{equation}
so that the entry bin of $f$ into $v$ is $b(e_{v,f}) \in \mathcal{T}$.

For each volume $v \in \VV$ and $t \in \mathcal{T}$, define the within-bin entry count
\[
  E_{v,t} \;\equiv\; \bigl|\{f \in \mathcal{F} : b(e_{v,f}) = t \}\bigr|.
\]
The rolling-hour (or a 60-min sliding window) entry \textbf{demand} starting at bin $t$ is
\[
  D_{v,t} \;\equiv\; \sum_{k=0}^{3} E_{v,\,t + k},
\]
i.e., the number of entries into $v$ over the hour $[t\Delta, (t+4)\Delta)$.

Let $C_{v,t} \in \mathbb{N}$ denote the declared (rolling-hour) entry \textbf{capacity} for volume $v$ associated with the rolling window starting at bin $t$. We define the excess (overload) as
\[
  G_{v,t} \;\equiv\; (D_{v,t} - C_{v,t})_+.
\]

A \textbf{hotspot} $h = (v, t_i, t_f)$ for volume $v$ is a maximal contiguous interval $[t_i,t_f]_{\mathbb{Z}} \subseteq \mathcal{T}$ such that
\[
  G_{v,t} > 0, \quad \forall\, t \in [t_i,t_f]_{\mathbb{Z}}.
\]

A \textbf{regulation} is a 5-tuple
\begin{equation}\label{eq:regulation_def}
  \rho \;=\; \bigl(v_c,\, t_{i,c},\, t_{f,c},\, \tau_c, \kappa_c \bigr),
\end{equation}
where $v_c \in \VV$ is the control volume (CV), $[t_{i,c}, t_{f,c}]_{\mathbb{Z}} \subseteq \mathcal{T}$ is the regulation's \emph{active window}, $\kappa_c$ is the flow selection operator, and $\tau_c \in \mathbb{N}$ is the hourly entry-rate limit imposed at $v_c$. The CV is a TV at which a metering queue is established; and in the EUROCONTROL operational context, this corresponds to the ``reference location'' \cite{EurocontrolFlowConops2024}. The flow selection operator $\kappa_c$ denotes the mapping
\begin{equation}\label{eq:flow_sel_op}
  \kappa_c:\mathcal{F}\to {0,1}
\end{equation}
which selects the flights to be metered by an FCFS queue at the CV by the regulation.

Because $D_{v,t}$ also includes traffic in the next three bins, a regulation should remain active for the same amount of time, or 45 minutes, after the hotspot ends. We therefore define the effective active window:
\begin{equation}\label{eq:Iceff}
  I_c^{\mathrm{eff}} \;\equiv\; [\,t_{i,c},\; t_{f,c} + \Delta_r\,]_{\mathbb{Z}} \cap \mathcal{T}, \quad \Delta_r = 3.
\end{equation}

During $I_c^{\mathrm{eff}}$, the regulated entry counts should satisfy:
\[
  D_{v_c,t}^{(\rho)} \;\le\; \tau_c, \qquad \forall\, t \in I_c^{\mathrm{eff}},
\]
where $D_{v_c,t}^{(\rho)}$ denotes the rolling-hour entry count at $v_c$ after the regulation-induced metering is applied.

A (partial) \textbf{regulation plan} is an ordered finite sequence of regulations:
\[
  s \;=\; (\rho^1, \rho^2, \ldots, \rho^K), \qquad K \in \mathbb{N}_0.
\]
Its cardinality is $|s| = K$. We regard $s$ as the \emph{state} of the decision process. Given the plan $s$, we write $D_{v,t}^{(s)}$ for the induced rolling-hour entry count after applying all regulations $\rho^{k}$ in $s$.
\subsubsection{FPFS Slot Allocation (FIFO Queue)}\label{subsec:fpfs_slot}
For a regulation $\rho=(v_c,t_{i,c},t_{f,c},\tau_c, \kappa_c)$ with margins $\Delta_e,\Delta_r$.
Let the slot spacing be
\[
  \sigma_c \equiv \frac{60}{\tau_c}\ \text{minutes}, \qquad \tau_c \in \mathbb{N},\ \tau_c \ge 1,
\]
and define the slot sequence anchored at $a_c = t_{i,c} - \Delta_e$ by
\[
  s_{c,0} \equiv a_c,\qquad s_{c,m} \equiv s_{c,0} + m\,\sigma_c,\quad m \in \mathbb{N}_0.
\]
Let the \textbf{set of regulated flights} be
\begin{equation}\label{eq:regulated_flight_list}
  \mathcal{F}_c \equiv \left\{f\in\mathcal{F}:\ f\text{ traverses }v_c,\ \hat e^{(s)}_{v_c,f}\in I_c^{eff},\ \kappa_c(f)=1\right\}.
\end{equation}
Order $\mathcal{F}_c$ by nondecreasing planned entry times and a fixed deterministic tie-breaker (e.g., by flight identifier):
\[
  e_{v_c,f^{(1)}} \le e_{v_c,f^{(2)}} \le \cdots \le e_{v_c,f^{(N_c)}},\quad N_c \equiv |\mathcal{F}_c|.
\]
Define $m_0 \equiv -1$ and, for $j=1,\ldots,N_c$,
\begin{equation}
  m_j \equiv \max\!\left\{\, m_{j-1}+1,\ \left\lceil \frac{e_{v_c,f^{(j)}} - s_{c,0}}{\sigma_c} \right\rceil \right\},
  \qquad
  \hat e_{v_c,f^{(j)}}^{(\rho)} \equiv s_{c,0} + m_j\,\sigma_c.
\end{equation}
The \textbf{concrete delay} (in minutes) assigned to $f^{(j)}$ by $\rho$ is
\begin{equation}
  d_{f^{(j)}}^{(\rho)} \equiv \hat e_{v_c,f^{(j)}}^{(\rho)} - e_{v_c,f^{(j)}} \;\;\in\; [0,\infty).
\end{equation}
For flights not subject to $\rho$ (i.e., $f \notin \mathcal{F}_c$), we set $\hat e_{v_c,f}^{(\rho)} \equiv e_{v_c,f}$ and $d_f^{(\rho)} \equiv 0$.

\subsubsection{Objective Function}\label{subsec:objective_fn}

Given a (partial) regulation plan $s=(\rho^1,\ldots,\rho^K)$, we define a scalar objective $J(s)$ to be minimized.

We assume that the plan-induced delay $d_f^{(s)}$ shifts the entire trajectory of flight $f$ uniformly in time (i.e., without wind-speed compensation). For any TV $v\in\VV(f)$, the revised (post-regulation) entry time is
\[
  \hat e_{v,f}^{(s)} \;\equiv\; e_{v,f} + d_f^{(s)}.
\]
In the absence of regulation, $d_f^{(s)}=0$, and thus $\hat e_{v,f}^{(s)}\equiv e_{v,f}$.

As in Section \ref{sec:formalization_reg}, the within-bin entry counts induced by regulation plan $s$ are
\[
  E_{v,t}^{(s)} \;\equiv\; \bigl|\{\, f\in\mathcal{F} : v\in\VV(f)\ \text{and}\ b(\hat e_{v,f}^{(s)})=t \,\}\bigr|,
  \qquad v\in\VV,\ t\in\mathcal{T}.
\]
The post-regulation rolling-hour entry demand is
\[
  D_{v,t}^{(s)} \;\equiv\; \sum_{k=0}^{3} E_{v,t+k}^{(s)},\qquad v\in\VV,\ t\in\mathcal{T},
\]
with the convention that $E_{v,t}^{(s)}=0$ for $t\notin\mathcal{T}$.

\paragraph{Objective.}
Let $C_{v,t}\in\mathbb{N}$ be the hourly entry capacity for volume $v$ associated with the rolling window
starting at bin $t$. For weights $w_{\mathrm{DELAY}}, w_{\mathrm{REG}} \geq 0$, define
\begin{equation}\label{eq:j_eq}
  J(s) \;=\; J_{\mathrm{CAP}}(s) \;+\; w_{\mathrm{DELAY}}\, J_{\mathrm{DELAY}}(s),
\end{equation}
with
\[
  J_{\mathrm{CAP}}(s) \;\equiv\; \sum_{v\in\VV} \sum_{t\in\mathcal{T}} \bigl(D_{v,t}^{(s)} - C_{v,t}\bigr)_+,
  \qquad
  J_{\mathrm{DELAY}}(s) \;\equiv\; \sum_{f\in\mathcal{F}} d_f^{(s)}.
\]

\paragraph{Optimization problem statement.}
Let $\mathcal{S}$ denote the set of admissible regulation plans $s=(\rho^1,\ldots,\rho^K)$. Under the FPFS slot allocator, the plan-induced quantities $E^{(s)}$, $D^{(s)}$, and $d^{(s)}$ are defined as above. The planning problem is to find $s^*$ that minimizes $J(s)$:
\[
  \boxed{s^* = \arg\min_{\,s \in \mathcal{S}} \; J(s)}.
\]

\subsection{Flow-Centric DCB Heuristics}

Based on the objective function in Section \ref{subsec:objective_fn}, we propose two heuristics to find promising flows to regulate. However it is important to stress that these heuristics are not prescriptive because flow and rate selection are strongly coupled, so even the best rate choice may still degrade network performance.

\subsubsection{Flight and Flow Footprint}\label{subsubsec:footprint_and_time_alignment}
For a flight $f \in \mathcal{F}$, the ordered TV path of $f$ by their corresponding entry time $\{e_{v,f}\}$ is
\begin{equation}
	L_f \equiv \bigl(v_1(f), v_2(f), \ldots, v_{K_f}(f)\bigr),
	\qquad
	v_k(f) \in \VV.
\end{equation}
Its TV footprint is the set
\begin{equation}
	\operatorname{Foot}(f) \equiv \{v_1(f), v_2(f), \ldots, v_{K_f}(f)\} \subseteq \VV.
\end{equation}
For a flow $\CC_j \subseteq \mathcal{F}$, the flow footprint is
\begin{equation}
	\operatorname{Foot}(\CC_j) \equiv \bigcup_{f \in \CC_j} \operatorname{Foot}(f).
\end{equation}

For each volume $v$ in $\operatorname{Foot}(\CC_j)$, we define its associated \emph{flow volume's touched window} as:
\begin{equation}
	\II_{v,\CC_j} = [b(e_{v,min}), b(e_{v,max})],
\end{equation}
where
\begin{align}
	e_{v,min} &= \min_f \{e_{v,f}, f \in \CC_j \} \\
	e_{v,max} &= \max_f \{e_{v,f}, f \in \CC_j\}
\end{align}
which is the earliest and latest entry time of the flow $\CC_j$ into volume $v$ accordingly. 
\subsubsection{Nominal Relief (NomRel) and Induced Overload (InLoad)}\label{subsubsec:nomrel_inload}
We define the set $\HV(\CC_j)$ of ``hot volumes'' associated with the flow $\CC_j$ as:
\begin{equation}
	\HV(\CC_j) = \operatorname{Foot}(\CC_j) \cap \{ v \in \VV: \exists t \in \II_{v,\CC_j} \quad D_{v,t} \geq C_{v,t} \},
\end{equation}
in other words, the traffic volumes where the flow "touch" at least one time bin that observes overload. We also define the flow's ``touched window'' as:
\begin{equation}
	\TC(\CC_j) = \{ (v, \II_{v, \CC_j}): v \in \operatorname{Foot}(\CC_j) \},
\end{equation}
and the flow's ``hot cells'' as:
\begin{equation}
	\HC(\CC_j) = \{ (v, \II_{v, \CC_j} \cap \{ t: D_{v,t} \geq C_{v,t} \}): v \in \operatorname{Foot}(\CC_j) \}.
\end{equation}

Intuitively, when we select a flow to solve potentially many hotspots at the same time, our primary goal is to relieve the excess traffic at all hot cells, while measuring the potential ``induced load'' due to demand shifts at the time bins that immediately follow the ``touched window.'' 

\begin{figure}
	\centering
	\includegraphics[width=\linewidth]{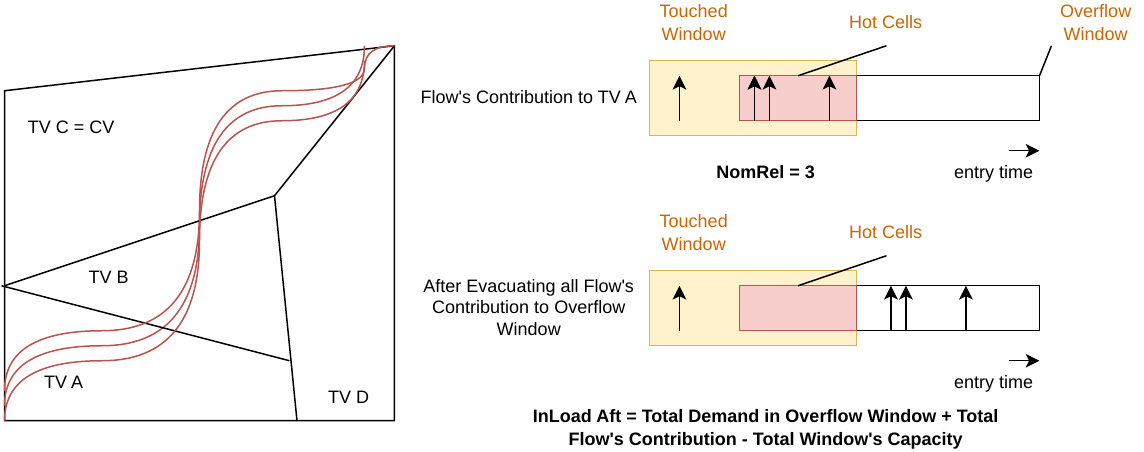}
	\caption{Illustration of the computation of the NomRel and InLoad heuristics. In essence, NomRel shows optimistically how much overload could be relieved from the current hotspots, whereas InLoad shows pessimistically how much secondary overload will be induced elsewhere in the network.}
	\label{fig:rzheuristics}
\end{figure}

The \emph{Nominal Relief} is proposed to measure the excess relief potential if we ``evacuate'' all of the flow's contribution to demand outside its hot cells:
\begin{equation}
	\operatorname{NomRel} (\CC_j) = \sum_{v \in \VV} \sum_{t \in \HC(\CC_j)} D_{v,t,\CC_j},
\end{equation}
where $D_{v,t,\CC_j}$ is the demand attributed to flow $\CC_j$. The higher $\operatorname{NomRel}$ is, the greater the potential for regulating a flow to bring positive benefits to the network.

As regulations only shift traffic to later time bins, they increase the risk of inducing secondary hotspots or worsening excess traffic in already overloaded cells. The time bins that immediately follow the touched window present the highest risk. We define the \emph{overflow window} as:

\begin{equation}
	\operatorname{OC}(\CC_j) = \{ (v, \II_{v, \CC_j}^+): v \in \operatorname{Foot}(\CC_j) \}, 
\end{equation}
where
\begin{equation}
	\II_{v, \CC_j}^+ = \{ t+1, t+2, t+3, t+4 \}; \quad \quad t = b(e_{v,max})
\end{equation}
or 4 time bins that immediately follow the flow volume's touched window. 

\begin{equation}
	\operatorname{InLoad\_Aft}(\CC_j) = \sum_v \left(\sum_{t \in \II_{v,\CC_j}^+} D_{v,t} + \sum_{t \in \II_{v,\CC_j}} D_{v,t,\CC_j} - \sum_{t \in \II_{v,\CC_j}^+} C_{v,t} \right)_+,
\end{equation}
and compared to the before-regulation excess count in the overflow window:
\begin{equation}
		\operatorname{InLoad\_Bef}(\CC_j) = \sum_v \left(\sum_{t \in \II_{v,\CC_j}^+} D_{v,t} - \sum_{t \in \II_{v,\CC_j}^+} C_{v,t} \right)_+.
\end{equation}

The InLoad heuristics is defined as:
\begin{equation}
	\operatorname{InLoad} = \operatorname{InLoad\_Bef} - \operatorname{InLoad\_Aft}.
\end{equation}

\subsubsection{Empirical Study of the Heuristics Effectiveness}
To quantify the effectiveness of the proposed heuristics, we compare the heuristic values against the general reward function $\Delta J = J_{\text{pre\_regulation}} - J_{\textit{post\_regulation}}$, as well as the same reward function scoped only to the hot cells $\operatorname{HC}$ for NomRel, \emph{after adjusting for the optimal rate that yields the highest objective value}, for 5,620 flows extracted from randomly selected hotspots from the initial (pre-regulated) data.

\begin{figure}[htbp]
	\centering
	\begin{subfigure}[b]{0.48\textwidth}
		\centering
		\includegraphics[width=\textwidth]{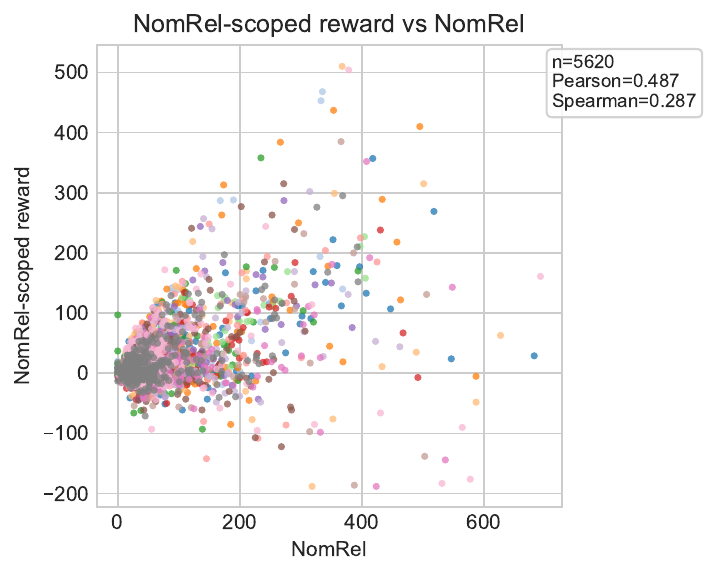}
		\caption{NomRel vs Rate-Optimal Excess Relief scoped to Hot Cells}
		\label{fig:nomrel_v_nomrel_scoped}
	\end{subfigure}
	\hfill
	\begin{subfigure}[b]{0.48\textwidth}
		\centering
		\includegraphics[width=\textwidth]{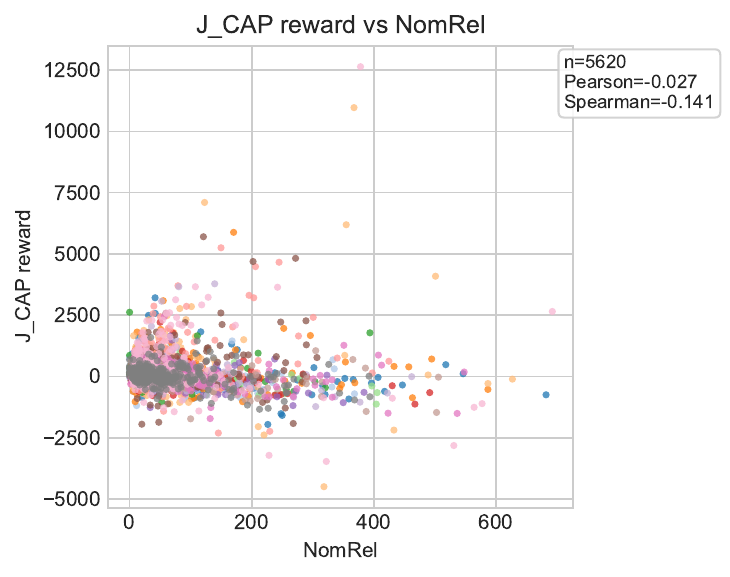}
		\caption{NomRel vs Actual Rate-Optimal Network-wide Excess Relief}
		\label{fig:nomrel_v_reward}
	\end{subfigure}
	\caption{NomRel's flow selection power adjusted for the optimal entry rate.}
	\label{fig:nomrel_performance}
\end{figure}

From Figure \ref{fig:nomrel_v_nomrel_scoped}, it could be seen that the entry rate setting does not weaken the correlation between NomRel and $J_{CAP}$ within hot cells too much. However, this relationship largely disappears when $J_{CAP}$ is evaluated over all traffic volumes and time bins, where the correlation becomes almost non-existent.

\begin{figure}
	\centering
	\includegraphics[width=0.5\linewidth]{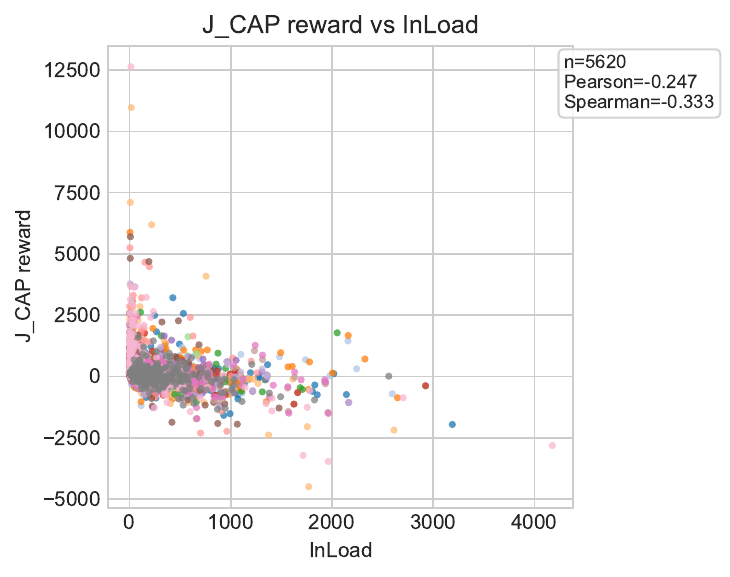}
	\caption{InLoad vs Actual Rate-Optimal Network-wide Excess Relief}
	\label{fig:jcaprewardinload}
\end{figure}

Figure \ref{fig:jcaprewardinload} shows a similar pattern: the overall correlation between rate-optimal $J_{CAP}$ and InLoad remains weak. Taken together, the NomRel and InLoad results indicate that entry rate capping heavily influences the final outcome, and that rate setting cannot be easily separated from flow selection cleanly. Nevertheless, the figures also suggest that there exist specific regimes in which these heuristics can be exploited to support flow selection within the Regulation Zero framework.

\begin{figure}[htbp]
	\centering
	\begin{subfigure}[b]{0.48\textwidth}
		\centering
		\includegraphics[width=\textwidth]{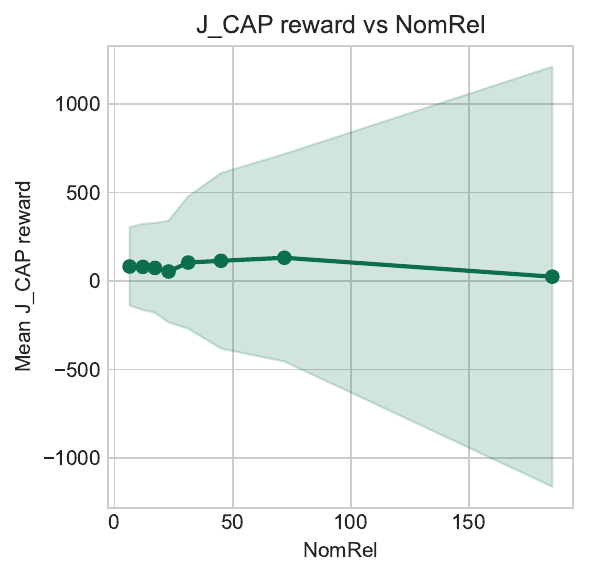}
		\caption{Histogram of $J_{CAP}$ grouped by NomRel}
		\label{fig:bintrend_jcap_nomrel}
	\end{subfigure}
	\hfill
	\begin{subfigure}[b]{0.48\textwidth}
		\centering
		\includegraphics[width=\textwidth]{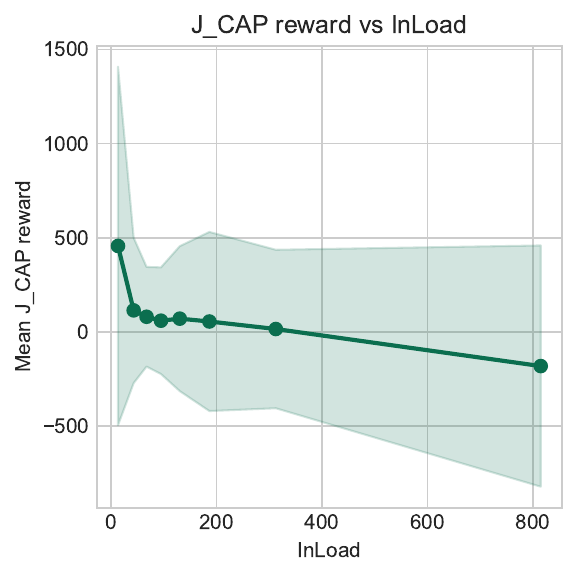}
		\caption{Histogram of $J_{CAP}$ grouped by InLoad}
		\label{fig:bintrend_jcap_inload}
	\end{subfigure}
	\caption{Bin Trend Plots of NomRel and InLoad against $J_{CAP}$.}
	\label{fig:bin_trend_plots}
\end{figure}

Figure \ref{fig:bin_trend_plots} shows the general trend of increasing variance of $J_{CAP}$ as both NomRel and InLoad increase. This trend is driven mostly by the number of flights in the flow. Intuitively, the more flights there are in a flow, the larger the NomRel score, and the greater the potential excess relief and induced load that shifting the flow could have on the network. However, it can be seen that for NomRel, there exists a ``sweet spot'' between 25 and 60, where the variance on the positive half of $J_{CAP}$ increases much more rapidly than the general trend. For InLoad, the low-InLoad regime, where the InLoad value is less than 80, has the best chance of yielding a good $J_{CAP}$ score after rate tuning.

These findings can be explained as follows: a good flow candidate is one that is not too large in size; otherwise, it can significantly increase the risk of inducing secondary overload. The best candidates usually have ample capacity in the immediate time bins following the hotspots (i.e., the ``valleys'' following the ``peaks'' in the demand histogram). However, it is usually necessary to test several candidates that appear similar in terms of heuristics, as $J_{CAP}$ cannot be reliably predicted from the heuristics alone.

\subsection{The Regulation Zero Framework}

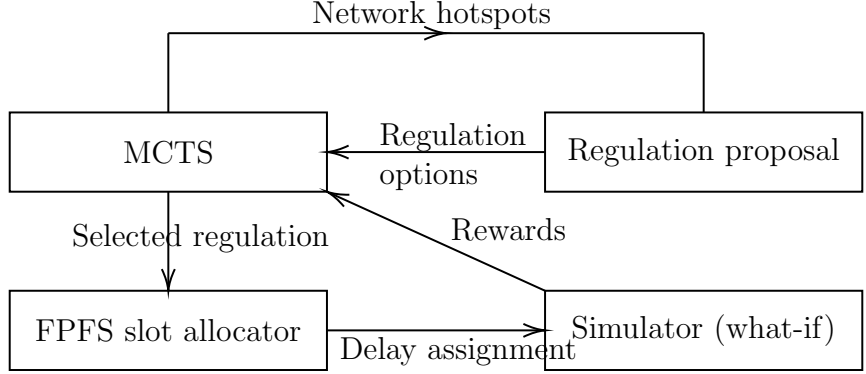
\begin{figure}
  \centering\begin{tikzpicture}[x=0.75pt,y=0.75pt,yscale=-1,xscale=1]

\draw   (340,70) -- (500,70) -- (500,110) -- (340,110) -- cycle ;
\draw   (70,70) -- (230,70) -- (230,110) -- (70,110) -- cycle ;
\draw   (340,160) -- (500,160) -- (500,200) -- (340,200) -- cycle ;
\draw   (70,160) -- (230,160) -- (230,200) -- (70,200) -- cycle ;
\draw    (340,90) -- (232,90) ;
\draw [shift={(230,90)}, rotate = 360] [color={rgb, 255:red, 0; green, 0; blue, 0 }  ][line width=0.75]    (10.93,-3.29) .. controls (6.95,-1.4) and (3.31,-0.3) .. (0,0) .. controls (3.31,0.3) and (6.95,1.4) .. (10.93,3.29)   ;
\draw    (150,110) -- (150,158) ;
\draw [shift={(150,160)}, rotate = 270] [color={rgb, 255:red, 0; green, 0; blue, 0 }  ][line width=0.75]    (10.93,-3.29) .. controls (6.95,-1.4) and (3.31,-0.3) .. (0,0) .. controls (3.31,0.3) and (6.95,1.4) .. (10.93,3.29)   ;
\draw    (230,180) -- (338,180) ;
\draw [shift={(340,180)}, rotate = 180] [color={rgb, 255:red, 0; green, 0; blue, 0 }  ][line width=0.75]    (10.93,-3.29) .. controls (6.95,-1.4) and (3.31,-0.3) .. (0,0) .. controls (3.31,0.3) and (6.95,1.4) .. (10.93,3.29)   ;
\draw    (340,160) -- (231.82,110.83) ;
\draw [shift={(230,110)}, rotate = 24.44] [color={rgb, 255:red, 0; green, 0; blue, 0 }  ][line width=0.75]    (10.93,-3.29) .. controls (6.95,-1.4) and (3.31,-0.3) .. (0,0) .. controls (3.31,0.3) and (6.95,1.4) .. (10.93,3.29)   ;
\draw    (150,30) -- (288,30) ;
\draw [shift={(290,30)}, rotate = 180] [color={rgb, 255:red, 0; green, 0; blue, 0 }  ][line width=0.75]    (10.93,-3.29) .. controls (6.95,-1.4) and (3.31,-0.3) .. (0,0) .. controls (3.31,0.3) and (6.95,1.4) .. (10.93,3.29)   ;
\draw    (150,30) -- (150,70) ;
\draw    (420,30) -- (420,70) ;
\draw    (290,30) -- (420,30) ;

\draw (420,90) node   [align=left] {\begin{minipage}[lt]{108.8pt}\setlength\topsep{0pt}
\begin{center}
Regulation proposal
\end{center}

\end{minipage}};
\draw (150,90) node   [align=left] {\begin{minipage}[lt]{108.8pt}\setlength\topsep{0pt}
\begin{center}
MCTS
\end{center}

\end{minipage}};
\draw (420,180) node   [align=left] {\begin{minipage}[lt]{108.8pt}\setlength\topsep{0pt}
\begin{center}
Simulator (what-if)
\end{center}

\end{minipage}};
\draw (150,180) node  [font=\normalsize] [align=left] {\begin{minipage}[lt]{108.8pt}\setlength\topsep{0pt}
\begin{center}
FPFS slot allocator
\end{center}

\end{minipage}};
\draw (255,74) node [anchor=north west][inner sep=0.75pt]   [align=left] {Regulation \\ options};
\draw (100,125) node [anchor=north west][inner sep=0.75pt]   [align=left] {Selected regulation};
\draw (235,182) node [anchor=north west][inner sep=0.75pt]   [align=left] {Delay assignment};
\draw (291,122) node [anchor=north west][inner sep=0.75pt]   [align=left] {Rewards};
\draw (221,12) node [anchor=north west][inner sep=0.75pt]   [align=left] {Network hotspots};

\end{tikzpicture}
  \caption{Regulation Zero Key Components}
  \label{fig:rz_key_components}
\end{figure}

Figure~\ref{fig:rz_key_components} outlines the key components of the Regulation Zero framework. The MCTS maintains a search tree of partial regulation plans and their associated rewards, defined as the objective-improvement score $\Delta J = J_{\text{pre\_regulation}} - J_{\text{post\_regulation}}$.
In each simulation, the search begins from an empty partial plan $s_0$ and expands by selecting a hotspot where a new regulation may be applied. The Regulation Proposal engine then takes the list of flights contributing to the hotspot, groups related flights into flows based on their impact profiles, and evaluates a set of candidate entry rate reductions using the FPFS Slot Allocator (Section \ref{subsec:fpfs_slot}). The most promising proposals are then added to the tree, and the process restarts at a new hotspot selected by the MCTS.
\subsubsection{Regulation Proposal Engine}\label{subsec:proposal_engine}
Given a hotspot $h \equiv (v_h, t_{i,h}, t_{f,h})$ selected by MCTS from the current partial plan $s$, the proposal engine generates a small, high-quality set of candidate regulations $\rho=(v_c,t_{i,c},t_{f,c},\tau_c, \kappa_c)$ with $v_c=v_h$ to expand the search tree. The engine proceeds in three steps: (i) extracting flows relevant to $h$ by detecting communities in a flight-similarity graph derived from their ``impact footprints,'' (ii) testing multiple entry-rate reduction values, and (iii) evaluating the immediate reward of applying each regulation.

\paragraph{Hotspot-relevant flight set}
We define the flight set associated with a hotspot $h$ in Eq.~(\ref{eq:regulated_flight_list}). By construction, $\mathcal{F}_h$ includes all flights whose planned entries into $v_h$ contribute to at least one overloaded rolling-hour bin in $[t_{i,h},t_{f,h}]_{\mathbb{Z}}$.

\paragraph{Flight Impact Similarity}
Given $f,g \in \mathcal{F}_h$, we define their \textbf{demand impact similarity} as the Jaccard similarity between their footprints:
\begin{equation}\label{eq:jaccard}
  s_h(f,g) \;\equiv\; \frac{\bigl| \operatorname{Foot}(f) \cap \operatorname{Foot}(g) \bigr|}{\bigl| \operatorname{Foot}(f) \cup \operatorname{Foot}(g) \bigr|} \;\;\in\; [0,1],
\end{equation}
with the convention $s_h(f,g)=0$ if both sets are empty. Intuitively, $s_h$ is large for flights that cross similar volume set and small otherwise, making it a proxy for whether a single regulation can coherently meter both flights.

Figure~\ref{fig:jaccardexample} illustrates this computation. For two flights contributing to hotspot EGLMU during a given time segment, we retrieve their TV footprints and count the shared TVs. The Jaccard similarity is then the ratio of the intersection cardinality to the union cardinality, which in this example equals $1/3$.

\begin{figure}
  \centering
  \begin{tikzpicture}[x=0.75pt,y=0.75pt,yscale=-1,xscale=1]
	
	\draw   (210,60) -- (280,60) -- (280,80) -- (210,80) -- cycle ;
	\draw   (290,60) -- (360,60) -- (360,80) -- (290,80) -- cycle ;
	\draw   (370,60) -- (440,60) -- (440,80) -- (370,80) -- cycle ;
	\draw   (450,60) -- (520,60) -- (520,80) -- (450,80) -- cycle ;
	\draw   (210,90) -- (280,90) -- (280,110) -- (210,110) -- cycle ;
	\draw   (290,90) -- (360,90) -- (360,110) -- (290,110) -- cycle ;
	\draw   (370,90) -- (440,90) -- (440,110) -- (370,110) -- cycle ;
	\draw   (450,90) -- (520,90) -- (520,110) -- (450,110) -- cycle ;
	
	\draw (245,70) node  [font=\scriptsize] [align=left] {\begin{minipage}[lt]{47.6pt}\setlength\topsep{0pt}
			\begin{center}
				LFMWMFDZ
			\end{center}
			
	\end{minipage}};
	\draw (325,70) node  [font=\scriptsize,color={rgb, 255:red, 255; green, 0; blue, 0 }  ,opacity=1 ] [align=left] {\begin{minipage}[lt]{47.6pt}\setlength\topsep{0pt}
			\begin{center}
				EGLMU
			\end{center}
			
	\end{minipage}};
	\draw (405,70) node  [font=\scriptsize] [align=left] {\begin{minipage}[lt]{47.6pt}\setlength\topsep{0pt}
			\begin{center}
				LFMAIET
			\end{center}
			
	\end{minipage}};
	\draw (485,70) node  [font=\scriptsize] [align=left] {\begin{minipage}[lt]{47.6pt}\setlength\topsep{0pt}
			\begin{center}
				LFMMFDZ
			\end{center}
			
	\end{minipage}};
	\draw (245,100) node  [font=\scriptsize] [align=left] {\begin{minipage}[lt]{47.6pt}\setlength\topsep{0pt}
			\begin{center}
				LFMWMFDZ
			\end{center}
			
	\end{minipage}};
	\draw (325,100) node  [font=\scriptsize,color={rgb, 255:red, 255; green, 0; blue, 0 }  ,opacity=1 ] [align=left] {\begin{minipage}[lt]{47.6pt}\setlength\topsep{0pt}
			\begin{center}
				EGLMU
			\end{center}
			
	\end{minipage}};
	\draw (405,100) node  [font=\scriptsize] [align=left] {\begin{minipage}[lt]{47.6pt}\setlength\topsep{0pt}
			\begin{center}
				LFMRAES
			\end{center}
			
	\end{minipage}};
	\draw (485,100) node  [font=\scriptsize] [align=left] {\begin{minipage}[lt]{47.6pt}\setlength\topsep{0pt}
			\begin{center}
				LFFOPKZ
			\end{center}
			
	\end{minipage}};
	\draw (365,150) node   [align=left] {\begin{minipage}[lt]{346.8pt}\setlength\topsep{0pt}
			$\displaystyle {\textstyle d_{h}( f,g) =\frac{|\{LFMWMFDZ,\ EGLMU\} |}{|\{LFMWMFDZ,\ EGLMU,LFMAIET,LFMMFDZ,LFMRAES,LFFOPKZ\}} =\frac{1}{3}}$
	\end{minipage}};
	\draw (176,62) node [anchor=north west][inner sep=0.75pt]   [align=left] {$\displaystyle f$};
	\draw (176,90) node [anchor=north west][inner sep=0.75pt]   [align=left] {$\displaystyle g$};

\end{tikzpicture}
  \caption{Demand-impact similarity calculation for two flights $f, g$ given the hotspot $h$ EGLMU.}
  \label{fig:jaccardexample}
\end{figure}

\paragraph{Graph construction and community detection for flow extraction}\label{subsub:flow_x}
We construct a flight footprint similiarity graph
$$
G_h \;=\; (\VV_h^{\mathrm{graph}},\, \mathcal{E}_h,\, w),
\qquad
\VV_h^{\mathrm{graph}} \equiv \mathcal{F}_h,
$$
with binary edge weights
$$
w_{fg} \;=\;
\begin{cases}
  1 \quad\text{if $s_h(f,g) \ge \tau_{\min}$,}\\
  0 \quad\text{otherwise.}
\end{cases}
$$
That is, we retain an edge between $f$ and $g$ only when their footprints overlap sufficiently. As $\tau_{\min}$ increases, the extracted flows become smaller. In practice, we find that values of $\tau_{\min}$ between 0.6 and 0.8 tend to yield the best results: large flows are more likely to induce secondary hotspots elsewhere in the network, whereas small flows risk increasing the number of decision variables.

Leiden algorithm then partitions $\mathcal{F}_h$ into communities $\mathcal{P}_h=\{\CC_1,\ldots,\CC_R\}$, which iteratively refines the partition to maximize modularity \cite{traag2019louvain}. From (\ref{eq:flow_sel_op}), let $\kappa_{\mathcal{C}}:\mathcal{F}\to\{0,1\}$ denote the flow selection operator for community $\mathcal{C}$, i.e.,
$$
\kappa_{\mathcal{C}}(f) \equiv \mathbf{1}[f\in\mathcal{C}],
$$
and thus we obtain a flow candidate
$$
\rho \equiv (v_h, t_i, t_f, \cdot, \kappa_{\mathcal{C}}).
$$
The remaining component still required to form a concrete regulation proposal is the new entry rate $\tau$, which we will detail its selection next.

\begin{figure}
  \centering
  \includegraphics[width=.75\linewidth]{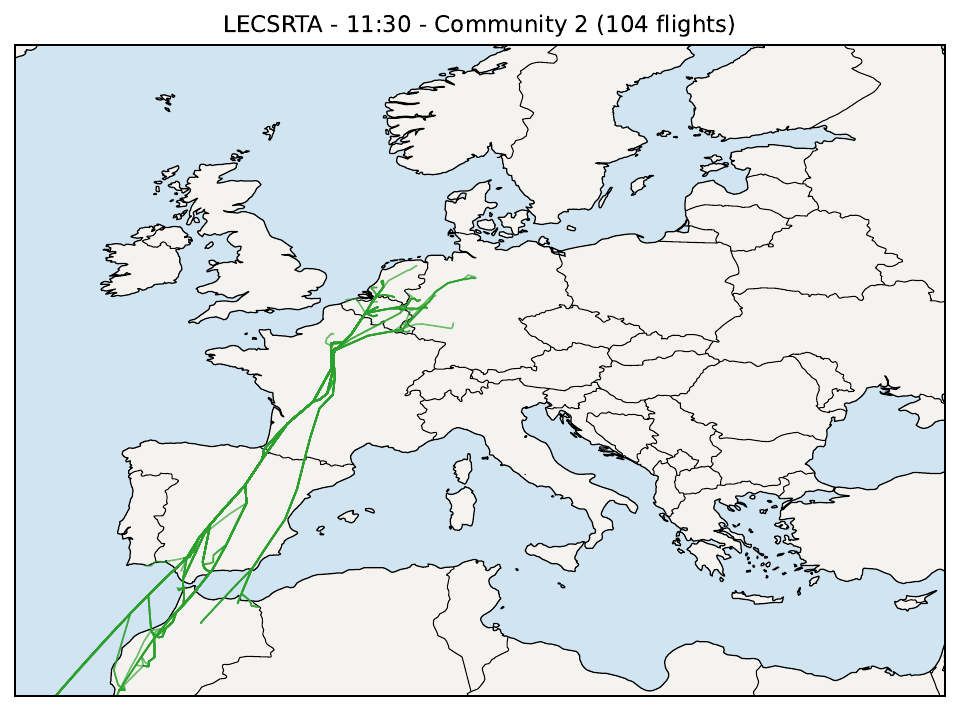}
  \caption{Sample flow extracted by the community detection algorithm for hotspot segment LECSRTA (11:30--12:15, 23/07/2023), as a candidate for regulation. The flow includes flights traversing Madrid and Paris airspace.} \label{fig:flow_x_sample}
\end{figure}

\paragraph{Flow and Rate Selection}
Because evaluating any regulation requires running the slot allocator (Section~\ref{subsec:fpfs_slot}), which is computationally expensive, we restrict our attention to  only a small set of candidate flows for entry-rate tuning and subsequent evaluation. The goal of the search is to find up to $K$ proposals $\rho = (v_h, t_{i,h}, t_{f,h}, \tau, \kappa)$, each evaluated using the FPFS slot allocator on the full traffic set and scored by its incremental improvement $\Delta J = J(s) - J(s \cup \rho)$.

To achieve this, we rely heavily on the heuristics derived in Section \ref{subsubsec:nomrel_inload} to identify promising flows. For brevity, we consider only proposals involving a single flow (referred to as a \emph{singleton proposal}), although the extension to multi-flow proposals should be straightforward.

\paragraph{Inputs and Hyperparameters}

\begin{itemize}
  \item Flow set $\mathcal{P}_h = \{\CC_1, \ldots, \CC_R\}$: the set of candidate flows extracted for hotspot $h$.
  \item max\_flows\_in\_regulation $(M_{\max})$:  maximum number of flows to be evaluated by the FCFS slot allocator.
  \item top-$k$ proposals $(k_{\mathrm{top}})$: number of best-performing regulation proposals returned to the MCTS.
  \item min\_flights\_per\_flow $(N_{\min})$: minimum number of flights a flow must contain to be considered eligible.
  \item Grid of rate factors $\Gamma$: discrete set of candidate rate multipliers used for grid search, e.g.\ $\Gamma = \{0.6, 0.7, 0.8, 0.9 \dots 1.2\}$.
\end{itemize}

\paragraph{Rate initialization}
Let $C_{\min}(h)\equiv \min_{t\in I_h^{\mathrm{eff}}} C_{v_h,t}$ be a conservative per-hour capacity bound over the effective window. For a set of flows $S\subseteq \mathcal{P}_h$, define its (weighted) share of the total demand in $I_h^{\mathrm{eff}}$ as
\[
  p(S\,|\,s)\;\equiv\;
  \frac{\sum_{t\in I_h^{\mathrm{eff}}} \omega_t\, D_{v_h,t}(S)}
  {\sum_{t\in I_h^{\mathrm{eff}}} \omega_t\, D_{v_h,t}^{(s)}},
  \qquad
  \omega_t\;\equiv\;\bigl(D_{v_h,t}^{(s)}-C_{v_h,t}\bigr)_+ + \varepsilon,
\]
where $\varepsilon>0$ avoids division by zero. The weights $\omega_t$ emphasize time bins with positive overload.

We initialize the nominal rate for $S$ as the conservative capacity multiplied by this demand share:
\[
  \tau_0(S)\;\equiv\;\mathrm{min}\!\bigg(C_{\min}(h)\, p(S\,|\,s), D_{v_h} (S) \bigg),
\]
where $D_{v_h} (S)$ is the flow's demand contribution normalized to the hourly time scale. The set of candidate rates explored during grid search is given by multiplicative perturbations around $\tau_0(S)$:
\[
  \mathcal{T}(S)\;=\;
  \bigl\{\, \tau_0(S)\cdot(1+\gamma)\;:\;\gamma\in \Gamma\,\bigr\}.
\]

The search explores unions of the top-$r$ flows (by $\phi$) for $r=1,\ldots,R_{\max}$ and, for each union, evaluates candidate rates in $\mathcal{T}(S)$. Each candidate regulation is simulated via FPFS on the full traffic set and scored by $\Delta J$.

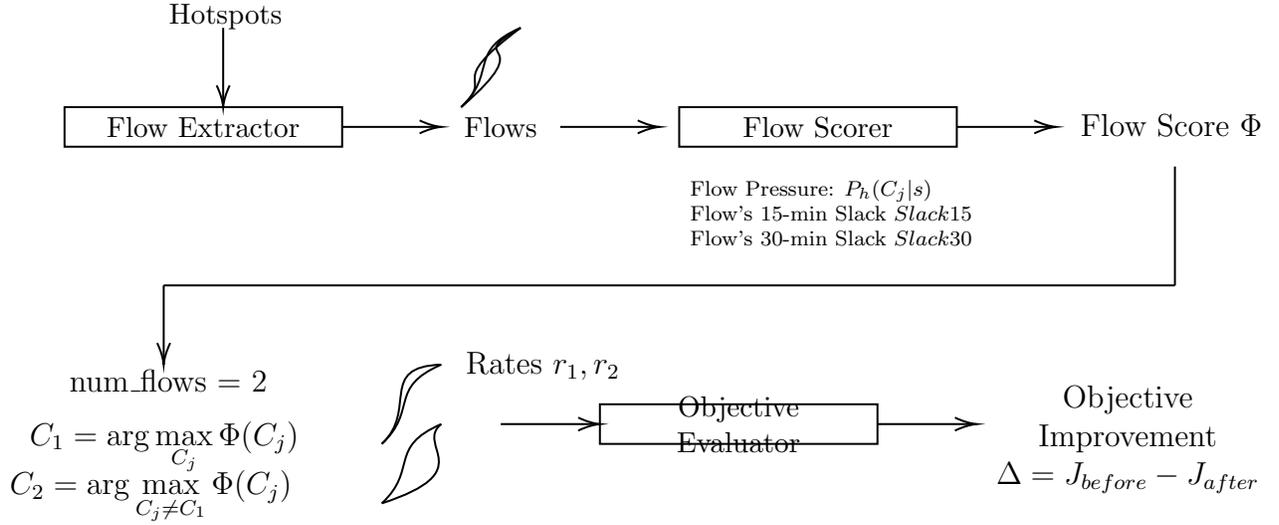
\begin{figure}
  \centering
  \begin{tikzpicture}[x=0.75pt,y=0.75pt,yscale=-1,xscale=1]
	
	\draw   (40,110) -- (180,110) -- (180,130) -- (40,130) -- cycle ;
	\draw    (180,120) -- (228,120) ;
	\draw [shift={(230,120)}, rotate = 180] [color={rgb, 255:red, 0; green, 0; blue, 0 }  ][line width=0.75]    (10.93,-3.29) .. controls (6.95,-1.4) and (3.31,-0.3) .. (0,0) .. controls (3.31,0.3) and (6.95,1.4) .. (10.93,3.29)   ;
	\draw    (120,70) -- (120,108) ;
	\draw [shift={(120,110)}, rotate = 270] [color={rgb, 255:red, 0; green, 0; blue, 0 }  ][line width=0.75]    (10.93,-3.29) .. controls (6.95,-1.4) and (3.31,-0.3) .. (0,0) .. controls (3.31,0.3) and (6.95,1.4) .. (10.93,3.29)   ;
	\draw   (350,110) -- (490,110) -- (490,130) -- (350,130) -- cycle ;
	\draw    (290,120) -- (338,120) ;
	\draw [shift={(340,120)}, rotate = 180] [color={rgb, 255:red, 0; green, 0; blue, 0 }  ][line width=0.75]    (10.93,-3.29) .. controls (6.95,-1.4) and (3.31,-0.3) .. (0,0) .. controls (3.31,0.3) and (6.95,1.4) .. (10.93,3.29)   ;
	\draw    (490,120) -- (538,120) ;
	\draw [shift={(540,120)}, rotate = 180] [color={rgb, 255:red, 0; green, 0; blue, 0 }  ][line width=0.75]    (10.93,-3.29) .. controls (6.95,-1.4) and (3.31,-0.3) .. (0,0) .. controls (3.31,0.3) and (6.95,1.4) .. (10.93,3.29)   ;
	\draw    (260,270) -- (308,270) ;
	\draw [shift={(310,270)}, rotate = 180] [color={rgb, 255:red, 0; green, 0; blue, 0 }  ][line width=0.75]    (10.93,-3.29) .. controls (6.95,-1.4) and (3.31,-0.3) .. (0,0) .. controls (3.31,0.3) and (6.95,1.4) .. (10.93,3.29)   ;
	\draw   (310,260) -- (450,260) -- (450,280) -- (310,280) -- cycle ;
	\draw    (450,270) -- (498,270) ;
	\draw [shift={(500,270)}, rotate = 180] [color={rgb, 255:red, 0; green, 0; blue, 0 }  ][line width=0.75]    (10.93,-3.29) .. controls (6.95,-1.4) and (3.31,-0.3) .. (0,0) .. controls (3.31,0.3) and (6.95,1.4) .. (10.93,3.29)   ;
	\draw    (600,140) -- (600,200) ;
	\draw    (90,200) -- (600,200) ;
	\draw    (90,200) -- (90,238) ;
	\draw [shift={(90,240)}, rotate = 270] [color={rgb, 255:red, 0; green, 0; blue, 0 }  ][line width=0.75]    (10.93,-3.29) .. controls (6.95,-1.4) and (3.31,-0.3) .. (0,0) .. controls (3.31,0.3) and (6.95,1.4) .. (10.93,3.29)   ;
	\draw    (240,110) .. controls (258.33,93) and (230,100) .. (270,70) ;
	\draw    (240,110) .. controls (258.33,93) and (243,81.67) .. (270,70) ;
	\draw    (240,110) .. controls (281.67,74.33) and (239,87.67) .. (270,70) ;
	\draw    (200,280) .. controls (218.33,263) and (203,251.67) .. (230,240) ;
	\draw    (200,310) .. controls (250.33,293.67) and (216.33,281) .. (230,270) ;
	\draw    (200,280) .. controls (218.33,263) and (197.67,241) .. (230,240) ;
	\draw    (200,310) .. controls (218.33,293) and (203,281.67) .. (230,270) ;
	
	\draw (110,120) node  [font=\small] [align=left] {\begin{minipage}[lt]{95.2pt}\setlength\topsep{0pt}
			\begin{center}
				Flow Extractor
			\end{center}
			
	\end{minipage}};
	\draw (260,120) node  [font=\small] [align=left] {\begin{minipage}[lt]{40.8pt}\setlength\topsep{0pt}
			\begin{center}
				Flows
			\end{center}
			
	\end{minipage}};
	\draw (120,60) node  [font=\small] [align=left] {\begin{minipage}[lt]{40.8pt}\setlength\topsep{0pt}
			\begin{center}
				Hotspots
			\end{center}
			
	\end{minipage}};
	\draw (420,120) node  [font=\small] [align=left] {\begin{minipage}[lt]{95.2pt}\setlength\topsep{0pt}
			\begin{center}
				Flow Scorer
			\end{center}
			
	\end{minipage}};
	\draw (354,145) node [anchor=north west][inner sep=0.75pt]  [font=\scriptsize] [align=left] {Flow Pressure: $\displaystyle P_{h}( C_{j} |s)$\\Flow's 15-min Slack $\displaystyle Slack15$\\Flow's 30-min Slack $\displaystyle Slack30$};
	\draw (551,112) node [anchor=north west][inner sep=0.75pt]   [align=left] {Flow Score $\displaystyle \Phi $};
	\draw (41,241) node [anchor=north west][inner sep=0.75pt]   [align=left] {num\_flows = 2};
	\draw (21,267) node [anchor=north west][inner sep=0.75pt]   [align=left] {$\displaystyle C_{1} =\arg\max_{C_{j}} \Phi ( C_{j})$};
	\draw (11,292) node [anchor=north west][inner sep=0.75pt]   [align=left] {$\displaystyle C_{2} =\arg\max_{C_{j} \neq C_{1}} \Phi ( C_{j})$};
	\draw (380,270) node  [font=\small] [align=left] {\begin{minipage}[lt]{95.2pt}\setlength\topsep{0pt}
			\begin{center}
				Objective Evaluator
			\end{center}
			
	\end{minipage}};
	\draw (503,249) node [anchor=north west][inner sep=0.75pt]   [align=left] {\begin{minipage}[lt]{107.6pt}\setlength\topsep{0pt}
			\begin{center}
				Objective Improvement\\$\displaystyle \Delta =J_{before} -J_{after}$
			\end{center}
			
	\end{minipage}};
	\draw (241,232) node [anchor=north west][inner sep=0.75pt]   [align=left] {Rates $\displaystyle r_{1} ,r_{2}$};

\end{tikzpicture}
  \caption{The Regulation Proposal Engine's Main Components, with \texttt{num\_flows = 2}.}
  \label{fig:regen}
\end{figure}

Figure \ref{fig:regen} provides a high-level view of the entire data-flows within the Regulation Proposal engine employed by Regulation Zero.
%


\subsection{Hierarchical Monte-Carlo Tree Search}
\label{subsec:regzero-mcts}
\subsubsection{Flow Regulations as a Markov Decision Process Problem}
As discussed earlier, it is essential to solve the traffic regulations in sequence. This motivates us to formulate the regulation search as planning in a finite-horizon Markov decision process (MDP)
\(
  \mathcal{M} = (\mathcal{S}, \mathcal{A}, P, R, \gamma)
\),
with state space \(\mathcal{S}\), \emph{joint} action space
\(\mathcal{A} = \mathcal{R}\times\mathcal{H}\) where \(a=(r,h)\) combines a regulation proposal \(r\in\mathcal{R}\) and a hotspot \(h\in\mathcal{H}\),
transition kernel \(P(\cdot \mid s,a)\),
reward \(R(s,a) = \Delta J(s,a)\) given by the objective improvement
(weighted reduction of excess counts and delay),
and a discount factor \(\gamma \in (0,1]\) for numerical stability.
For an episode \(s_0,s_1,\dots,s_T\), the objective is to maximize the expected discounted return
\begin{equation}
  J^\pi(s_0) \;=\; \mathbb{E}_\pi\!\left[ \sum_{t=0}^{T-1} \gamma^t \,\Delta J_t \right],
  \quad \text{with } s_{t+1}\sim P(\cdot \mid s_t,a_t), \; a_t\sim \pi(\cdot\mid s_t).
  \label{eq:objective}
\end{equation}

Because applying regulations results in a deterministic outcome:
\begin{equation*}
  P(s_{t+1} | s_t, a_t) = \delta_{\{s = s_{t+1}\}}.
\end{equation*}

\begin{figure}
  \centering
  \input{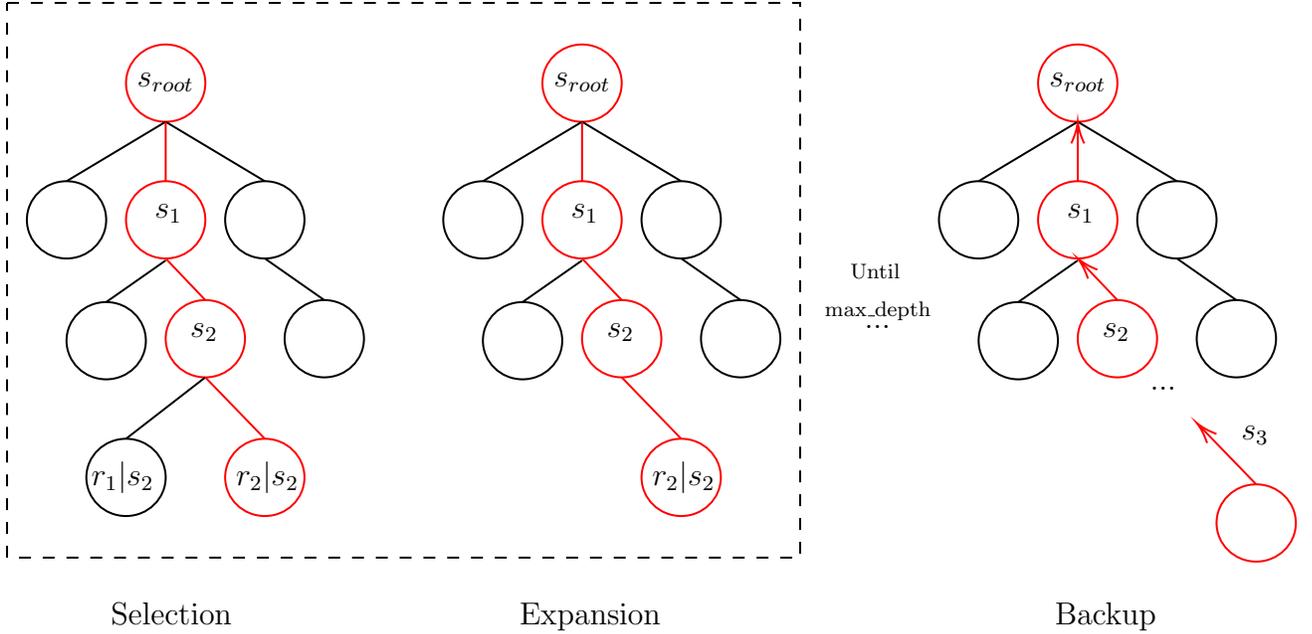}
  \caption{Illustration of the three main phases of the MCTS algorithm employed by Regulation Zero.}
  \label{fig:mcts}
\end{figure}

The key steps of our MCTS implementation version can be described in Figure \ref{fig:mcts}.

\subsubsection{Hierarchical action and priors.}
Because the flow extraction result $\mathcal{P}_h$ is tied to the selected hotspot segment $h$, we factor the search prior policy as
\[
  \pi(a\mid s) \;=\; \pi_H(h\mid s)\,\pi_R(r\mid h,s),
\]
and assume that the priors $P_H(h\mid s)>0$ and $P_R(r\mid h,s)>0$.

Let $\mathcal{H}(s)$ be the set of hotspot segments. We define the severity of a hotspot as the cumulative exceedance across its segment under plan $s$:
\begin{equation}\label{eq:phi_h_def}
  \phi(h \mid s)
  \;\equiv\;
  \sum_{t \in [t_{i,h},\,t_{f,h}]_{\mathbb{Z}}}
  \Bigl( D_{v_h,t}^{(s)} - C_{v_h,t} \Bigr)_{+},
\end{equation}
where $(x)_{+} \equiv \max\{x,0\}$.

We sample hotspots according to a tempered softmax:
\begin{equation}\label{eq:PH}
  P_H(h \mid s)
  \;=\;
  \frac{\exp\!\bigl(\phi(h\mid s)/\tau_H\bigr)}
  {\displaystyle \sum_{h'\in \mathcal{H}(s)} \exp\!\bigl(\phi(h'\mid s)/\tau_H\bigr)}
\end{equation}

Conditioned on $h$, the Regulation Proposal Engine returns a finite candidate set
$\mathcal{R}(h,s)=\{r_1,\ldots,r_m\}$ together with predicted immediate improvements $\widehat{\Delta J}(s,h,r)$. Then, the regulation to be selected will follow the PUCT rule.

\subsubsection{PUCT-driven proposal selection}
Within the search tree, we select proposals using a PUCT criterion instantiated for the \((h,r)\) branching factorization:
\begin{equation}
  r^* \;=\; \arg\max_{r}
  \left[
    Q(s,h,r) \;+\; c_{\mathrm{PUCT}}\, P_R(r\mid h,s)\,
    \frac{\sqrt{\sum_b n(s,h,b)}}{1+n(s,h,r)}
  \right],
  \label{eq:puct_selection}
\end{equation}
where \(Q(s,h,r)\) is the current value estimate and \(n(s,h,r)\) counts visits to child \((s,h,r)\).

The prior \(P_R(r\mid h,s)\) guides the search process to explore new regulations in the beginning of the search phase, but gradually subsides, rendering the policy becoming more greedy as total visits $\sum_b n(s,h,b)$ increase.

After a regulation is selected and applied, the corresponding delays are imposed on the affected flights, and the flight list is updated accordingly. This update triggers a re-evaluation of active hotspots and a new flow extraction step. The reward obtained at iteration $k$ is defined as $\Delta J_k = \hat{\Delta J}(s, h, r^*)$. The algorithm then continues to extend the sequence until reaching a predefined search depth.

\subsubsection{Value backups}
At the end of the trajectory, the cumulative discounted reward is:
\begin{equation}
  G \!=\! \sum_{k} \gamma^{k-1}\Delta J_k.
\end{equation}

For each visited node on the path, we perform the following updates:
\begin{align}
  n(s,h,r) &\leftarrow n(s,h,r) + 1,\\
  W(s,h,r) &\leftarrow W(s,h,r) + G,\\
  Q(s,h,r) &\leftarrow \frac{W(s,h,r)}{n(s,h,r)}.
\end{align}

All the steps of the MCTS can be summarized in Algorithm \ref{alg:regzero}.

\begin{algorithm}[t]
  \caption{Regulation Zero: Hierarchical, KL-Regularized MCTS}
  \label{alg:regzero}
  \begin{algorithmic}[1]
    \State \textbf{Input:} priors \(P_H, P_R\); severity \(\phi\); discount \(\gamma\); \(c_{\mathrm{PUCT}}\); \texttt{max\_depth}
    \For{simulation $=1,\dots,N$}
    \State $s \leftarrow s_0$; path $\leftarrow [\;]$; $d \leftarrow 0$
    \While{$d < \texttt{max\_depth}$ and $s$ not terminal}
    \State \textbf{Hotspot:} sample $h \sim \pi_H(\cdot\mid s) \propto P_H(\cdot\mid s)\exp(\phi/\tau_H)$
    \If{some $(s,h, \cdot)$ unexpanded} \State retrieve/update $\{r(h)\}$ via \texttt{RegulationProposal} \EndIf
    \State \textbf{Proposal:} select $r$ by \eqref{eq:puct_selection}
    \State Append $(s,h,r)$ to path
    \State Execute $a=(h,r)$; observe $\Delta J$; sample $s' \sim P(\cdot\mid s,a)$
    \State $s \leftarrow s'$; $d \leftarrow d+1$
    \EndWhile
    \State Compute return $G \leftarrow \sum_{t=0}^{d-1} \gamma^t \Delta J_t$
    \For{each node $(s,h,r)$ in path} \State update $n,W,Q$ with $G$ \EndFor
    \EndFor
    \State \textbf{return} action at root with largest value
  \end{algorithmic}
\end{algorithm}

\subsection{Comments on Architecture Design}
In this section, we provide a rationale for the design of our framework. Conventional approaches in reinforcement learning (RL) studies often formulate the action space around atomic operations, such as adding or removing a regulation. However, our experiments indicate that this design frequently leads to cyclic exploration patterns, as the vast combinatorial space of possible regulatory modifications causes the policy to become trapped in repetitive loops. While one could impose a simulation budget to curb unbounded exploration, the long trajectory lengths inherent to such settings substantially weaken the gradient signal during backpropagation, thereby impeding effective learning.

The role of the regulation proposal (RP) is fundamental in ensuring the stability of the MCTS search process. One valid view of RP is that RP is the local-search extension of MCTS's exploration at each node. Not only would this help narrowing down the search space, but also renders the solution transparent to verification.

Following the analysis presented in Appendix~\ref{sec:poliopt}, \emph{Regulation Zero} optimizes a KL-regularized objective that promotes structured exploration at both the hotspot and regulation levels:
\begin{equation}
  \Omega_s(\pi) = \tau_h \, \mathrm{KL}\!\left[\pi_H(\cdot\,|\,s) \,\|\, P_H(\cdot\,|\,s)\right]
  + \mathbb{E}_{h \sim \pi_H}\!\left[\lambda_R (s,h) \, \mathrm{KL}\!\left[P_R(\cdot\,|\,h,s) \,\|\, \pi_R(\cdot\,|\,h,s)\right]\right].
\end{equation}

The first term, $\tau_h \, \mathrm{KL}[\pi_H(\cdot|s) \| P_H(\cdot|s)]$, represents a reverse KL-divergence akin to that used in the Soft Actor-Critic (SAC) formulation \cite{haarnoja2018soft}, favoring mode-seeking behavior that concentrates attention on the most critical hotspots $h$.

In contrast, the second term, $\mathrm{KL}[P_R(\cdot|h,s) \| \pi_R(\cdot|h,s)]$, takes the form of a forward KL-divergence, encouraging mode-covering exploration within the regulation space.

Conceptually, this separation guides the policy to focus relief efforts on the most severe hotspots, while simultaneously enabling a comprehensive and diverse search across the regulatory landscape. The empirical results discussed in the following section illustrate how these dual objectives manifest in practice.

\section{Results}\label{sec:results}
To validate the proposed framework, we conducted a comprehensive evaluation using planned flight data obtained from EUROCONTROL’s DDR2. Instead of relying on R-NEST for CASA, we reimplemented the Slot Allocator within our open-source ATFM research platform, \emph{Flow’s Kitchen} \footnote{Source code is available at \url{https://github.com/thinhhoang95/project-cortex}.}. Performance was benchmarked against Simulated Annealing (SA) \cite{lavandier2021selective} and the Non-dominated Sorting Genetic Algorithm II (NSGA-II) \cite{fadil2017multi} to assess both solution quality and computational efficiency. We also compared against the previous version of Regulation Zero \cite{hoang2026flocon}, which used a different set of heuristics to guide flow selection; and a naive greedy capping policy, in which regulations are created to cap entries directly at the capacity limit, proceeding from hotspots with the highest exceedances to those with the lowest. In a strict sense, the greedy strategy could be seen as the first step of \cite{Dalmau2022Optimal}, from which unnecessary regulations could then be cancelled to further improve network delays. Table \ref{tab:optimizer_family_class} summarizes the family each comparison baseline belongs to. 

In both SA and GA implementations, taking inspiration from \cite{delahaye2018simulated}, we prioritize selecting flights with more ``interactions,'' which in our case means selecting with probability proportional to the number of ``hot cells'' they contribute to. This design introduces bias towards exceedance optimization, which is in a strict sense, what NomRel attempts to measure in RZ.

\begin{table}[htbp]
	\centering
	\begin{tabular}{ll}
		\toprule
		Optimizer & Family Class \\
		\midrule
		Regulation Zero (RZ) v2--this paper & \textbf{Flow-centric} \\
		Regulation Zero (RZ) v1 & \textbf{Flow-centric} \\
		Greedy & \textbf{Flow-centric} \\
		Simulated Annealing (SA) & Flight-centric \\
		NSGA-II (GA) & Flight-centric \\
		\bottomrule
	\end{tabular}
	\caption{Optimizer family classes.}
	\label{tab:optimizer_family_class}
\end{table}

All experiments were conducted on a dedicated cloud instance equipped with an AMD EPYC 7702P processor (64 cores, of which 16 were used) and 64GB of RAM. Implementations were developed in Python 3.13 with Numba. Objective values were computed using weights $w_{CAP} = 10$ and $w_{DELAY} = 1$, i.e., one excess entry is equivalent to 10 minutes of delay. Sector opening schemes and dynamic airspace reconfiguration were enabled during objective evaluation.

We selected seven full-day pan-European traffic scenarios from the SESAR DeepFlow validation dataset: 16/07, 17/07, 18/07, 20/07, 29/07, 03/08, and 05/08/2023. The number of flights per day ranged from 22,429 (29/07) to 24,833 (20/07). All algorithms managed to complete within 120 to 150 minutes, consistent with a pre-tactical optimization horizon.

\subsection{Hyperparameters Selection}
\begin{table}[htbp]
  \centering 
  \caption{Hyperparameters used for Regulation Zero}
  \label{tab:hyperparams}
  \begin{tabular}{ll}
    \toprule
    \textbf{Parameter} & \textbf{Value} \\
    \midrule
    sims & 128 \\
    depth & 64 \\
    commit\_depth & 64 \\
    flows\_threshold & 0.72 \\
    flows\_resolution & 1.0 \\
    max\_hotspots\_per\_node & 20 \\
    k\_proposals\_per\_hotspot & 6 \\
    puct\_c & 64 \\
    gamma & 0.999998 \\
    regulation\_selection\_softmax\_temperature & 24.0 \\
    hotspot\_sampling\_temperature & 6.0 \\
    {max\_delay\_per\_flight\_min} & $120$ \\
    \bottomrule
  \end{tabular}
\end{table}

\begin{table}[htbp]
  \centering

  \caption{Hyperparameters used for SA.}
  \label{tab:hyperparameterssa}
  \begin{tabularx}{0.5\linewidth}{X c}
    \toprule
    \textbf{Parameter Name} & \textbf{Value} \\
    \midrule
    {iters} & $10,000$ \\
    {T0} (Initial Temperature) & $15.0$ \\
    {cooling} (Cooling Rate) & $0.999$ \\
    {T\_min} (Minimum Temperature) & $1 \times 10^{-9}$ \\
    {max\_delay\_per\_flight\_min} & $120$ \\
    \bottomrule
  \end{tabularx}
\end{table}

\begin{table}[htbp]
	\centering
	\caption{Hyperparameters used by NSGA-II}
	\label{tab:hyperparametersga}
	\begin{tabularx}{0.7\linewidth}{X c}
		\toprule
		\textbf{Parameter Name} & \textbf{Value} \\
		\midrule
		\texttt{population\_size} & $64$ \\
		\texttt{generations} & $80$ \\
		\texttt{p\_crossover} & $0.9$ \\
		\texttt{mutations\_per\_child} & $2$ \\
		\texttt{mutate\_existing\_prob} & $0.7$ \\
		\texttt{max\_delay\_per\_flight\_min} & $120$ \\
		\texttt{(delay) step\_choices} & $\{2, 3, 4, 5\}$ \\
		\texttt{allow\_negative\_moves} & \texttt{True} \\
		\texttt{init\_delayed\_flights\_min} & $1$ \\
		\texttt{init\_delayed\_flights\_max} & $8$ \\
		\bottomrule
	\end{tabularx}
\end{table}

The hyperparameters of Regulation Zero were set according to Table \ref{tab:hyperparams}, SA's are in Table \ref{tab:hyperparameterssa} and NSGA-II's are in Table \ref{tab:hyperparametersga}. These hyperparameters were refined from the automatically tuned values retrieved from Optuna with 15 run trials to give the best results for both algorithms.

\subsection{Key Performance Indices Comparison}
\subsubsection{RZ delivers strong overall performance}
\begin{table}[!htbp]
	\centering
	\begin{tabular}{llrrrr}
		\toprule
		Case & Algorithm & Obj.\ Improvement & Exceedance Reduced & Total Delay (min) & Flights Delayed \\
		\midrule
		\multirow{5}{*}{0308}
		& RZ v1    & \underline{5790.8}    & \underline{1421.0}   & 8419.0            & \underline{229} \\
		& RZ v2    & \textbf{5920.8}       & \textbf{1481.0}      & 8889.0            & 258 \\
		& SA       & 3808.0                & 564.0                & \underline{1832.0} & 445 \\
		& GA       & 699.0                 & 113.0                & \textbf{431.0}    & \textbf{109} \\
		& Greedy & -142734.0             & -7844.0              & 64294.0           & 1226 \\
		\midrule
		\multirow{5}{*}{0508}
		& RZ v1    & \underline{6060.8}    & \textbf{1626.0}      & 10199.0           & 237 \\
		& RZ v2    & \textbf{6212.8}       & \underline{1478.0}   & 8567.0            & \underline{232} \\
		& SA       & 3125.0                & 465.0                & \underline{1525.0} & 377 \\
		& GA       & 668.0                 & 102.0                & \textbf{352.0}    & \textbf{92} \\
		& Greedy & -1270368.0            & -43552.0             & 837084.0          & 2904 \\
		\midrule
		\multirow{5}{*}{1607}
		& RZ v1    & \underline{3801.9} & \underline{1137.0} & 7568.0             & \underline{224} \\
		& RZ v2    & \textbf{5056.8}    & \textbf{1390.0}    & 8843.0             & 273 \\
		& SA       & 3006.0             & 448.0              & \underline{1474.0} & 370 \\
		& GA       & 602.0              & 91.0               & \textbf{308.0}     & \textbf{80} \\
		& Mindless & -58402.0           & -1197.0            & 47071.0            & 1070 \\
		\midrule
		\multirow{5}{*}{1707}
		& RZ v1    & \underline{4437.8}    & \underline{1272.0}   & 8282.0            & \underline{251} \\
		& RZ v2    & \textbf{5022.8}       & \textbf{1311.0}      & 8087.0            & 261 \\
		& SA       & 3333.0                & 482.0                & \underline{1487.0} & 380 \\
		& GA       & 648.0                 & 90.0                 & \textbf{252.0}    & \textbf{54} \\
		& Greedy & -283689.0             & -4406.0              & 239629.0          & 2038 \\
		\midrule
		\multirow{5}{*}{1807}
		& RZ v1    & \underline{4131.9}    & \underline{1327.0}   & 9138.0            & 264 \\
		& RZ v2    & \textbf{4964.8}       & \textbf{1360.0}      & 8635.0            & \underline{246} \\
		& SA       & 3259.0                & 478.0                & \underline{1521.0} & 366 \\
		& GA       & 757.0                 & 115.0                & \textbf{393.0}    & \textbf{101} \\
		& Greedy & -112406.0             & -5552.0              & 56886.0           & 1204 \\
		\midrule
		\multirow{5}{*}{2007}
		& RZ v1    & \underline{3881.9}    & \underline{1097.0}   & 7088.0            & \underline{207} \\
		& RZ v2    & \textbf{4645.8}       & \textbf{1386.0}      & 9214.0            & 243 \\
		& SA       & 3189.0                & 465.0                & \underline{1461.0} & 377 \\
		& GA       & 719.0                 & 114.0                & \textbf{421.0}    & \textbf{100} \\
		& Greedy & -986501.0             & -27137.0             & 715131.0          & 2410 \\
		\midrule
		\multirow{5}{*}{2907}
		& RZ v1    & \underline{5954.8}    & \underline{1414.0}   & 8185.0            & \underline{229} \\
		& RZ v2    & \textbf{6385.8}       & \textbf{1754.0}      & 11154.0           & 307 \\
		& SA       & 4530.0                & 645.0                & \underline{1920.0} & 459 \\
		& GA       & 902.0                 & 129.0                & \textbf{388.0}    & \textbf{84} \\
		& Greedy & -800930.0             & -23012.0             & 570810.0          & 2236 \\
		\bottomrule
	\end{tabular}
	\caption{Performance comparison across cases and algorithms. Bold indicates the best result within each case and metric; underlining indicates the second-best result. This paper's algorithm is presented as RZ v2.}
	\label{tab:case_algorithm_results}
\end{table}

\begin{table}[!htbp]
	\centering
	\begin{tabular}{llrrrrr}
		\toprule
		Case & Algorithm & \makecell{Changed\\(TV,TB) Cells} & \makecell{Changed\\TVs} & \makecell{Over-cap\\Reductions} & \makecell{Under-cap\\Increases} & \makecell{Beneficial\\Pairs} \\
		
		\midrule
		\multirow{5}{*}{0308}
		& RZ v1    & 7778  & 732 & 1135 & 2135 & 3270 \\
		& RZ v2    & 8649  & 733 & 1329 & 2240 & 3569 \\
		& SA       & 5820  & 726 & 488  & 819  & 1307 \\
		& GA       & 1988  & 487 & 175  & 254  & 429  \\
		& Greedy & 12310 & 858 & 1048 & 2939 & 3987 \\
		\midrule
		\multirow{5}{*}{0508}
		& RZ v1    & 8183  & 720 & 1482 & 2365 & 3847 \\
		& RZ v2    & 8463  & 716 & 1309 & 2272 & 3581 \\
		& SA       & 4657  & 674 & 417  & 728  & 1145 \\
		& GA       & 1405  & 433 & 147  & 189  & 336  \\
		& Greedy & 33646 & 893 & 4061 & 6816 & 10877 \\
		\midrule
		\multirow{5}{*}{1607}
		& RZ v1    & 8127  & 774 & 1103 & 1900 & 3003 \\
		& RZ v2    & 8784  & 788 & 1325 & 2082 & 3407 \\
		& SA       & 4909  & 696 & 394  & 651  & 1045 \\
		& GA       & 1349  & 385 & 124  & 162  & 286  \\
		& Greedy & 15432 & 905 & 1380 & 3665 & 5045 \\
		\midrule
		\multirow{5}{*}{1707}
		& RZ v1    & 8766  & 759 & 1332 & 2226 & 3558 \\
		& RZ v2    & 8187  & 788 & 1167 & 2012 & 3179 \\
		& SA       & 4768  & 699 & 371  & 667  & 1038 \\
		& GA       & 906   & 314 & 106  & 148  & 254  \\
		& Greedy & 23651 & 932 & 2421 & 5624 & 8045 \\
		\midrule
		\multirow{5}{*}{1807}
		& RZ v1    & 9731  & 794 & 1236 & 2553 & 3789 \\
		& RZ v2    & 8359  & 747 & 1197 & 2136 & 3333 \\
		& SA       & 4594  & 702 & 389  & 664  & 1053 \\
		& GA       & 1784  & 452 & 147  & 210  & 357  \\
		& Greedy & 11760 & 917 & 657  & 2922 & 3579 \\
		\midrule
		\multirow{5}{*}{2007}
		& RZ v1    & 7311  & 746 & 1265 & 1857 & 3122 \\
		& RZ v2    & 8607  & 777 & 1392 & 2318 & 3710 \\
		& SA       & 4678  & 697 & 388  & 709  & 1097 \\
		& GA       & 1687  & 455 & 173  & 237  & 410  \\
		& Greedy & 28871 & 926 & 4242 & 4659 & 8901 \\
		\midrule
		\multirow{5}{*}{2907}
		& RZ v1    & 7514  & 707 & 1290 & 1934 & 3224 \\
		& RZ v2    & 10245 & 771 & 1489 & 2555 & 4044 \\
		& SA       & 5148  & 693 & 439  & 763  & 1202 \\
		& GA       & 1460  & 426 & 130  & 197  & 327  \\
		& Greedy & 26976 & 875 & 4014 & 4620 & 8634 \\
		\bottomrule
	\end{tabular}
	\caption{Comparison of changed cells, changed TVs, and capacity-related effects across cases and algorithms.}
	\label{tab:changed_pairs_results}
\end{table}

\begin{figure}
	\centering
	\includegraphics[width=\linewidth]{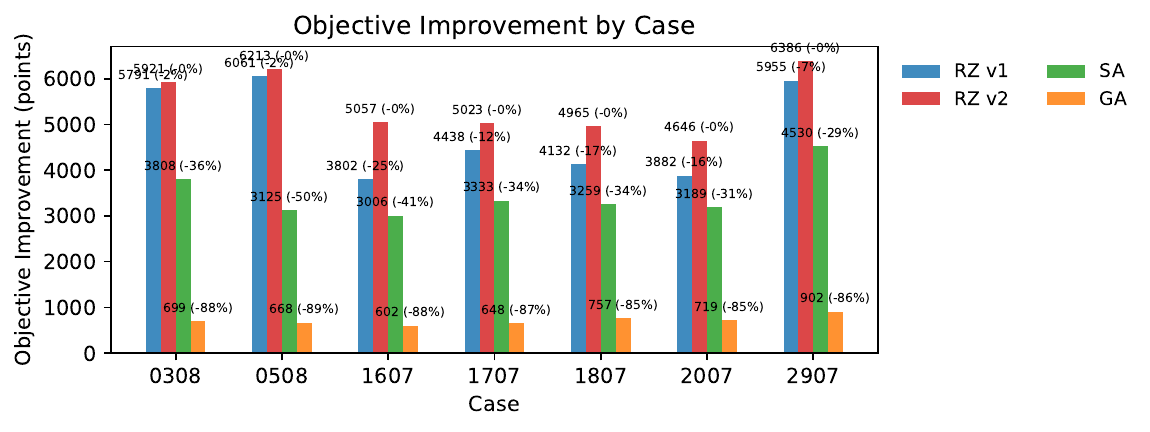}
	\caption{Objective Improvement by Case, by Algorithm. Higher is better. RZ shows state-of-the-art results.}
	\label{fig:objimprovementbycase}
\end{figure}

Table \ref{tab:case_algorithm_results} compares the key performance indices of our framework (RZ v2) against four baselines. As shown in Figure \ref{fig:objimprovementbycase}, Regulation Zero achieves state-of-the-art DCB performance, often reaching up to 30-40\% higher objective improvements than the next best optimizer family's performance, SA. The refined heuristic set in RZ v2 further contributes an additional $\sim$10\% performance gain over the RZ v1 \cite{hoang2026flocon}. At the same time, NSGA-II appeared to get stuck in a local minima early, while Greedy capping created a lot of secondary overloads---the classic symptom of Regulation Cascading.

We observe that the Greedy capping strategy generated secondary hotspots early in the process, indicating that it is already an ineffective strategy from the beginning. This observation aligns with EUROCONTROL Flow ConOps \cite{EurocontrolFlowConops2024}, which similarly cautions against blanket capping and instead advocates for more targeted regulation of individual or selected flows. The large negative outcomes arise from repeatedly applying this flawed strategy.

\subsubsection{RZ's Solutions Concentrate on Smaller Sets of Flight, Touch Fewer Traffic Volumes, and More Likely to Target High Leverage Points in the Network}

Because RZ optimizes over flows automatically extracted based on their impact, it also concentrates delay on a smaller subset of flights. In practice, RZ affects roughly 25\% fewer flights than SA (Table \ref{tab:case_algorithm_results}) and keeps the impact localized in a tighter geographical bound.

\begin{figure}
	\centering
	\includegraphics[width=\linewidth]{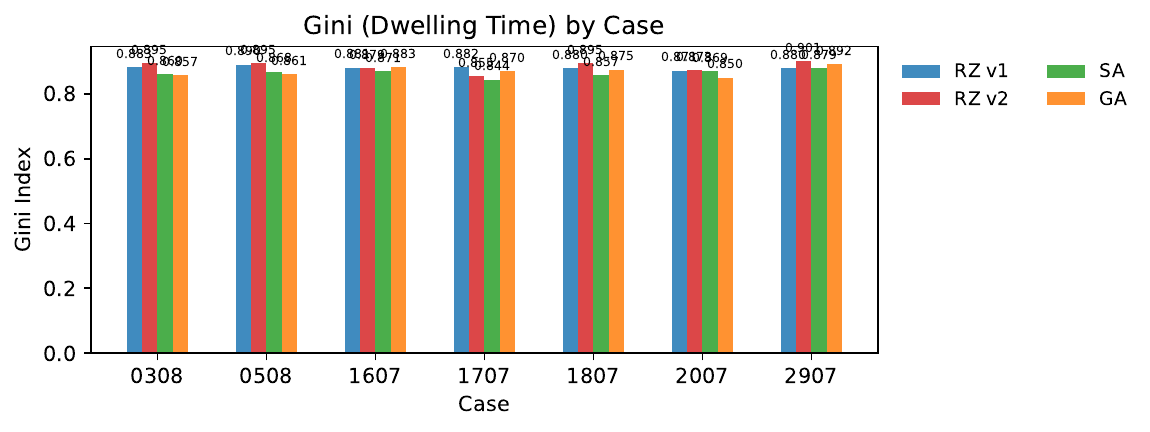}
	\caption{Traffic-Volume Gini indices measure the delay distribution fairness across Traffic Volumes.}
	\label{fig:ginidwellingbycase}
\end{figure}

Table \ref{tab:changed_pairs_results} corroborates this finding: the number of traffic volume-time pairs (i.e., ``cells'') experiencing demand changes differs by less than 10\% between SA and RZ, yet the resulting DCB performance (measured in points) can improve by up to 30\%. Figure \ref{fig:ginidwellingbycase} further shows that RZ exhibits higher traffic volume Gini indices, indicating that delays are distributed more across traffic volumes in a more localized pattern compared to SA and GA.

Table \ref{tab:changed_pairs_results} further indicates that RZ outperforms both SA and GA in terms of \emph{beneficial cells}, defined as traffic volume–time bin pairs that either reduce overload or absorb additional traffic in underloaded conditions. In particular, RZ yielded 0.41 beneficial cells per cell changed, compared to SA at around 0.23. This suggests that RZ not only limits the scope of its impact but also redistributes traffic more effectively from ``hot'' to ``cold'' regions of the network by concentrating its actions on \emph{beneficial cells} much more than SA and GA.

The ``tighter'' RZ delay assignment pattern, in our view is neither inherently good or bad; their value depends strongly on the operational context. If DCB can be managed more effectively through capacity-side measures (e.g., opening additional sectors or adjusting configurations), then concentrating heavier delays on a narrower portion of the network may be preferable. Conversely, if fairness across ACCs is a primary concern, a more diffuse delay distribution, such as that produced by SA, may be more appropriate.

\subsection{Optimization Progression Comparison}
\subsubsection{RZ is Able to Overcome Local Minimas Much More Effective}
\begin{figure}[htbp]
	\centering
	
	\begin{subfigure}[b]{0.48\textwidth}
		\centering
		\includegraphics[width=\textwidth]{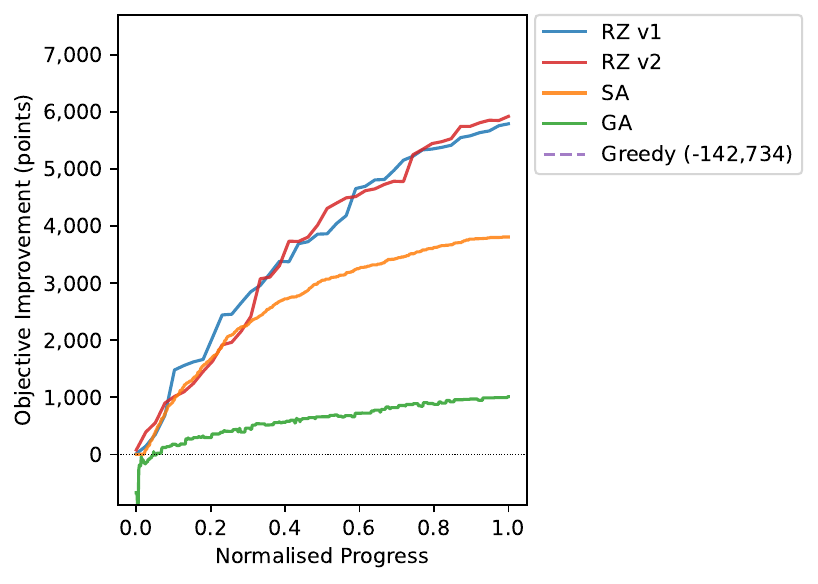}
		\caption{03/08}
	\end{subfigure}
	\hfill
	\begin{subfigure}[b]{0.48\textwidth}
		\centering
		\includegraphics[width=\textwidth]{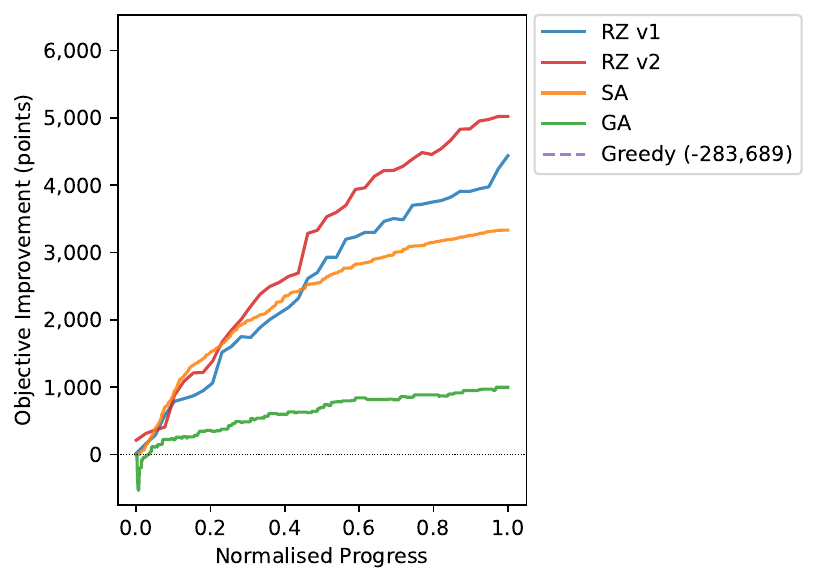}
		\caption{17/07}
	\end{subfigure}
	
	\vspace{0.5em}
	
	\begin{subfigure}[b]{0.48\textwidth}
		\centering
		\includegraphics[width=\textwidth]{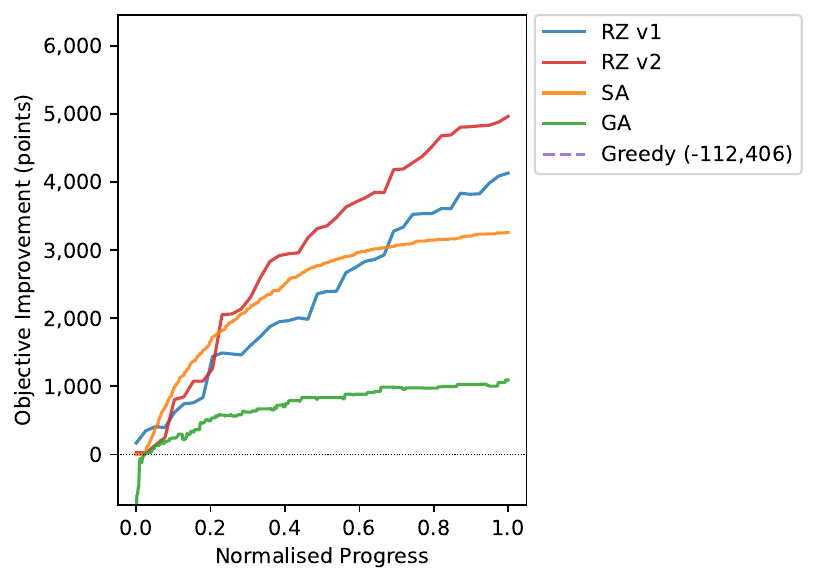}
		\caption{18/07}
	\end{subfigure}
	\hfill
	\begin{subfigure}[b]{0.48\textwidth}
		\centering
		\includegraphics[width=\textwidth]{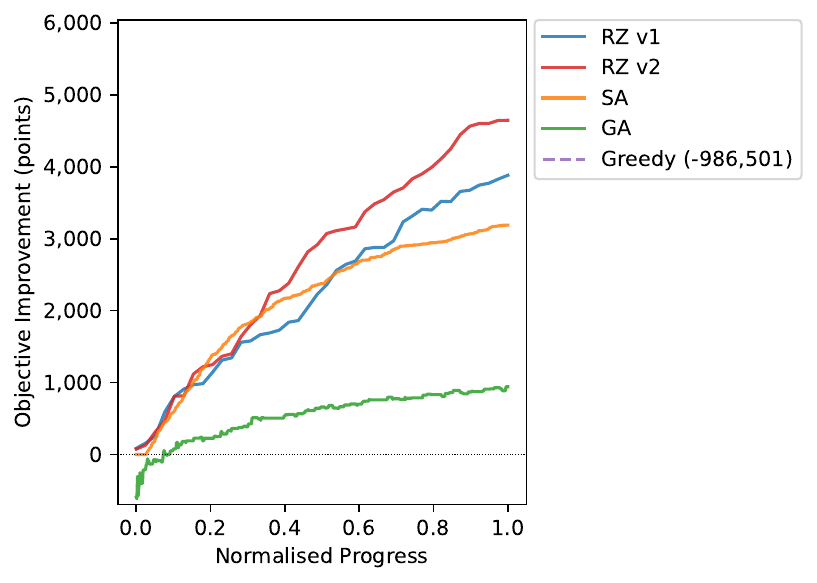}
		\caption{20/07}
	\end{subfigure}
	
	\caption{Comparison of objective improvement score against relative run time. Higher is better.}
	\label{fig:multialgos_obj_imp_2x2}
\end{figure}

Figure \ref{fig:multialgos_obj_imp_2x2} shows the progression of objective improvement as a function of normalized run time. Each algorithm was allotted the same maximum time budget of 2 hours and 30 minutes. Although differences in implementation and hyperparameter tuning make direct comparisons impractical, the results could be interpreted in terms of convergence behavior.

Overall, RZ remains competitive with SA even in the early stages of the optimization process and significantly outperforms GA. Both SA and GA tend to reach local minima early, plateauing at around 50\% of the allotted run time. In contrast, both RZ variants continue to improve throughout, suggesting that their performance advantage stems from their higher ability to escape local minima. This is highly plausible as in flow-centric methods, traffic is moved in blocks, thus reduces the number of decision variables and the effective dimensionality of the problem.

\subsubsection{RZ Excels at Exceedance Reduction; SA and GA Excel at Delay Optimization.}

\begin{figure}
	\centering
	\includegraphics[width=\linewidth]{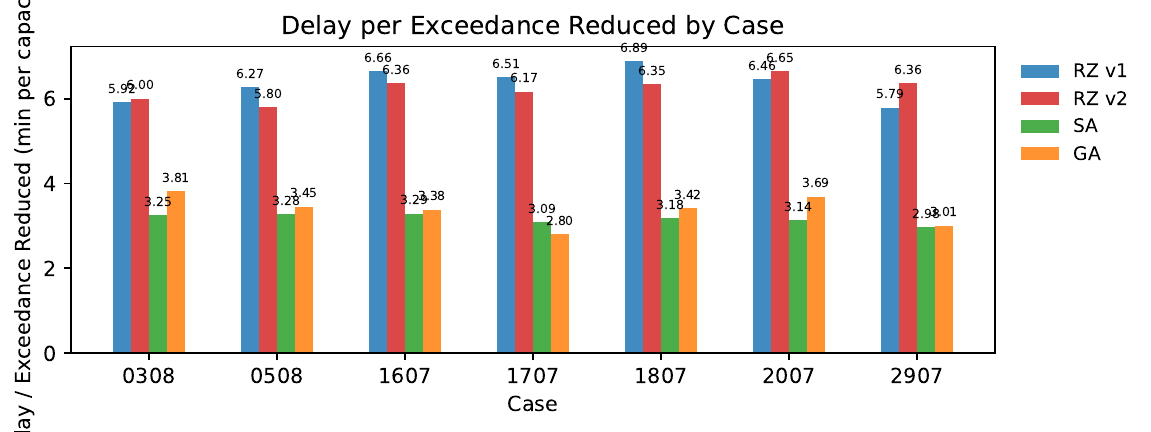}
	\caption{Average Delay imposed per Exceedance Reduced per Case, per Algorithm. Lower is Better.}
	\label{fig:delayperexceedancereducedbycase}
\end{figure}

An important question concerns the source of improvement: do the gains come primarily from reducing overload or from optimizing delays? Figure \ref{fig:delayperexceedancereducedbycase} suggests that flow-centric methods are less delay-efficient. The RZ family typically requires roughly twice as much delay to relieve one unit of exceedance, whereas SA and GA are comparable at around 3 minutes per exceedance capacity-hour.

In our view, the most plausible explanation is that the objective function (\ref{eq:j_eq}) lumps both the exceedance term $J_{CAP}$ and the delay penalty $J_{DELAY}$ together. Gradient-descent-like methods such as SA and GA receive far more frequent feedback from $J_{DELAY}$ than from $J_{CAP}$, and therefore tend to prioritize delay reduction over exceedance mitigation.

This line of reason is further supported by Figure \ref{fig:delayperexceedancereducedbycase}, where NSGA-II exhibits slightly higher delay per unit of exceedance reduced. As a population-based method with crossover, it deviates further from gradient-like behavior than SA, thus less favoring delay optimization. More direct evidence is provided in Figure \ref{fig:multialgos_exceedance_leverage_derivative_2x2}, which measures the contribution of $J_{CAP}$ to the overall objective at each step.

The results show that exceedance reduction for SA and GA is heavily concentrated at the start of the optimization, plateauing as early as 25\% into the run. In contrast, RZ delivers consistent exceedance reduction throughout the whole optimization process. This suggests that flight-centric methods allocate not just most of their effort, but also their time to optimizing delays rather than resolving overloads.

\begin{figure}[htbp]
	\centering
	
	\begin{subfigure}[b]{0.48\textwidth}
		\centering
		\includegraphics[width=\textwidth]{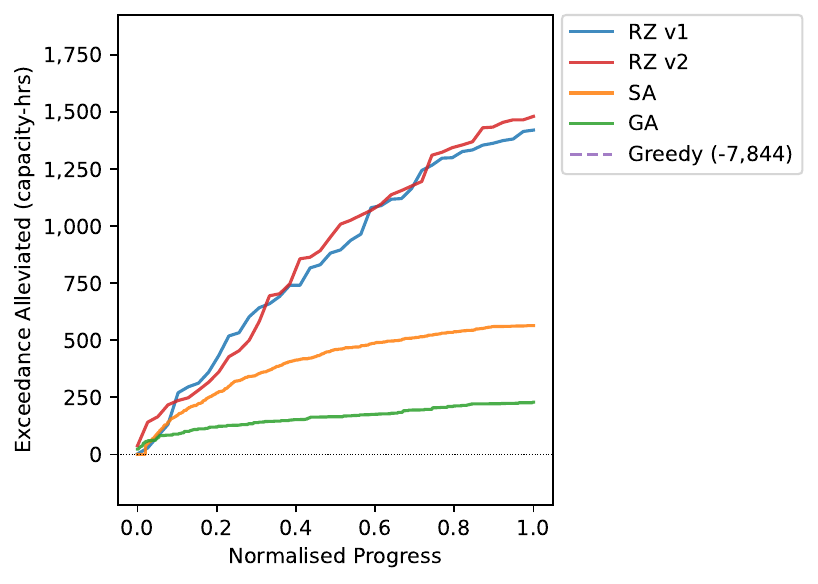}
		\caption{03/08}
	\end{subfigure}
	\hfill
	\begin{subfigure}[b]{0.48\textwidth}
		\centering
		\includegraphics[width=\textwidth]{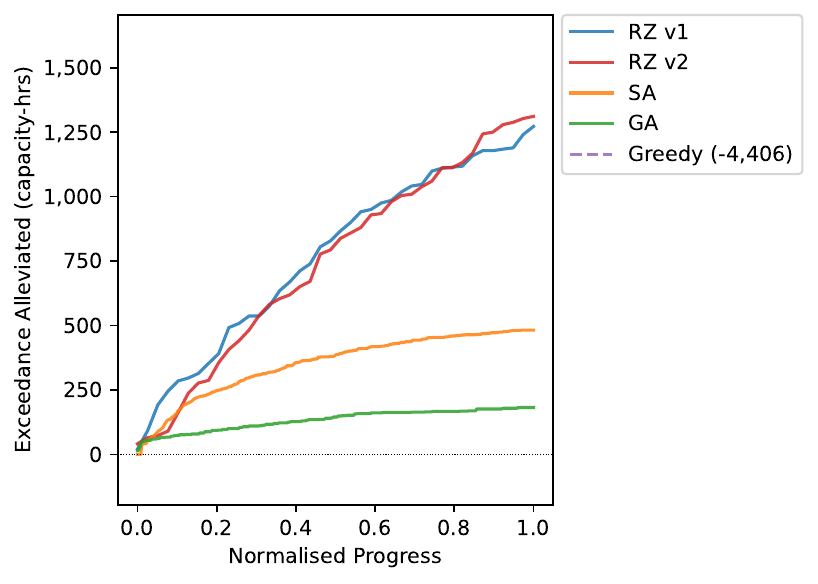}
		\caption{17/07}
	\end{subfigure}
	
	\vspace{0.5em}
	
	\begin{subfigure}[b]{0.48\textwidth}
		\centering
		\includegraphics[width=\textwidth]{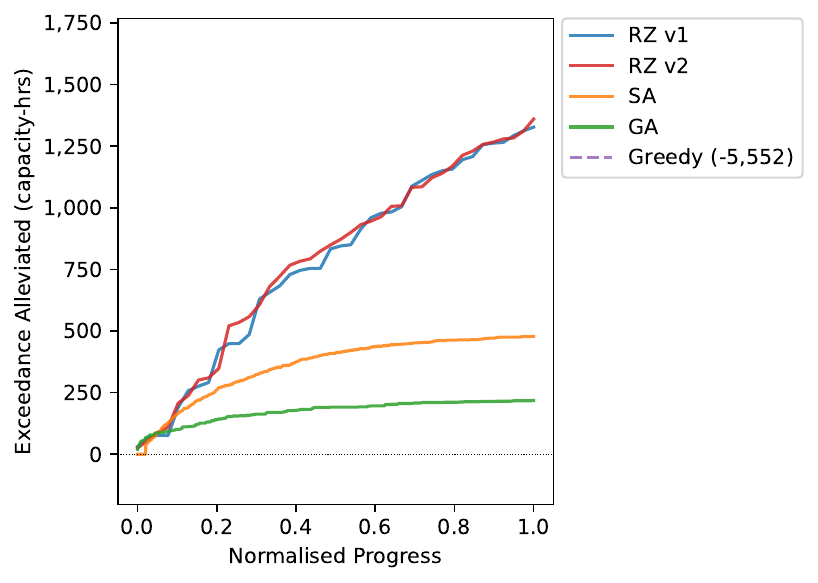}
		\caption{18/07}
	\end{subfigure}
	\hfill
	\begin{subfigure}[b]{0.48\textwidth}
		\centering
		\includegraphics[width=\textwidth]{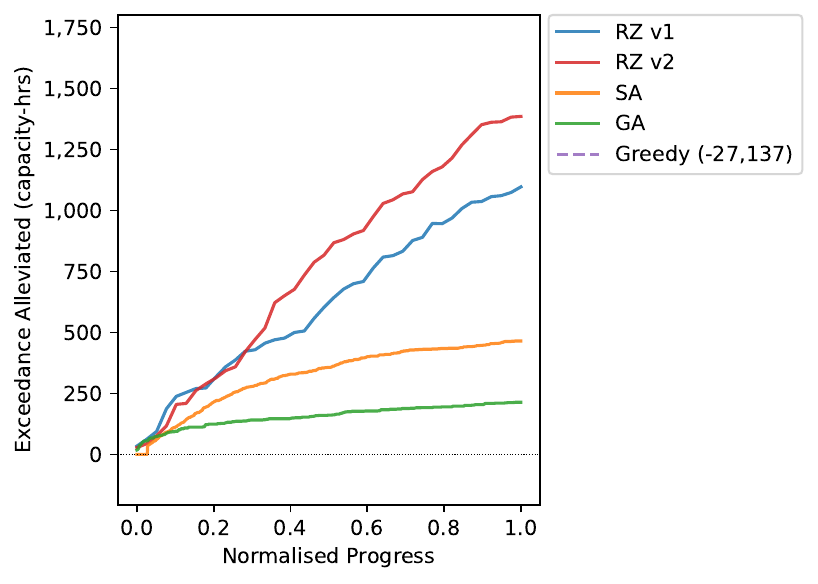}
		\caption{20/07}
	\end{subfigure}
	
	\caption{Comparison of $J_{CAP}$ score improvement against relative run time. Higher is better.}
	\label{fig:multialgos_exc_red_2x2}
\end{figure}

\begin{figure}[htbp]
	\centering
	
	\begin{subfigure}[b]{0.48\textwidth}
		\centering
		\includegraphics[width=\textwidth]{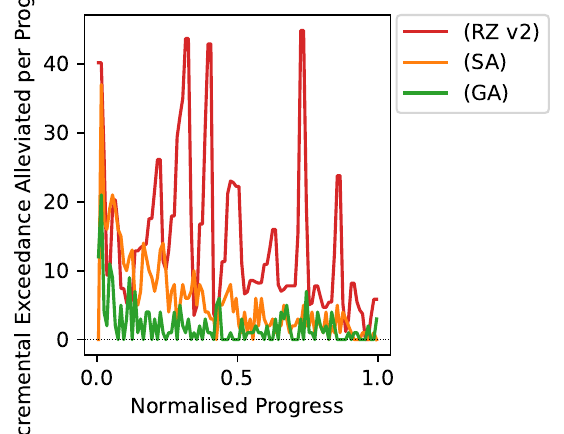}
		\caption{03/08}
	\end{subfigure}
	\hfill
	\begin{subfigure}[b]{0.48\textwidth}
		\centering
		\includegraphics[width=\textwidth]{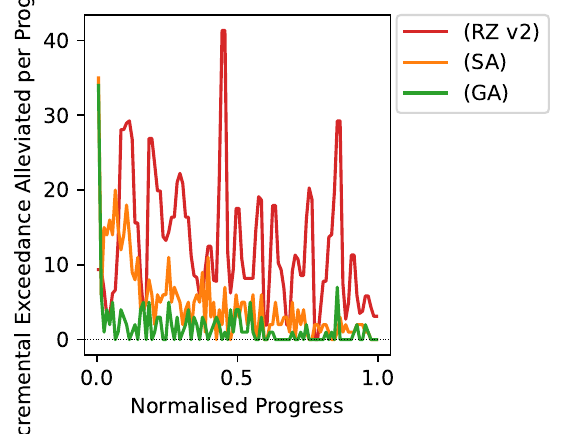}
		\caption{17/07}
	\end{subfigure}
	
	\vspace{0.5em}
	
	\begin{subfigure}[b]{0.48\textwidth}
		\centering
		\includegraphics[width=\textwidth]{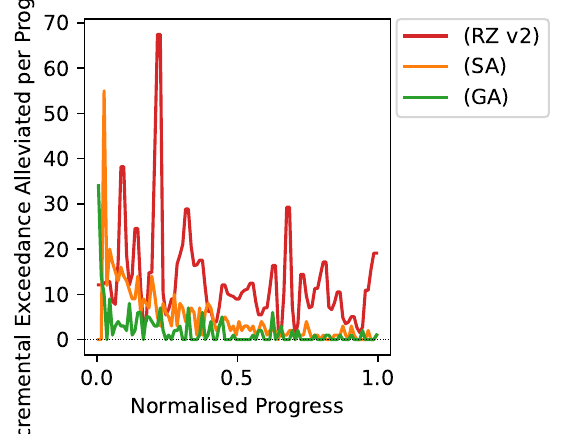}
		\caption{18/07}
	\end{subfigure}
	\hfill
	\begin{subfigure}[b]{0.48\textwidth}
		\centering
		\includegraphics[width=\textwidth]{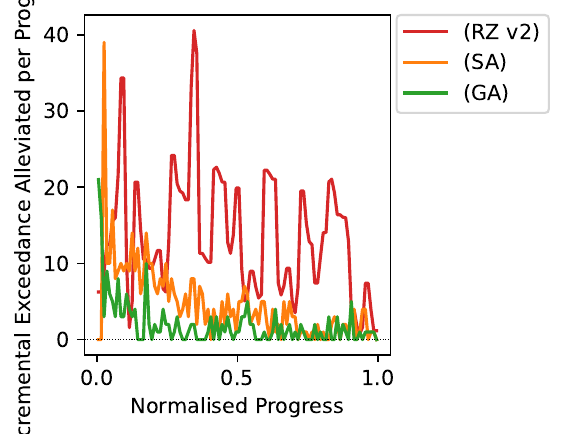}
		\caption{20/07}
	\end{subfigure}
	
	\caption{Comparison of step-wise contribution of $J_{CAP}$ to the general objective function. Higher is better.}
	\label{fig:multialgos_exceedance_leverage_derivative_2x2}
\end{figure}

\begin{figure}
	\centering
	\includegraphics[width=\linewidth]{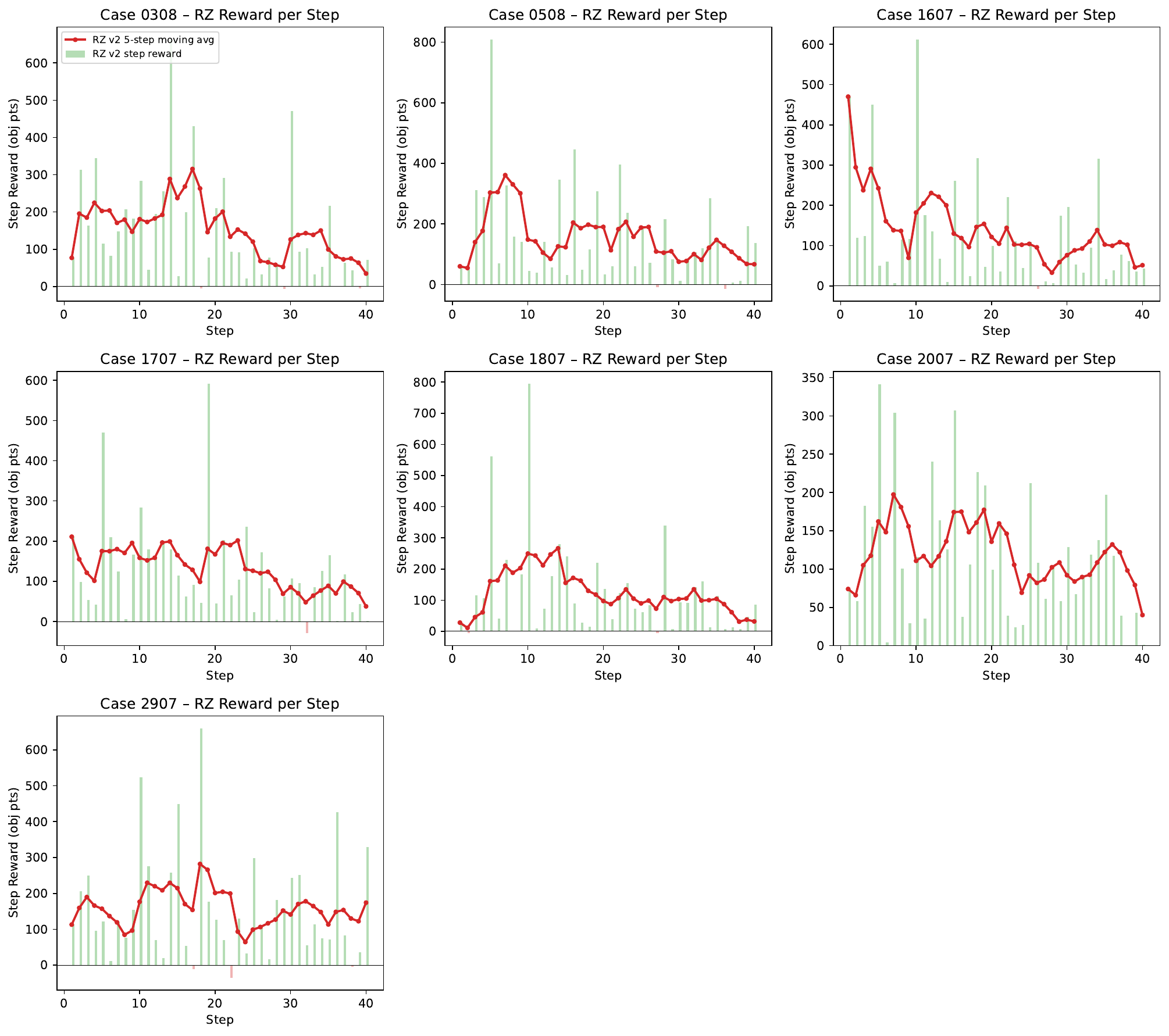}
	\caption{Reward per step and cumulative reward with the number of regulations imposed.}
	\label{fig:rzrewardperstep}
\end{figure}

A final observation is that, with additional computational budget, RZ still has room for further improvement. Figure \ref{fig:rzrewardperstep} shows that per-episode step rewards do not exhibit clear convergence by the end of the optimization process. In several cases, such as 16/07 and 29/07, the trend suggests that better network DCB states could be achieved with extended run time.

\subsection{Ablation Studies}
\subsubsection{RZ’s Delay Efficiency is Invariant to Problem Difficulty, Flow Composition, and Flow Size.}\label{subsec:delayeff}
While the general characteristics of RZ solutions have been established, it remains unclear whether their apparent inefficiency in delay usage depends on problem difficulty. In other words, given that RZ achieves higher objective values, it is plausible that additional delay is required to achieve further exceedance reduction beyond certain objective thresholds.

To investigate this, we derive three subcases from the full-day traffic demand: \emph{6h}, enforcing capacity constraints only between 06:00 and 12:00; \emph{12h}, enforcing between 04:00 and 16:00; and \emph{24h}, corresponding to full-day enforcement. All algorithmic hyperparameters are kept unchanged. We report results for four representative days: 17/07, 18/07, 20/07, and 05/08, which were randomly selected from the original seven-day dataset to limit computational cost.

\begin{figure}
	\centering
	\includegraphics[width=\linewidth]{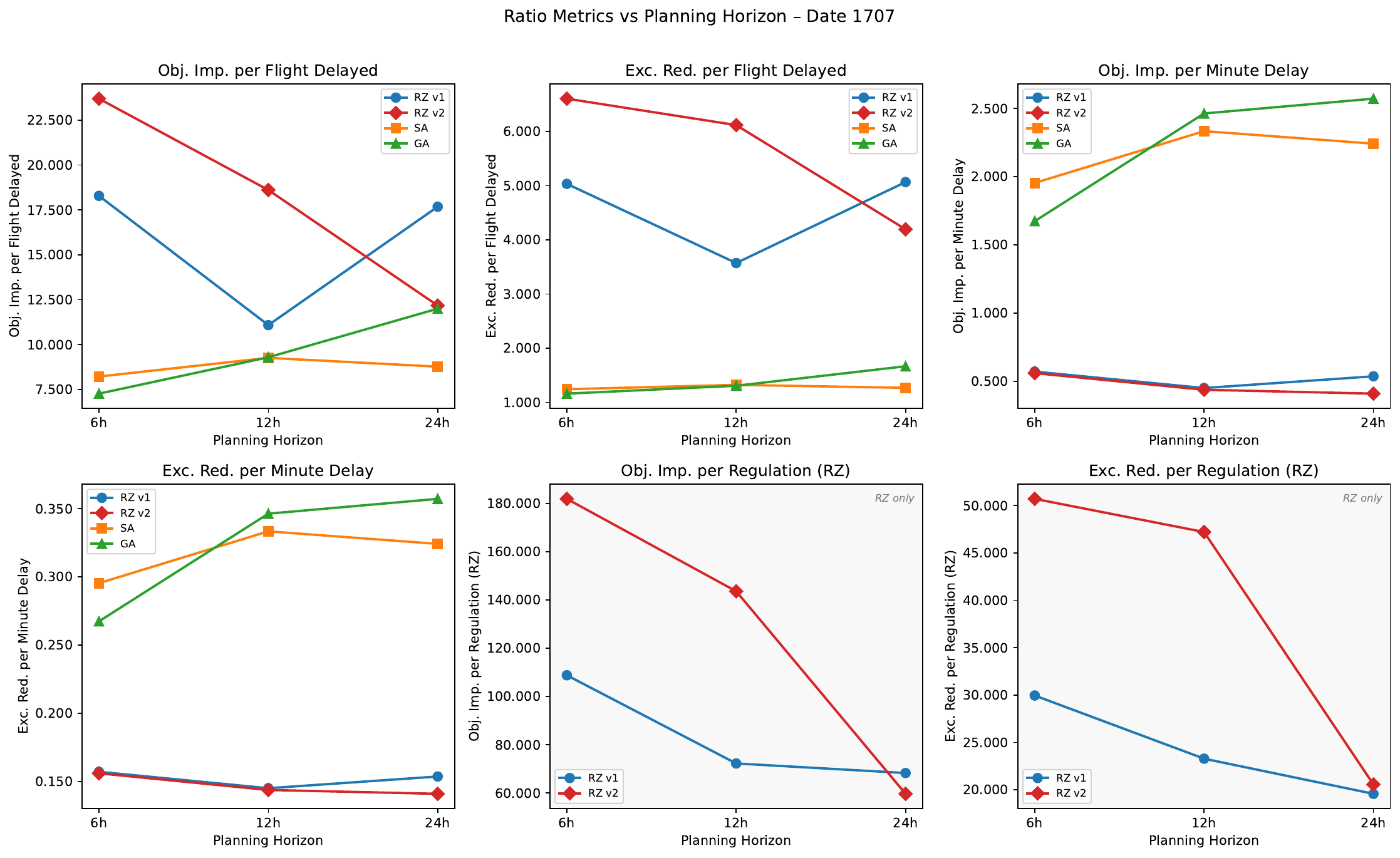}
	\caption{Comparison of ratio metrics for 3 subcases: 6h, 12h and full-day for the 17/07 case.}
	\label{fig:ratiometrics1707}
\end{figure}

\begin{figure}
	\centering
	\includegraphics[width=\linewidth]{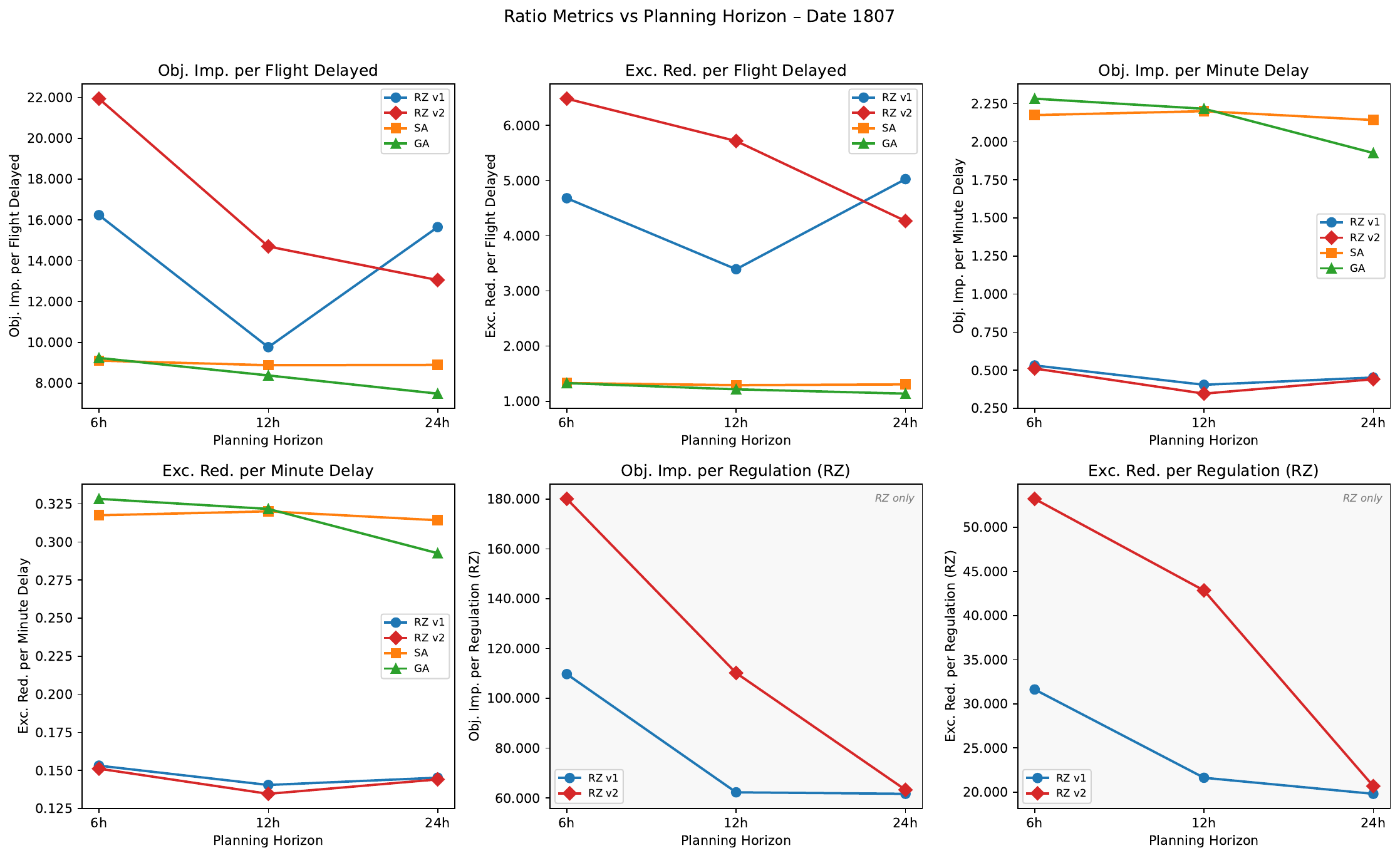}
	\caption{Comparison of ratio metrics for 3 subcases: 6h, 12h and full-day for the 18/07 case.}
	\label{fig:ratiometrics1807}
\end{figure}

\begin{figure}
	\centering
	\includegraphics[width=\linewidth]{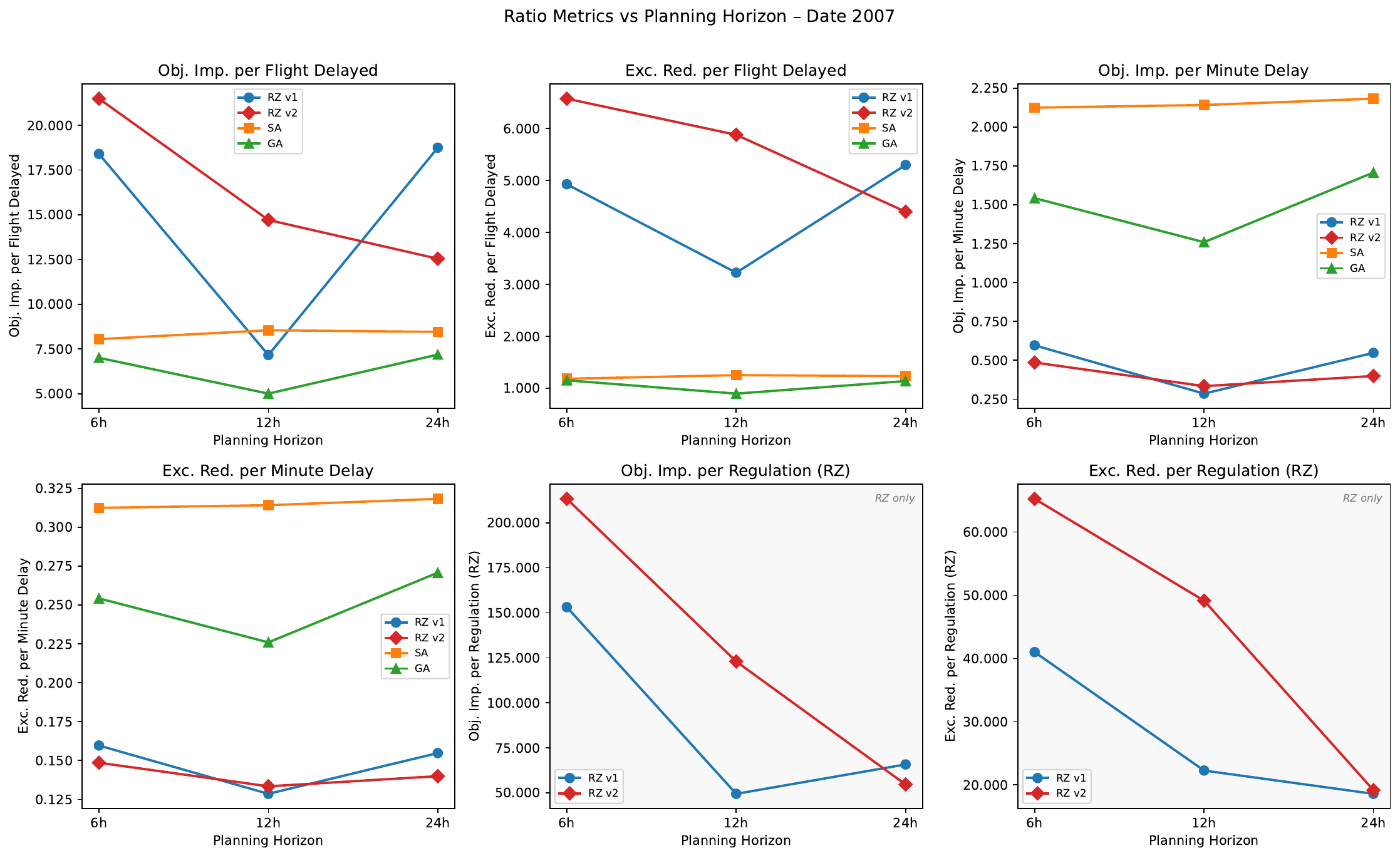}
	\caption{Comparison of ratio metrics for 3 subcases: 6h, 12h and full-day for the 20/07 case.}
	\label{fig:ratiometrics2007}
\end{figure}

\begin{figure}
	\centering
	\includegraphics[width=\linewidth]{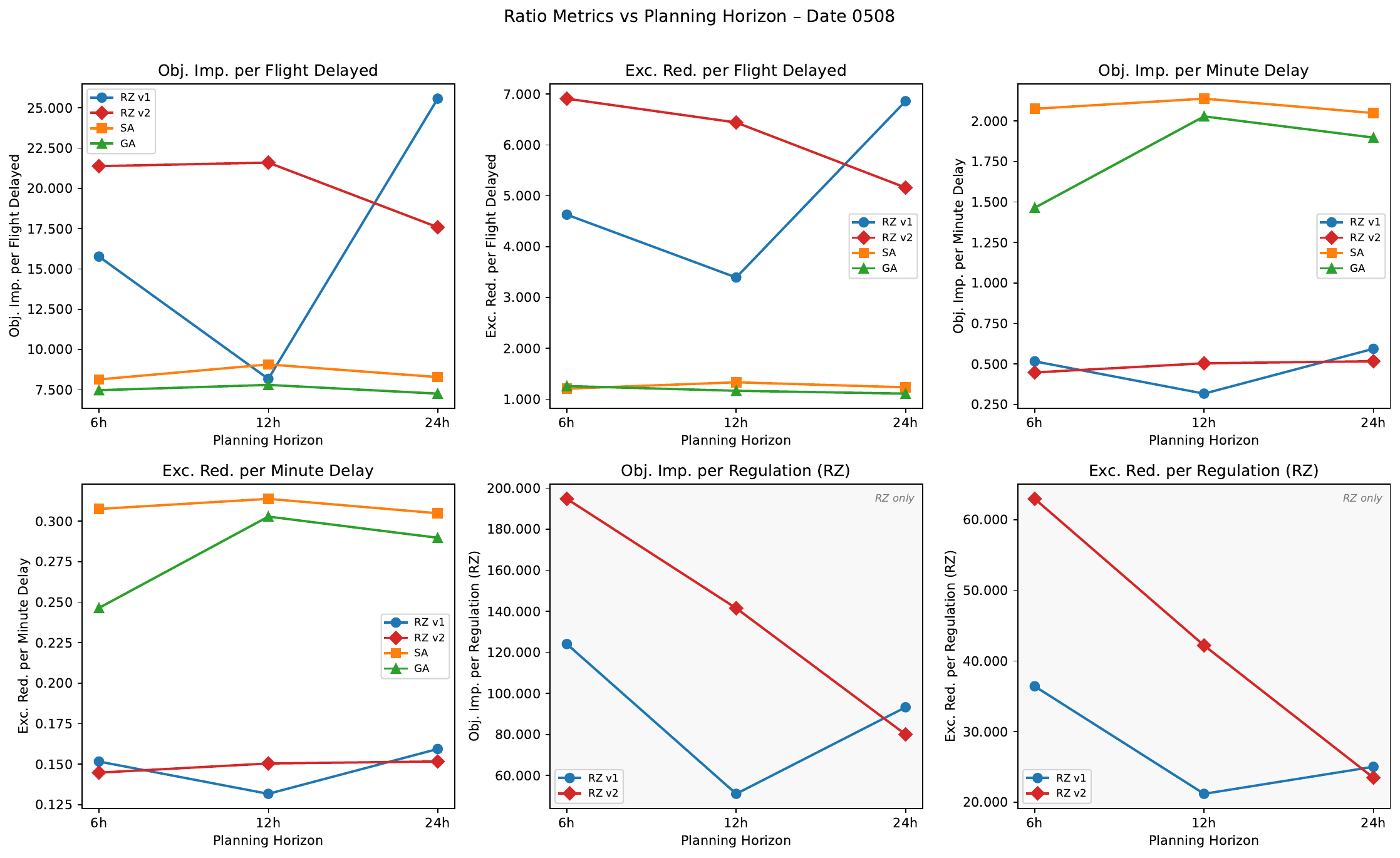}
	\caption{Comparison of ratio metrics for 3 subcases: 6h, 12h and full-day for the 05/08 case.}
	\label{fig:ratiometrics0508}
\end{figure}

Figures \ref{fig:ratiometrics1707}-\ref{fig:ratiometrics0508} show that RZ’s delay-spending behavior is largely invariant to problem size and difficulty. Exceedance reduction per minute of delay remains stable at around 0.15: roughly half that of flight-centric methods such as SA and GA. In contrast, performance ratios, such as objective improvement or exceedance reduction per delayed flight, decline with increasing problem difficulty, as expected.

One other possibility is that the flow selection process may introduce implicit regularization through flow size. We conducted another ablation study over three threshold values (0.53, 0.69, 0.82), which produce progressively larger flows, shows no clear trend in delay-use efficiency, however (Figure \ref{fig:flowthresholdsweep4panel}).

\begin{figure}
	\centering
	\includegraphics[width=\linewidth]{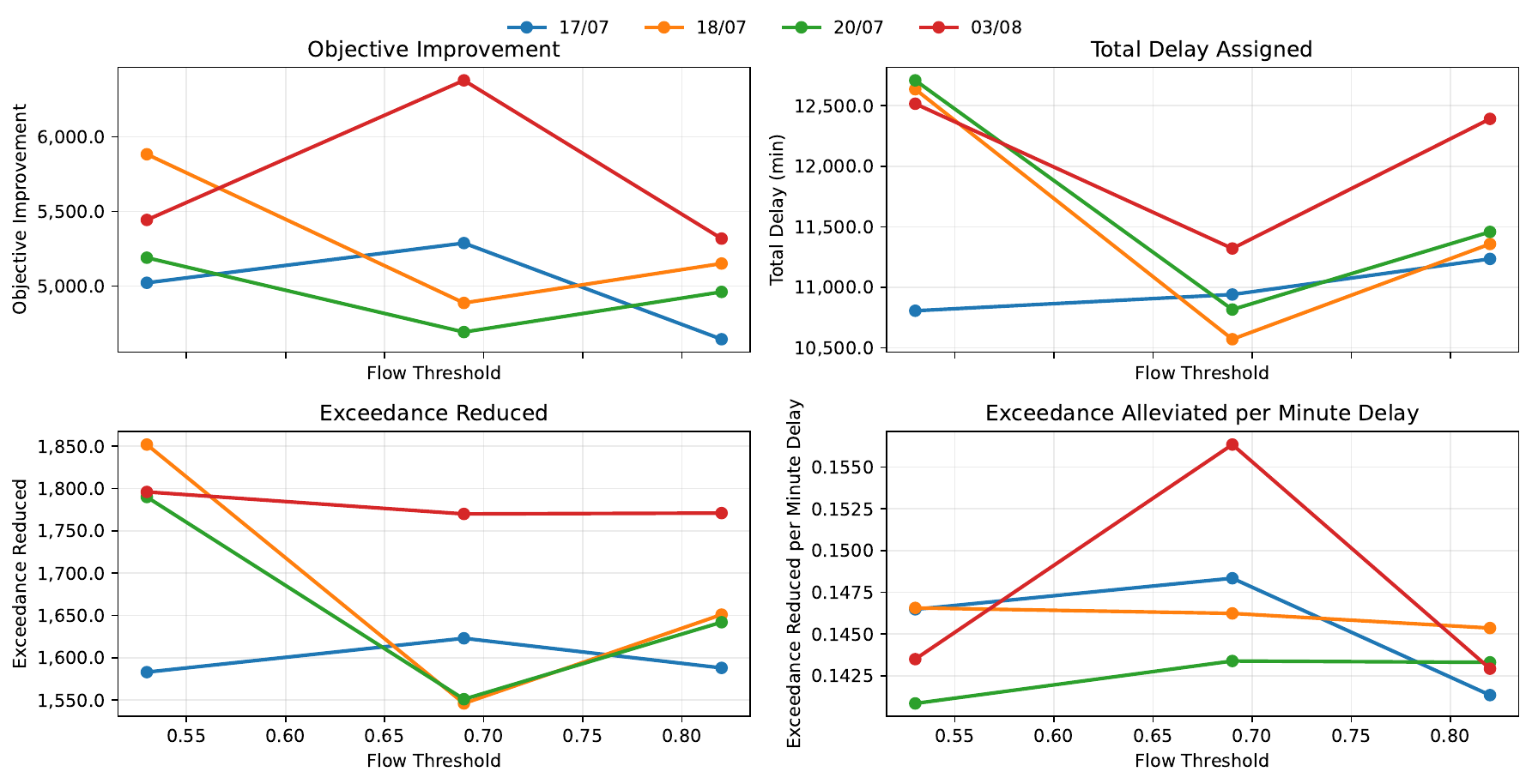}
	\caption{Comparison of Flow Threshold / Flow Size effects for four full-day cases of 17/07, 18/07, 20/07 and 03/08}
	\label{fig:flowthresholdsweep4panel}
\end{figure}

Because RZ only depends on three components: MCTS, Regulation Proposal and FPFS Slot Allocator, and that the delay efficiency does not deviating much from RZ v1, which used different heuristic sets for flow selection (and therefore, different flows were selected); the evidence together points in the direction that this lack of efficiency appears to be robust, and associate with the FPFS slot allocator. We hypothesize that this efficiency is the price that has to be paid to maintain the FPFS fairness constraint on the solution, but we admit that it is currently unknown the exact mechanism how this manifested at this moment.

\subsubsection{v2's NomRel and InLoad Heuristics Help Find Effective Flows Faster, and Where to Regulate May Be Less Important than Finding the Right Flows}

\begin{figure}
	\centering
	\includegraphics[width=\linewidth]{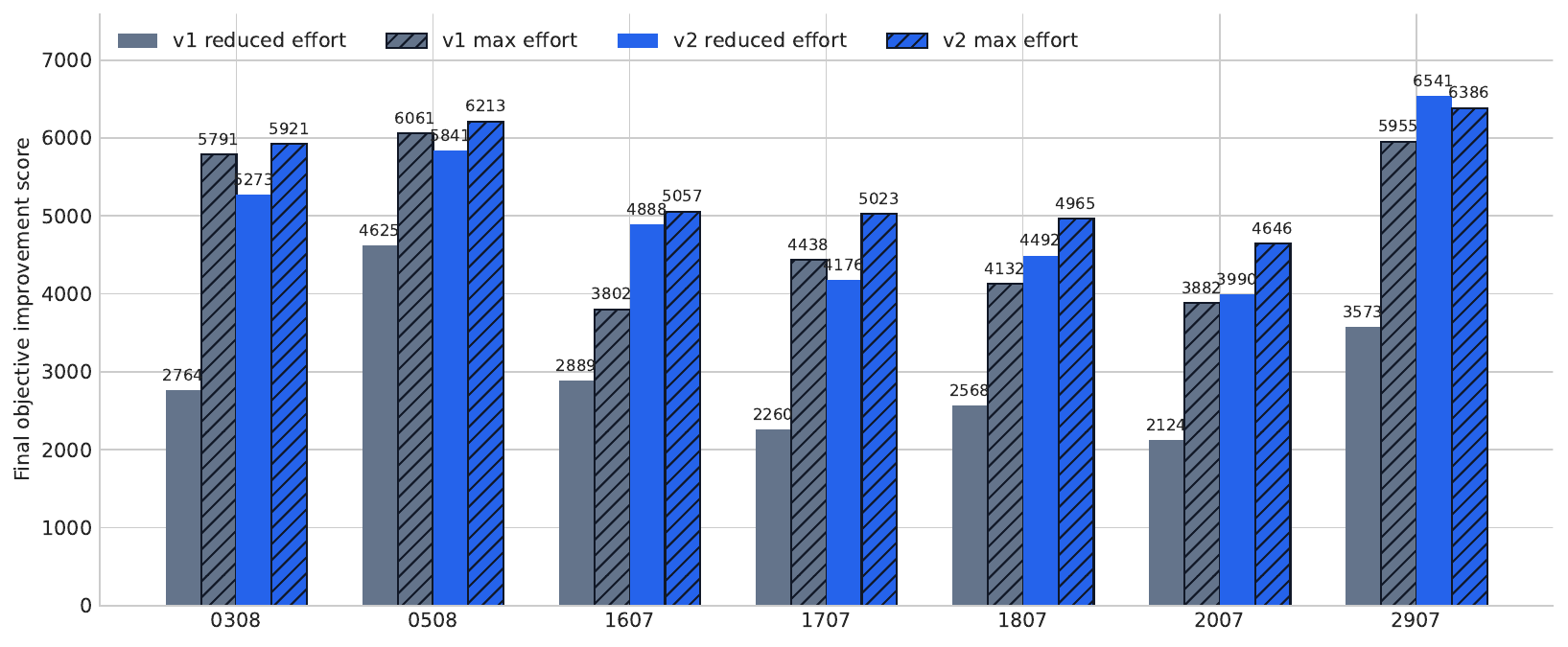}
	\caption{Final Objective Improvement comparison between RZ v1 and v2 for reduced effort (hotspots=5) and max effort (hotspots=20).}
	\label{fig:fulldayreducedeffortv1v2objimprovementfourseries}
\end{figure}

To assess how much the NomRel and InLoad heuristics improve computational efficiency within the full optimization pipeline, we conducted additional experiments on the same seven full-day cases, reducing the number of hotspots used for flow extraction from 20 to 5. This corresponds to an approximate 80\% reduction in computational effort. However, this does not translate directly into the same time savings: in practice, runtime only decreased by about 50\%, with total execution time ranging from 75 to 85 minutes, mostly due to logging and disk I/O overhead.

The results in Figure \ref{fig:fulldayreducedeffortv1v2objimprovementfourseries} show that the refined flow-selection heuristics substantially improved RZ's performance. Notably, the reduced-effort v2 configuration now matches the performance of v1 at full effort and remains within 5-10\% of v2 at full effort. This trend is consistent across all 7 cases, suggesting that effective flow-centric optimization does not require a large pool of hotspots, provided that high-quality flows are selected. In practice, this indicates that finding the specific reference location is less critical than identifying the high-leverage flows that span multiple traffic volumes and ACCs, as many hotspots appear to lead to the same flows identified.

\subsubsection{Objective Weight Variation Study: RZ Dominates Large-Scale and High-Exceedance Regimes}
In this section, we attempt to estimate the Pareto frontier through varying the exceedance weight $w_{CAP}$. The results are presented in Figure \ref{fig:alphaweightparetogrid} and \ref{fig:alphaweightratiogrid}.

\begin{figure}
	\centering
	\includegraphics[width=\linewidth]{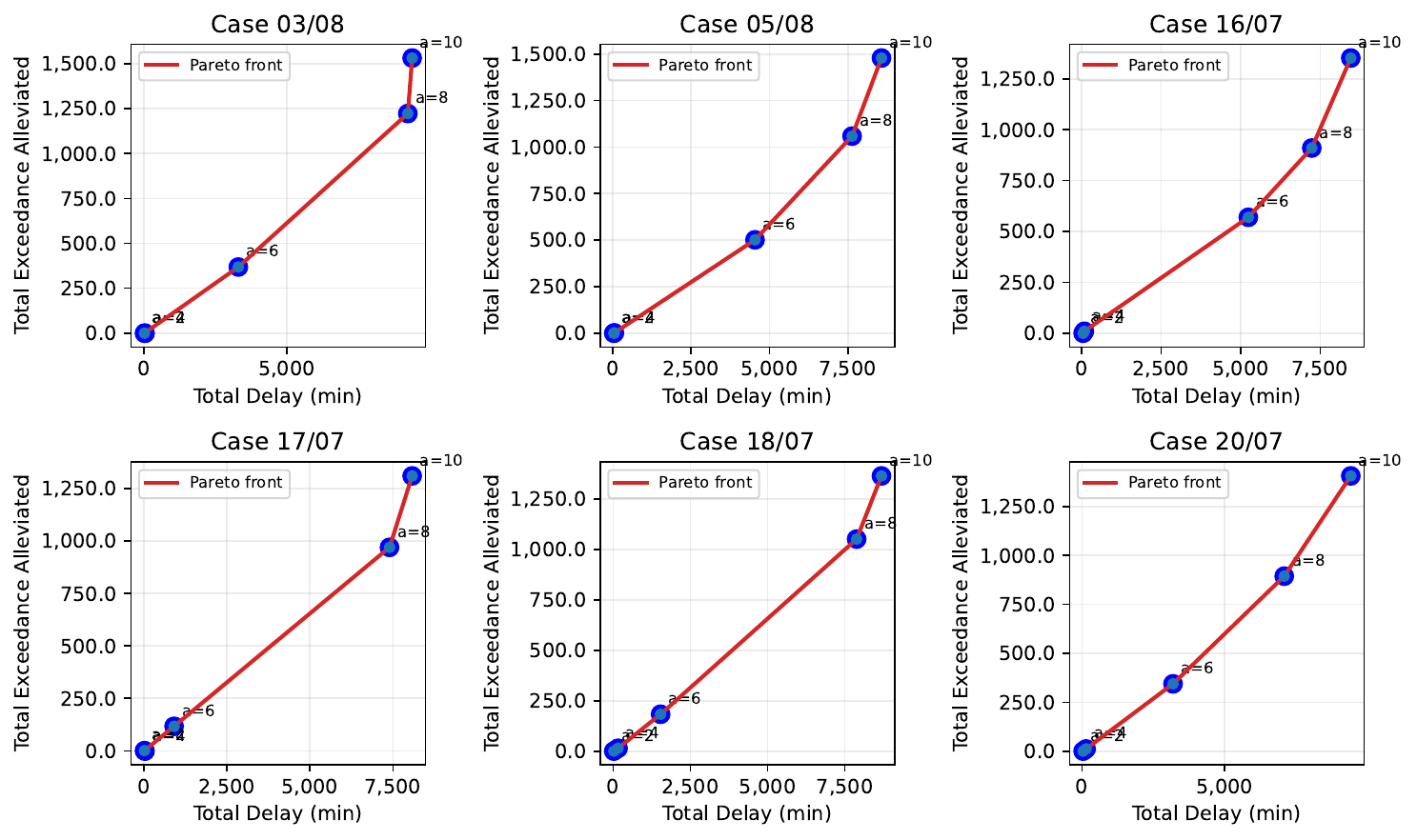}
	\caption{RZ's Pareto Frontier for six full-day cases with $w_{CAP}$ varying from 2 to 10.}
	\label{fig:alphaweightparetogrid}
\end{figure}

\begin{figure}
	\centering
	\includegraphics[width=\linewidth]{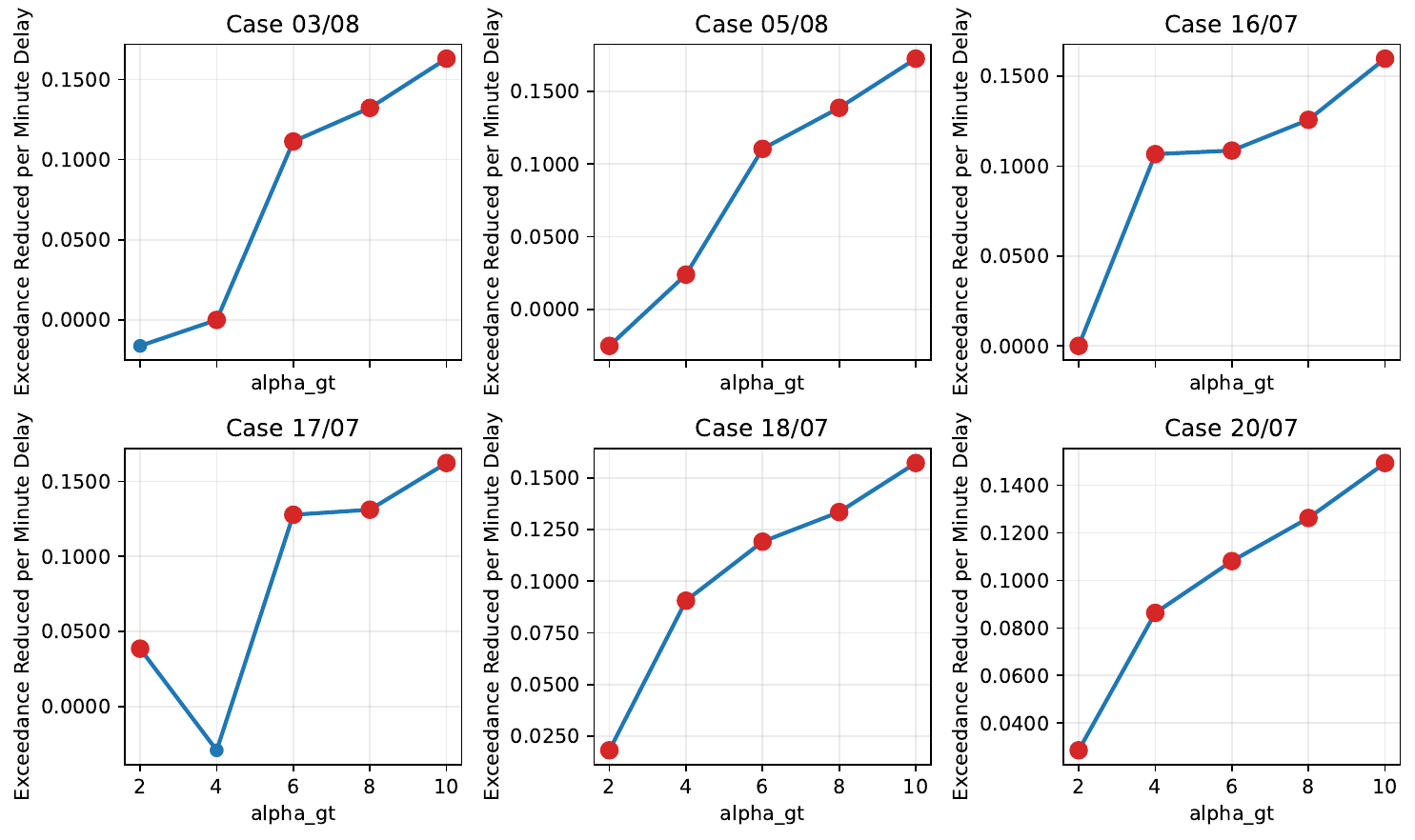}
	\caption{Variation of Exceedance Reduced per Minute Delay when $w_{CAP}$ was varied from 2 to 10.}
	\label{fig:alphaweightratiogrid}
\end{figure}

The results reveal a consistent and robust pattern in RZ’s behavior. Across all cases, the Pareto front remains stable and favors effective exceedance reduction as larger delay budgets are allowed, given a fixed number of regulations. Conversely, forcing RZ to prioritize delay minimization leads to low performance, even solution collapses, as observed at low $\alpha = w_{CAP}$ values (2 and 4). Figure \ref{fig:alphaweightratiogrid} presents the finding in greater details: exceedance reduction per minute of delay is observed to climb steadily as more weight is put on exceedance.

Taken together, these results strongly indicate that RZ is most effective in high-exceedance, high-delay regimes. This reinforces our earlier findings in \cite{hoang2026flocon} that RZ is particularly well suited to large-scale, continent-wide problems, rather than smaller regional or ACC-level scenarios.

\subsubsection{Overcoming Regulation Cascading is Essential to Unlock High Performance: an Ablation of the MCTS component}

\begin{table}[!htpb]
	\centering
	\setlength{\tabcolsep}{4pt}
	\renewcommand{\arraystretch}{1.1}
	\begin{tabular}{llcccc}
		\toprule
		Case & Policy & \makecell{Final Objective\\Improvement} & \makecell{Exceedances\\Alleviated} & \makecell{Total Delay\\(min)} & Regulations \\
		\midrule
		\multirow{2}{*}{17/07}
		& RZ v2  & 5023.0 & 1311.0 & 8087.0 & 40 \\
		& BRPP & 2049.0 & 703.0  & 4981.0 & 40 \\
		\midrule
		\multirow{2}{*}{18/07}
		& RZ v2  & 4965.0 & 1360.0 & 8635.0 & 40 \\
		& BRPP & 2000.0 & 651.0  & 4510.0 & 40 \\
		\midrule
		\multirow{2}{*}{20/07}
		& RZ v2  & 4646.0 & 1386.0 & 9214.0 & 40 \\
		& BRPP & 2219.0 & 871.0  & 6491.0 & 40 \\
		\midrule
		\multirow{2}{*}{03/08}
		& RZ v2  & 5921.0 & 1481.0 & 8889.0 & 40 \\
		& BRPP & 1598.0 & 688.0  & 5282.0 & 40 \\
		\bottomrule
	\end{tabular}
	\caption{Comparison of RZ v2 (MCTS-enabled) and Best Regen Proposal Policy (BRPP/MCTS-disabled) performance across four full-day cases.}
	\label{tab:rzv2_greedy_comparison}
\end{table}

To study the contribution of the MCTS loop, we conducted an ablation study using a simplified policy, termed \emph{Best Regulation Proposal Policy} (BRPP), in which only the highest-reward proposal is selected at each step. The key difference from RZ is the absence of the lookahead component: whereas RZ explores multiple steps ahead to identify an optimal sequence of regulations, BRPP only selects the best immediate action. For comparability, we constrain BRPP to apply the same number of regulations as in the RZ solution.

\begin{figure}
	\centering
	\includegraphics[width=\linewidth]{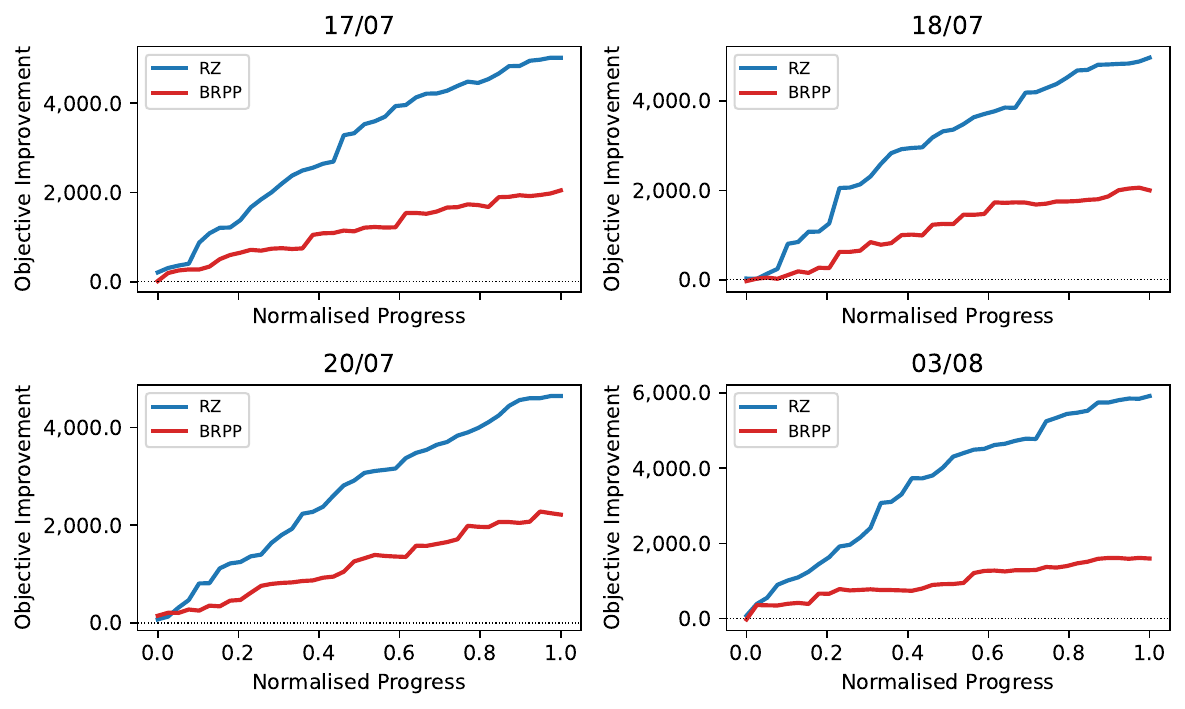}
	\caption{Improvement of Objective between RZ v2 (MCTS-enabled) and BRPP (MCTS-disabled)}
	\label{fig:objectiveimprovementrzvsgreedy}
\end{figure}

The results demonstrate that removing MCTS significantly degrades performance, with the final objective reaching only about half of that achieved with MCTS (Table \ref{tab:rzv2_greedy_comparison}). The objective improvement trajectory also exhibits more plateau segments, indicating that locally optimal ``best'' proposals often yield minimal global benefit. This implies that the MCTS component of RZ plays an important role in mitigating Regulation Cascading (RC), and that RC could account for up to 50\% loss of effectiveness of the regulation plan.

\section{Discussion}
\subsection{A Broader Context for Interpretation}
Before interpreting our results, we situate them against comparable work in the literature and emphasize about differences in problem scope, geographic coverage, and action space prevent direct numerical comparison. 

Our experiments span roughly 22,000 to 25,000 flights, which is substantially larger than the scale typically reported. The two studies we could find of comparable scale are \cite{delahaye2018simulated} and \cite{aerospace8100288}, covering 29,852 pan-European and 8,836 French-airspace flights respectively; however, both address 4D trajectory deconfliction rather than DCB. \cite{delahaye2018simulated}, in particular, adjusts not only takeoff times but also reroutes flights, reporting up to 46.23\% of flights receiving route extensions. The isolated performance of delay-only actions in that setting is not reported, so the numbers are not directly comparable to ours.

\cite{Dalmau2022Optimal} conducted experiments on R-NEST, and its first phase resembles the Greedy policy we present. Their scope, however, is restricted to the ACCs involved in the ISOBAR2 project: LFMMCTA (Marseille), LFEECTA (Reims), LECMCTA (Madrid), LECBCTA (Barcelona), and LECPCTA (Palma). In our broader experiments, applying a greedy strategy to LF or LE airspaces strongly induced secondary exceedances, predominantly in the EG airspaces, an effect that would not surface within the narrower ISOBAR2 scope. \cite{Dalmau2021Indispensable} targets a distinct sub-problem: cancelling unnecessary imposed regulations rather than constructing a regulation chain from scratch. Furthermore, access to historical regulation data was limited, so a direct comparison against authority-issued regulations is infeasible for RZ.

\subsection{Flow-centric and Flight-centric Dominate Different Problem Regimes}
The results in Section \ref{sec:results} indicate that flow-centric methods such as RZ v2 and RZ can achieve better DCB solutions, but with distinct structural characteristics. In particular, RZ prioritizes exceedance reduction, whereas flight-centric methods (SA and GA) tend to favor delay optimization (Figure \ref{fig:rz-sa-ga-axes}). Furthermore, by operating in a reduced decision space, RZ avoids the early plateauing observed in SA and GA and continues to improve throughout the optimization process (Figure \ref{fig:rzrewardperstep}). This feature makes RZ inherently more scalable and better suited to larger-scale problems.

Importantly, these results suggest that flow-centric and flight-centric approaches are not direct competitors but complementary methods, each excelling in different regimes. Figure \ref{fig:rz-sa-ga-axes} also points to a promising research direction: understanding the source of RZ’s lower delay efficiency. One potential avenue is the development of hybrid strategies, where RZ solutions are used to warm-start flight-centric methods. This hybrid approach potentially preserve RZ's exceedance advantages while salvaging its delay efficiency.

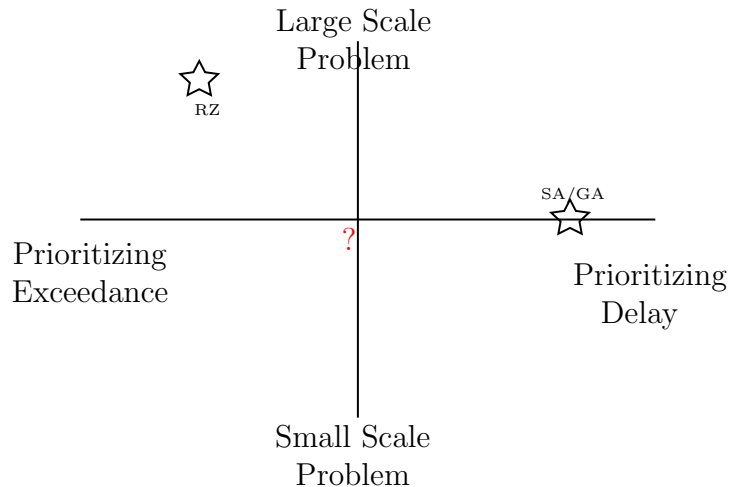
\begin{figure}
	\centering
	\begin{tikzpicture}[x=0.75pt,y=0.75pt,yscale=-1,xscale=1]
	
	\draw    (50,120) -- (340,120) ;
	\draw    (190,30) -- (190,220) ;
	\draw   (110,40) -- (112.94,45.66) -- (119.51,46.56) -- (114.76,50.97) -- (115.88,57.19) -- (110,54.25) -- (104.12,57.19) -- (105.24,50.97) -- (100.49,46.56) -- (107.06,45.66) -- cycle ;
	\draw   (297,110) -- (299.94,115.66) -- (306.51,116.56) -- (301.76,120.97) -- (302.88,127.19) -- (297,124.25) -- (291.12,127.19) -- (292.24,120.97) -- (287.49,116.56) -- (294.06,115.66) -- cycle ;
	
	\draw (14,130) node [anchor=north west][inner sep=0.75pt]  [font=\normalsize] [align=left] {\begin{minipage}[lt]{58.86pt}\setlength\topsep{0pt}
			\begin{center}
				Prioritizing\\Exceedance
			\end{center}
			
	\end{minipage}};
	\draw (297,132) node [anchor=north west][inner sep=0.75pt]  [font=\normalsize] [align=left] {\begin{minipage}[lt]{50.34pt}\setlength\topsep{0pt}
			\begin{center}
				Prioritizing\\Delay
			\end{center}
			
	\end{minipage}};
	\draw (121,12) node [anchor=north west][inner sep=0.75pt]  [font=\normalsize] [align=left] {\begin{minipage}[lt]{97.97pt}\setlength\topsep{0pt}
			\begin{center}
				Large Scale Problem
			\end{center}
			
	\end{minipage}};
	\draw (121,222) node [anchor=north west][inner sep=0.75pt]  [font=\normalsize] [align=left] {\begin{minipage}[lt]{97.39pt}\setlength\topsep{0pt}
			\begin{center}
				Small Scale Problem
			\end{center}
			
	\end{minipage}};
	\draw (106.12,60.19) node [anchor=north west][inner sep=0.75pt]  [font=\tiny] [align=left] {RZ};
	\draw (281,102) node [anchor=north west][inner sep=0.75pt]  [font=\tiny] [align=left] {SA/GA};
	\draw (200,130) node  [color={rgb, 255:red, 255; green, 0; blue, 0 }  ,opacity=1 ] [align=left] {\begin{minipage}[lt]{27.2pt}\setlength\topsep{0pt}
			?
	\end{minipage}};

\end{tikzpicture}
	\caption{Positioning of RZ against flight-centric methods in categories of DCB problems.}
	\label{fig:rz-sa-ga-axes}
\end{figure}

\subsection{Self-Regulating Regulations}
Another important question is whether RZ offers benefits beyond delay efficiency, aside from its fairness properties. From a queueing perspective, one would expect some degree of inherent robustness in the solution, particularly in the presence of operational uncertainty.

In practice, regulations are not applied once in isolation. As conditions evolve due to airline adjustments, weather developments, or other factors, CASA is rerun periodically (typically every 15 minutes), producing updated delay assignments before the lock-in horizon. In flight-centric optimization, there is generally no guarantee that successive solutions remain close to one another unless explicitly constrained through an additional term in the objective function. By contrast, under FCFS-based mechanisms, FPFS constraints preserve ordering, limiting arbitrary reassignments. This suggests that FPFS may introduce a stabilizing effect, reducing ``plan thrashing,'' which remains a key barrier to operational deployment. We call this feature \emph{Self-Regulating Regulations,} as we predict that the optimized regulations could be minimally updated or tuned without costly rerun from scratch, freeing operators for other work like coordination. Future work will investigate this hypothesis by incorporating scenario uncertainty and evaluating RZ under dynamic conditions.

Finally, we recognize that a complete evaluation of RZ on the authoritative R-NEST system is necessary to ensure operational compatibility. 

\section{Conclusion}
This paper introduced Regulation Zero 2, a flow-centric sequential planning framework that acts as a meta-planner for existing FPFS slot allocation machinery rather than replacing it. By operating natively in the regulation space and planning regulations as an ordered sequence, the framework directly targets Regulation Cascading---the nonlinear interactions between regulations that undermine predictability in current operations.

Three findings stand out from our evaluation on seven pan-European summer-peak days. First, RZ v2 achieves state-of-the-art DCB performance, delivering 30–40\% higher objective improvements than Simulated Annealing and substantially outperforming NSGA-II and greedy capping baselines, while concentrating impact on roughly 25\% fewer flights. Second, the ablation of the MCTS lookahead component degraded performance by approximately 50\%, providing direct quantitative evidence that Regulation Cascading is a first-order concern and that sequential planning is essential rather than optional. Third, the new NomRel and InLoad heuristics enabled a reduced-effort configuration to match full-effort v1 performance, suggesting that identifying high-leverage flows matters more than exhaustively searching over reference locations.

Future work will focus on validating the Self-Regulating Regulations hypothesis under demand and capacity uncertainty, exploring hybrid strategies that combine RZ's exceedance-reduction strength with the delay efficiency of flight-centric methods, and confirming operational compatibility through evaluation on the authoritative R-NEST system.

\appendix
\section{Regulation Zero as Regularized Policy Optimization}\label{sec:poliopt}

This appendix demonstrates that the formulation of Regulation Zero introduced in Section~\ref{sec:method} can be interpreted as a form of \emph{regularized policy optimization} that jointly employs forward and reverse Kullback-Leibler (KL) divergences to shape both hotspot and regulation proposal selections. For tractability, we focus on the simplified setting where \emph{only one hotspot/expansion attempt is associated with each node in the search tree}, deferring extensions for future research.

\subsection{Regulation Zero attempts to track a policy regularized by a forward KL for hotspot selection and reverse KL for proposal selection}
If we define the factorized policy $\pi(h,r|s)=\pi_H(h|s)\pi(r|h,s)$ over the two priors, and the corresponding regularization term as:
\begin{equation}
	\Omega_s(\pi) = \tau_h \kl{\pi_H(\cdot | s)}{P_H(\cdot | s)}+ \expect{h \sim \pi_H}{ \lambda_R (s,h) \kl{P_R(\cdot|h,s)}{\pi_R(\cdot|h,s))}}
\end{equation}

Define a general vector of state backup values $q = q(s,h,r)$, we define the \textit{regularized objective} as:
\begin{equation}\label{eq:gs_problem}
	G_s(q) = \sup_\pi [\expect{h,r \sim \pi}{q(s,h,r) - \Omega_s(\pi)}]
\end{equation}

The following Proposition shows that the solution to $G_s(q)$ has a decomposition form:

\begin{proposition}\label{prop:pi_ref}
	For each \(h\), let \(q_h(\cdot) = q(s,h,\cdot)\) and \(\lambda_h = \lambda_R(s,h)\). The inner maximizer \(\bar{\pi}_R(\cdot|h,s)\) and its attained value \(U_h(s)\) solve
	\[
	\bar{\pi}_R(r|h,s) = \frac{\lambda_h \, P_R(r|h,s)}{\alpha_h - q_h(r)},
	\quad
	\text{with } \alpha_h > \max_r q_h(r) \text{ chosen s.t. } \sum_r \bar{\pi}_R(r|h,s) = 1,
	\]
	\[
	U_h(s) = \sup_{y} \left[ q_h \cdot y - \lambda_h \, \kl{P_R}{y} \right] = \alpha_h - \lambda_h - \lambda_h \sum_r P_R(r|h,s) \log\!\left(\frac{\alpha_h - q_h(r)}{\lambda_h}\right).
	\]
	With inclusive inner values \(U_h(s)\),
	\[
	\bar{\pi}_H(h|s) \propto P_H(h|s) \exp\!\left(U_h(s) / \tau_h\right),
	\quad
	\text{and}
	\quad
	G_s(q) = \tau_h \log \sum_h P_H(h|s) \exp\!\left(U_h(s) / \tau_h\right).
	\]
	Thus \(\bar{\pi}\) is the nested solution
	\[
	\bar{\pi}_H(h|s) \propto P_H(h|s) \exp\!\left(U_h / \tau_h\right),
	\qquad
	\bar{\pi}_R(\cdot|h,s) \text{ as above.}
	\]
\end{proposition}

\begin{proof}
	We begin by expanding $G_s(q)$:
	
	\begin{align*}
		G_s(q) &= \sup_\pi [\expect{h,r \sim \pi}{q(s,h,r)] - \Omega_s(\pi)} = \sup_\pi [\sum_{h, r} \pi_H(h|s)\pi_R(r|h,s) q(s,h,r) - \Omega_s(\pi)]\\
		&= \sup_\pi [ \sum_{h, r} \pi_H(h|s) \pi_R(r|h,s) q(s,h,r) - \tau_h \kl{\pi_H(\cdot | s)}{P_H(\cdot | s)} \\
		&- \sum_h \pi_H(h|s) \lambda_R (s,h) \kl{P_R(\cdot|h,s)}{\pi_R(\cdot|h,s))}] \\
		&= \sup_\pi \bigg[ \sum_h \pi_H(h|s) [\sum_r \pi_R(r|h,s) q(s,h,r) - \lambda_R (s,h) \kl{P_R(\cdot|h,s)}{\pi_R(\cdot|h,s))}] \\
		&- \tau_h \kl{\pi_H(\cdot | s)}{P_H(\cdot | s)} \bigg] \\
		&= \sup_{\pi_H} \bigg[ \sum_h \pi_H(h|s) \sup_{\pi_R} [\sum_r \pi_R(r|h,s) q(s,h,r) - \lambda_R (s,h) \kl{P_R(\cdot|h,s)}{\pi_R(\cdot|h,s))}] \\
		&- \tau_h \kl{\pi_H(\cdot | s)}{P_H(\cdot | s)} \bigg]
	\end{align*}
	
	We define the inner problem:
	$U_h(q_h) = \sup_{y} [q_h^\intercal y - \lambda_R (s,h) \kl{P_R(\cdot|h,s)}{y}]$. Given $\lambda_R (s,h) > 0$, the problem is convex in y, and given the constraint for $y$ to be a valid probability distribution, we could define the Lagrangian:
	\begin{equation}
		\mathcal{L} = q_h^\intercal y - \lambda_R \kl{P_R}{y} - \alpha (y_r - 1)
	\end{equation}
	
	The stationary condition gives:
	\begin{equation}\label{eq:inner_loop_sol}
		y^* = \frac{\lambda P_R}{\alpha - q_h}
	\end{equation}
	and $\alpha$ must solve for:
	\begin{equation}
		\sum_r \frac{\lambda P_R(r)}{\alpha - q_h(r)} = 1
	\end{equation}
	in order to normalize the distribution $y = \pi_R$. The attained value corresponding to $y^*$ is:
	\[
	U_h(q_h) = q_h^\intercal y^* - \lambda_R \sum_r P_R(r) \log \frac{P_R(r)}{y^*(r)} = \alpha - \lambda_R - \lambda_R \sum_r P_R(r) \log \bigg(\frac{\alpha - q_r}{\lambda}\bigg)
	\]
	
	We notice that this expression has the same form as in \cite{grill2020monte}. Now, given that the inner problem attains its maximum at \( U_h(q_h) \) as derived above, the outer optimization problem can be expressed in the following closed form:
	
	\begin{equation}
		\sup_{x \in \Delta(\mathcal{H})} \; x^\top U(q_x) - \tau_h \, \mathrm{KL}(x \, \| \, P_H),
	\end{equation}
	
	where \( \Delta(\mathcal{H}) \) represents the probability simplex over \(\mathcal{H}\).
	
	Because this corresponds to the classical maximum-entropy regularized policy optimization formulation \cite{ziebart2010modeling}, the optimal solution admits the familiar log-sum-exp (softmax) form:
	
	\begin{equation}
		x_h^* \propto P_H(h) \, \exp\!\left(\frac{U_h}{\tau_h}\right),
		\qquad
		x^*(h) = 
		\frac{P_H(h) \, \exp\!\left(U_h / \tau_h\right)}
		{\sum_{h' \in \mathcal{H}} P_H(h') \, \exp\!\left(U_{h'} / \tau_h\right)}
	\end{equation}
\end{proof}

Proposition \ref{prop:pi_ref} says that the factored policy $\bar{\pi} = \bar{\pi}_R\bar{\pi}_H$ solves the regularized objective maximizing problem $G_s(q)$. The PUCT selection rule (\ref{eq:puct_selection}) was shown in \cite{grill2020monte} to be equal to:
\begin{equation}
	a^* = \arg \max_a \frac{\partial}{\partial n(a)} [q^\intercal y - \lambda_N \kl{\pi_\theta}{\hat{\pi}}]
\end{equation}
which can be viewed as a gradient descent step towards the optimizer of the regularized problem:
\begin{equation}
	\bar{\pi} = q^\intercal y - \lambda_N \kl{\pi_\theta}{\hat{\pi}}
\end{equation}
which in turn has the same solution form $y = \lambda_N \pi_\theta / (\alpha - q)$ as (\ref{eq:inner_loop_sol}). In other words, the inner loop's MCTS attempts to track the policy:
\begin{equation}
	\sup_{y} [q_h^\intercal y - \lambda_R (s,h) \kl{P_R(\cdot|h,s)}{y}]
\end{equation}

Combined with the outer loop's Boltzmann distribution's form, we can conclude that jointly, the Regulation Zero framework tracks the solution for the problem $G_s $(\ref{eq:gs_problem}).

\begin{equation*}
	\boxed{\hat{\pi} \approx \bar{\pi} = \arg\max_\pi [\expect{h,r \sim \pi}{q(s,h,r) - \Omega_s(\pi)}]},
\end{equation*}
if the outer loop's hotspot severity proxy function is assumed to track
\begin{equation}
	\phi(s) \approx U(s) = \sup_{y} \left[ q_h \cdot y - \lambda_h \, \kl{P_R}{y} \right].
\end{equation}

\subsection{Sample Complexity Estimation}
Let \(N\) be the total number of root simulations and let \(n_H(h)\) denote the number of times the hotspot
\(h\) is selected at the root, so that \(\sum_{h} n_H(h)=N\).  
We write \(\hat\pi_H,\hat\pi_R\) for Regulation Zero respectively.

For any pair \((h,r)\) we have
\begin{align*}
	\bigl|\hat\pi(h,r)-\bar\pi(h,r)\bigr|
	&=\bigl|\hat\pi_H(h)\,\hat\pi_R(r\mid h)-\bar\pi_H(h)\,\bar\pi_R(r\mid h)\bigr|\\
	&\le \bigl|\hat\pi_H(h)-\bar\pi_H(h)\bigr|\;\hat\pi_R(r\mid h)
	+\bar\pi_H(h)\,\bigl|\hat\pi_R(r\mid h)-\bar\pi_R(r\mid h)\bigr|\\
	&\le \bigl|\hat\pi_H(h)-\bar\pi_H(h)\bigr|
	+\bigl|\hat\pi_R(r\mid h)-\bar\pi_R(r\mid h)\bigr|.
\end{align*}
Taking the supremum over all \((h,r)\) yields the key reduction
\begin{equation}
	\|\hat\pi-\bar\pi\|_{\infty}
	\;\le\; \|\hat\pi_H-\bar\pi_H\|_{\infty}
	+ \max_{h}\,\|\hat\pi_R(\cdot\mid h)-\bar\pi_R(\cdot\mid h)\|_{\infty}.
	\label{eq:joint-reduction}
\end{equation}

For each hotspot \(h\), assuming that \(\bar\pi_R(\cdot\mid h)\) is locally constant as counts grow,
the reversed-KL tracking bound gives
\begin{equation}
	\|\hat\pi_R(\cdot\mid h)-\bar\pi_R(\cdot\mid h)\|_{\infty}
	\;\le\; \frac{|R|-1}{\,|R|+n_H(h)}.
	\label{eq:inner-bound}
\end{equation}

If the hotspot at each root selection is sampled from \(\bar\pi_H\),
Hoeffding’s inequality together with a union bound implies that, with probability at least
\(1-\delta\),
\begin{equation}
	\|\hat\pi_H-\bar\pi_H\|_{\infty}
	\;\le\; \sqrt{\frac{1}{2N}\,\log\!\frac{2|H|}{\delta}}.
	\label{eq:outer-bound}
\end{equation}

Hence, for any run with total simulations \(N\) and realised gate counts \(n_H(h)\), we obtain,
with probability at least \(1-\delta\),
\begin{equation}
	\boxed{\|\hat\pi-\bar\pi\|_{\infty}
		\;\le\;
		\sqrt{\frac{1}{2N}\,\log\!\frac{2|H|}{\delta}}
		\;+\;
		\max_{h}\frac{|R|-1}{\,|R|+n_H(h)}}
	\label{eq:anytime-bound}
\end{equation}

\section*{Acknowledgments}
This work received funding from the SESAR DeepFlow Project [Grant Agreement ...]. The contents reflect only the authors' views. We are grateful for the discussions with EUROCONTROL Network Manager controllers.

\bibliographystyle{ieeetr}
\bibliography{main}

\end{document}